\documentclass[12pt,a4paper]{amsart}
\usepackage{amsmath}
\usepackage{extarrows}
\usepackage{amssymb,amsthm,amscd}
\usepackage{fullpage}

\usepackage{etex}

\usepackage[arrow, matrix, curve]{xy}
\usepackage{MnSymbol}

\usepackage{hyperref}

\input{diagrams.tex}
\newarrow {Dashto} {}{dash}{}{dash}>

\usepackage{graphicx}
\usepackage{tikz}
\usepackage{tikz-cd}
\usetikzlibrary{matrix,arrows,positioning,calc,chains}

\usepackage[lmargin=3cm,rmargin=3cm,tmargin=2cm,bmargin=2cm, marginpar=2cm ]{geometry}
\usepackage[draft,multiuser,english]{fixme}

\fxsetup{layout=inline}
\FXRegisterAuthor{p}{anp}{AP}

\FXRegisterAuthor{z}{anz}{MZ}

\usepackage{enumerate}
\usepackage[normalem]{ulem} 

\theoremstyle{plain}
\newtheorem{theorem}{Theorem}[section]
\newtheorem{lemma}[theorem]{Lemma}

\newtheorem{serreth}[theorem]{Serre's Theorem}

\newtheorem{proposition}[theorem]{Proposition}

\newtheorem{corollary}[theorem]{Corollary}
\theoremstyle{definition}
\newtheorem{definition}[theorem]{Definition}

\newtheorem{example}[theorem]{Example}
\newtheorem{examples}[theorem]{Examples}
\newtheorem{notation}[theorem]{Notation}
\newtheorem{remark}[theorem]{Remark}
\newtheorem{remarks}[theorem]{Remarks}
\newtheorem{sit}[theorem]{}

\def\Aut{\operatorname{Aut}}

\def\SAut{\operatorname{SAut}}

\def\Pic{\operatorname{Pic}}

\def\Quot{\operatorname{Frac}}

\def\ML{\operatorname{ML}}

\def\supp{\operatorname{supp}}
\def\Lie{\operatorname{Lie}}
\def\SL{\operatorname{SL}}
\def\Spec{\operatorname{Spec}}
\def\Spf{\operatorname{Spf}}
\def\spec{\operatorname{Spec}}
\def\Stab{\operatorname{Stab}}

\def\codim{\operatorname{codim}}
\def\rk{\operatorname{rk}}

\def\id{{\mathrm{id}}}
\def\PP{{\mathbb P}}
\def\O{{\mathcal O}}
\def\Q{{\mathcal H}}
\def\cT{{\mathcal T}}
\def\ZZ{{\mathbb Z}}
\def\CC{{\mathbb C}}

\def\QQ{{\mathbb Q}}
\def\NN{{\mathbb N}}
\def\N{{\mathbb N}}
\def\G{{\mathbb G}}
\def\m{{\mathfrak m}}
\def\A{{\mathbb A}}
\def\AA{{\mathbb A}}
\def\UU{{\mathbb U}}

\def\K{{\mathbb K}}

\def\TTT{{\mathbb T}}

\def\yE{y_{\scriptscriptstyle E}}
\def\embed{\hookrightarrow}

\def\Aff{\operatorname{Aff}}

\def\reg{{\rm reg}}

\def\sing{{\rm sing}}

\def\GL{\operatorname{GL}}

\def\SL{\operatorname{SL}}
\def\PSL{\operatorname{PSL}}

\def\Arc{\operatorname{Arc}}
\def\Pui{\operatorname{Pui}}
\def\T{\mathrm{T}}

\def\fps{\mathrm{fps}}
\def\ML{\mathrm{ML}}

\def\cI{\mathcal{I}}

\def\Tt{{\mathcal T}}

\def\cO{{\mathcal O}}
\def\0{\circ}

\def\bangle#1{\langle #1 \rangle}

\def\amalgGT{\varinjlim(T,\mathcal{G})}

\theoremstyle{definition}

\theoremstyle{plain}

\renewcommand{\a}{\mathbb A}

\newcommand{\ka}{\mathbb K}

\newcommand{\p}{\mathbb P}

\newcommand{\F}{\mathcal F}
\DeclareMathOperator{\Sing}{Sing}
\DeclareMathOperator{\mult}{mult}

\DeclareMathOperator{\ord}{ord}

\DeclareRobustCommand{\ie}{i.\,e.~}

\newcommand{\nlin}{\unitlength1mm\begin{picture}(0,9.25)
                       \put(0,0.75){\line(0,1){8.5}}
                      \end{picture}}

\newcommand{\vlin}[1]{\hspace{0.75mm}\unitlength1mm\begin{picture}(#1,0)
                       \put(0,0){\line(1,0){#1}}
                      \end{picture}\hspace{0.75mm}\rule[-3mm]{0mm}{4mm}}

\newcommand{\lin}{\vlin{8.5}}
\newcommand{\llin}{\vlin{11.5}}

\newcommand{\co}[1]{\unitlength1mm\begin{picture}(0,8)
    \put(0,0){\circle{1.5}}
    \put(0,3){\makebox(0,5)[b]{$#1$}}
                      \end{picture}}

\newcommand{\mybox}{\unitlength1mm\begin{picture}(0,1.5)
    \put(-0.75,-0.75){\line(0,1){1.5}}
    \put(-0.75,-0.75){\line(1,0){1.5}}
    \put(0.75,0.75){\line(0,-1){1.5}}
    \put(0.75,0.75){\line(-1,0){1.5}}
    \end{picture}}

\newcommand{\xbox}{\unitlength1mm\begin{picture}(0,1.5)
    \put(0,0){$\mybox$}
    \put(-0.75,0){\line(1,0){1.5}}
    \put(0,-0.75){\line(0,1){1.5}}
    \end{picture}}

\newcommand{\xboxo}[1]{\unitlength1mm\begin{picture}(0,8)
    \put(0,0){\xbox}
    \put(0,3){\makebox(0,5)[b]{$#1$}}
                      \end{picture}}

\newcommand{\cu}[1]{\unitlength1mm\begin{picture}(0,8)
    \put(0,0){\circle{1.5}}
    \put(0,-7){\makebox(0,4)[b]{$#1$}}
    \end{picture}
      \rule[-7mm]{0mm}{7mm}}

\newcommand{\cou}[2]{\unitlength1mm\begin{picture}(0,8)
    \put(0,0){\circle{1.5}}
    \put(0,3){\makebox(0,5)[b]{$#1$}}
    \put(0,-7){\makebox(0,4)[t]{$#2$}}
      \end{picture}
      \rule[-7mm]{0mm}{7mm}}

\newcommand{\cshiftup}[2]{\unitlength1mm\begin{picture}(0,9.25)
                       \put(0,10){\cou{#1}{#2}}
                      \end{picture}}

\newcommand{\xbou}[2]{\unitlength1mm\begin{picture}(0,8)
    \put(0,0){\xbox}
    \put(0,2){\makebox(0,5)[b]{$#1$}}
   \put(0,-7){\makebox(0,5)[b]{$#2$}}
      \end{picture}
      \rule[-7mm]{0mm}{7mm}}

\newcommand{\xbshiftup}[2]{\unitlength1mm\begin{picture}(0,9.25)
                       \put(0,10){\xbou{#1}{#2}}
                      \end{picture}}

\begin{document}

\title[Automorphism groups of affine surfaces]{On
automorphism groups of affine surfaces}

\author{S.~Kovalenko, A.~Perepechko,
M.~Zaidenberg}

\address{Mathematisches Institut, Albert-Ludwigs-Universit\"at Freiburg, Albertstra{\ss}e 19, 79104 Freiburg, Germany}
\email{sergei.kovalenko@rub.de}

\address{Kharkevich Institute for Information Transmission Problems, 19 Bolshoy Karetny per., 127994 Moscow, Russia}
\email{perepeal@gmail.com}

\address{Universit\'e Grenoble Alpes,
Institut Fourier,
CS 40700,
38058 Grenoble cedex 09, France}
\email{Mikhail.Zaidenberg@ujf-grenoble.fr}

\thanks{{\bf Acknowledgements:} The second author was supported by the SPbGU grant for post-doctoral students no. 6.50.22.2014 and partially granted by the ``Dynasty'' foundation.
He also expresses his gratitude to the Institut Fourier, Grenoble. The third author is grateful to the INdAM (Istituto Nazionale di Alta Matematica "F. Severi") in Rome, 
where a part of the work was done. The second and the third authors are grateful to the Max-Planck Institute for Mathematics, Bonn, for hospitality and generous support.}

\begin{abstract} This is a survey on
the automorphism groups in various classes of affine algebraic surfaces and the algebraic group actions on such surfaces. Being infinite-dimensional,
these  automorphism groups share some important features of algebraic groups. At the same time, they can be studied from the viewpoint of the combinatorial group theory, 
so we put a special accent on group-theoretical aspects (ind-groups,  amalgams, etc.).\ We provide different approaches to classification, prove certain new results, 
and attract attention to several open problems.  
\end{abstract}
\maketitle

\tableofcontents

\section{Introduction}

Our aim is to give a comprehensive survey of the automorphism groups of affine algebraic surfaces and algebraic group actions on such surfaces.
 We are using several different classifications of surfaces, according to the Makar-Limanov invariant, to the rank of the automorphism group, etc. However, our ultimate goal is to approach a reasonable classification
of the automorphism groups themselves. These groups are often infinite-dimensional, and even
`wild' in a sense, so, we do not reach the final goal.  Nevertheless, we are trying to put
some order in our present knowledge on the subject, and to indicate difficult open problems.
 The authors  apologize for incompleteness of the reference list and of the overview of the cited literature. For instance, we do not touch upon
the recent progress in studies of complete algebraic vector fields on affine surfaces,  the Lie algebras of algebraic vector fields, the related groups of (biholomorphic) automorphisms, etc.,
see, e.g., \cite{AKL, AKP, Do, KKL, Kr, KR, KL, Le} and the references therein. 

We  provide also several new results, especially concerning the automorphism groups of surfaces of classes $(\ML_1)$ and $(\ML_2)$, 
where one possesses by now a rather complete knowledge.  By contrast, we are far from a good understanding of the automorphism groups of surfaces of class $(\ML_0)$, 
that is, of Gizatullin surfaces.

Before passing to the content of the paper, we recall some general notions and facts.

\subsection{Classification 
according to the Makar-Limanov invariant}
\label{sss: ML-classification}

\begin{sit}\label{sit: ML-inv} Let $X$ be a normal affine variety over an algebraically closed field $\K$ of characteristic zero. 
The {\em special automorphism group}
$\SAut X$ is  the subgroup of $\Aut X$ generated by all its
one-parameter unipotent subgroups (\cite{AFKKZ}). This group is trivial if and only if
$X$ does not admit any nontrivial $\G_{\rm a}$-action.
The {\em Makar-Limanov invariant} $\ML(X)=\cO_X(X)^{\SAut X}$ is the subalgebra of invariants of $\SAut X$.
The normal affine varieties can be classified
according to the transcendence degree of the Makar-Limanov invariant or, which is the same, according to the {\em Makar-Limanov complexity} of $X$,
that is, the codimension of a general $\SAut X$-orbit. Recall (\cite{AFKKZ}) that these orbits are locally closed subvarieties in $X$. 
One says that a variety $X$ is \emph{of class} (ML$_i$), $i\in\{0,1,\ldots,\dim X\}$, if its Makar-Limanov complexity equals $i$. 

We restrict below to the case $\dim X=2$. A normal affine surface $X$ is of class
\begin{itemize}
\item (ML$_2$) if $X$ does not admit any $\A^1$-fibration
over an affine curve, see \cite[Rem. 1.7]{FKZ-deformations} or \cite[Lem. 1.6]{FZ-LND};
\item  (ML$_1$) if $X$ admits a unique such fibration;
\item (ML$_0$) if $X$ admits at least two distinct such fibrations.
\end{itemize}
The surfaces of class  (ML$_0$) are also called {\em Gizatullin surfaces}, cf.\ Definition \ref{def: giz}. 
The surfaces of classes $(\ML_0)$ and $(\ML_1)$ are {\em $\G_a$-surfaces}, which means that they admit an effective algebraic action of 
the additive group $\G_a$ of the field $\K$, while the $\ML_2$-surfaces do not admit such an action.
In this article we consider the additive (resp. multiplicative) group $\G_a$ (resp. $\G_m$) of the field as an algebraic group over $\K$.
\end{sit}

\begin{sit}\label{sit: LND} A derivation $\partial$ of a ring $A$ is called {\em locally nilpotent} if for any $a\in A$, $\partial^n a=0$ for some $n=n(a)\in\NN$.
For  a normal affine surface $X$, the coordinate ring $\mathcal{O}_X(X)$
\begin{itemize}\item does not admit any nonzero locally nilpotent derivation if $X\in ({\rm ML}_2)$;
\item  admits a unique such derivation up to a factor, which is a rational function on $X$, if
$X\in ({\rm ML}_1)$;  \item admits two non-proportional such derivations if $X\in ({\rm ML}_0)$.
\end{itemize}
\end{sit}

\begin{sit}\label{rem: NC versus SNC} There is a combinatorial counterpart of the ML-classification, see  Lemma \ref{lem: admissible-chains}.
Let $X$ be a normal affine surface, and let $V$ be a completion of $X$ with boundary divisor $D=V\setminus X$. 
Assume that $(V,D)$ is a minimal NC-completion, that is, $V$ is smooth near $D$ and $D$ is a normal crossing divisor such that 
no (smooth, rational) $(-1)$-component of  $D$ can be contracted without losing the NC-property. Let $\Gamma_D$ be the weighted dual graph of $D$. 
If $X$ admits an $\A^1$-fibration, then
 $(V,D)$ is in fact an SNC-completion, and  $\Gamma_D$ is a tree.

A vertex of $\Gamma_D$ is called a \emph{rupture vertex} if either it is of degree at least three, or the corresponding component of $D$ is irrational. 
The complement in $\Gamma_D$ to all rupture vertices consists of connected components called \emph{segments}.
A graph without rupture vertices consisting of a single linear segment is called a \emph{chain}.

A linear weighted graph is called {\em admissible} if all its weights are $\le -2$. 
Via birational transformations of $(X,D)$, any non-admissible linear segment of $\Gamma_D$ can be transformed into a segment with an end vertex of weight $0$.
\end{sit}

\begin{lemma}[{\cite[Rem.\ 1.7]{FKZ-deformations}}]\label{lem: admissible-chains}
The graph $\Gamma_D$ as in {\rm \ref{rem: NC versus SNC} }
\begin{itemize}\item has only admissible extremal
linear segments for $X\in ({\rm ML}_2)$;  
\item  is non-linear and has a non-admissible
extremal linear segment for $X\in ({\rm ML}_1)$ different from $\A^1\times\A^1_*$ 
\footnote{The surface $X=\A^1\times\A^1_*$ of class (ML$_1$) admits an SNC-completion $(X,D)$
with $\Gamma_D$ being the linear chain with weights $[[0,0,0]]$.};
\item is a chain (a {\em zigzag}) non-transformable into the linear chain with a sequence of weights
$[[0,0,0]]$ for $X\in ({\rm ML}_0)$.  \end{itemize}
\end{lemma}

\subsection{Classification according to rank}\label{ss: classif-rank}
The {\em rank} of an ind-group $G$ acting morphically and effectively on a variety $X$ is the maximal dimension of an algebraic torus contained in $G$. 
This rank does not exceed the dimension of $X$, and $X$ is toric in the case of equality. A surface $X$ with ${\rm rk}\Aut X\ge 1$ is called a {\em $\G_m$-surface}. 
The rank distinguishes toric surfaces, non-toric $\G_m$-surfaces, and surfaces without $\G_m$-actions; indeed, their ranks are $2,1$, and $0$, respectively.

\subsection{A general classification scheme}\label{ss: general scheme}
The two independent classifications of normal affine surfaces, according to the rank of the automorphism group and according to the Makar--Limanov complexity as defined before,
give altogether 9 classes of  affine surfaces denoted $(\ML_i,r)$, $(i,r)\in\{0,1,2\}^2$, where $r$ is the rank of $\Aut X$ and $i$ the Makar-Limanov complexity of $X$.

To describe the automorphism groups of affine surfaces and the algebraic group actions on them, one applies various means. Some of them found their place in our survey. The content of the present notes is as follows.
\begin{itemize}
\item In Section \ref{sec: generalities} we introduce different classes of groups:
ind-groups, nested ind-groups, amalgams, bearable groups.
\item In Section \ref{sec: first examples} we study classical examples of affine surfaces, including the toric surfaces, 
along with a presentation of their automorphism groups as amalgams.
\item Section \ref{sec: prelim} contains generalities on algebraic group actions on affine surfaces; 
\item in Subsection \ref{ss: DPD} we classify affine surfaces of rank $\ge 1$, along with one-parameter groups acting on such surfaces,  in terms of the DPD presentation.
\item In Section \ref{sec: Giz} we consider the automorphism groups of the $\ML_0$-surfaces, also called Gizatullin surfaces. 
We provide a classification of their one-parameter subgroups. For some particular classes of Gizatullin surfaces, we describe their automorphism groups as amalgams. 
The structure of the automorphism groups of general Gizatullin surfaces remains mysterious.
\item In Sections \ref{sec: dJ} we study the automorphism groups of $\A^1$-fibrations as nested ind-groups.
\item Sections \ref{sec:fiber-arc} -- \ref {sec: dJ-surf} deal with $\A^1$-fibrations on affine surfaces. 
In Section \ref{sec:fiber-arc} we introduce Puiseux arc spaces and study the actions of automorphisms on these spaces. 
\item In Section \ref{sec: dJ-surf} this techniques is applied in order to describe the automorphism groups of $\A^1$-fibrations on affine surfaces in more detail. 
The neutral component  of such an automorphism group occurs to be a metabelian nested ind-group of rank at most two, 
while the component group\footnote{That is, the group formed by the connected components.} is at most countable,  see Theorems~\ref{th:jonq-mu}, \ref{th:jonq-mu-general}, 
and Corollary \ref{cor: nested}.
\end{itemize}

\par\medskip\noindent
\textbf{Acknowledgment.}
The authors  are grateful to the referees for a thorough reading and numerous 
useful remarks that allowed to improve essentially the presentation  and to remove several inaccuracies that had slipped into the first draft of the paper.

\section{Ind-groups, amalgams, and all this}
\label{sec: generalities}
\subsection{Ind-groups}\label{ss: ind-grps}
Recall (cf.\ \cite{Ka3, Kr, Ku, Sh1, Sh2, Sta}) that an {\em ind-variety} is a union of an ascending sequence of 
algebraic varieties $X_i$ with closed embeddings $X_i\subset X_{i+1}$. 
An \emph{algebraic subvariety} of such an ind-variety $X$ is a subvariety of some $X_i$. An {\em ind-group} $G$
is an inductive limit
$G=\varinjlim \Sigma_i$ of an increasing sequence of algebraic varieties
$$\Sigma_1\subset \Sigma_2\subset\ldots\subset \Sigma_n\subset\ldots$$ with closed
embeddings
$\Sigma_i\subset \Sigma_{i+1}$, where $G$ is endowed with a group structure such
that for each pair
$(i,j)\in\NN^2$
the multiplication $(f,g)\mapsto f\cdot g^{-1}$
yields a morphism $\Sigma_i\times \Sigma_j\to \Sigma_{n(i,j)}$ 
for some $n(i,j)\in\NN$.
If all the $\Sigma_i$ are affine algebraic varieties, 
then $G=\varinjlim \Sigma_i$ is called an \emph{affine ind-group}.
 In particular, an (affine) ind-group is an (affine) ind-variety. 

The {\em neutral component} $G^\circ$ of an ind-group $G=\varinjlim \Sigma_i$ is defined as the
inductive limit of the connected components of $\Sigma_i$ passing through the neutral element $e\in G$.

A {\em morphism of ind-groups} $G=\varinjlim \Sigma_i$ and $G'=\varinjlim \Sigma'_i$ is a group homomorphism $\phi\colon G\to G'$ such that for any $i\ge 1$, 
$\phi(\Sigma_i)\subset \Sigma'_j$ for some $j=j(i)\ge 1$, and $\phi|_{\Sigma_i}\colon \Sigma_i\to \Sigma'_j$ is a morphism of varieties. 
Clearly, $\phi(G^\circ)\subset {G'}^\circ$.

 Two ind-group structures on the same abstract group $G$ are  \emph{equivalent} if the identity map yields an isomorphism of the corresponding ind-groups. 

A subgroup $H\subset G$ of an ind-group $G=\varinjlim \Sigma_i$ is {\em closed} if for any $i\ge 1$, the intersection $H\cap\Sigma_i$ is closed in $\Sigma_i$. 
In the latter case $H=\varinjlim (H\cap\Sigma_i)$ is an ind-group.

One says that an ind-group $G=\varinjlim \Sigma_i$ {\em acts morphically} on a variety $X$ if there is an action $G\times X\to X$ of $G$ on $X$ such that 
for each $i\in\NN$ the restriction $\Sigma_i\times X\to X$ is a morphism of algebraic varieties.  

The following proposition is well known;
see, e.g., \cite[Lem.\ 2.2]{Kraft-notes}\footnote{When the preprint of this paper was finished, Hanspeter Kraft acknowledged the third author  that Proposition \ref{prop: ind-str} and some other results in Section 2 will appear in a more general form in a forthcoming paper \cite{FK}, which is an extended version of \cite{Kraft-notes}. We thank Hanspeter Kraft for this information and for sending a preliminary version of \cite{FK}.} and \cite[Prop.\ 2.5]{Kr}. The first (unpublished) proof of (a) is due to Bialynicki-Birula; cf.\ also \cite[Rem.\ after Cor.\ 1.2]{Ka1}. 

\begin{proposition}\label{prop: ind-str} Let $X$ be an affine algebraic variety, and let $I\subset  \cO_X(X)$ be a proper ideal. The the following hold.
\begin{itemize}\item[(a)]
The automorphism group $\Aut X$ possesses a structure
of an affine ind-group acting morphically on $X$. 

\item[(b)] Let
$\Aut (X,I)$ be the subgroup of $\Aut X$ of all automorphisms of $X$
leaving $I$ invariant.
Then $\Aut (X,I)$ is a closed ind-subgroup in $\Aut X$.
%
\end{itemize}
\end{proposition}

\begin{proof} (a) Let $A=\cO_X(X)$. Fixing a closed embedding $X\hookrightarrow\A^n$, consider the following objects:
$$A_d=\{p|_X\,|\,p\in\K[x_1,\ldots,x_n],\,\,\deg p\le d\}\subset A\,,$$
$$V_d=A_d^n\subset {\rm Mor}\,(X,\A^n)\,,$$
$$W_d=\{\varphi\in V_d\,|\,\varphi(X)\subset X\}\subset {\rm Mor}\,(X,X)\,.$$
Clearly, $V_d$ is a finite-dimensional subspace of the $\K$-vector space ${\rm Mor}\,(X,\A^n)$, and $W_d$ is a closed (affine) algebraic subvariety in $V_d$. 
The map $$\Phi_{d,d'}\colon W_d\times W_{d'}\to W_{dd'},\quad (\varphi,\psi)\mapsto \psi\circ\varphi\,,$$
is a morphism of algebraic varieties. Hence 
$$\widetilde\Sigma_d:=\Phi_{d,d}^{-1}(\id_X)\subset W_d\times W_{d}$$ is a closed algebraic subvariety in $W_d\times W_{d}$ for any $d\ge 1$. 
Consider the natural embeddings $\Aut X\subset {\rm Mor}\,(X,X)\subset {\rm Mor}\,(X,\A^n)$. We have
$$\widetilde\Sigma_d=\{(\varphi,\varphi^{-1})\,|\,\varphi, \varphi^{-1}\in W_d\cap\Aut X\}\,.$$  
Let $\Sigma_d={\rm pr}_1(\widetilde\Sigma_d)\subset W_d\cap\Aut X$. The morphism 
$${\rm pr}_1\colon \widetilde\Sigma_d\to\Sigma_d,\quad (\varphi,\varphi^{-1})\mapsto \varphi\,,$$
is one-to-one. This allows to introduce a structure of an affine algebraic variety on 
$\Sigma_d$ borrowed from the one on $\widetilde\Sigma_d$,  so that $\Sigma_d\cong\widetilde\Sigma_d$.

\smallskip

\noindent {\bf Claim.} \emph{With this algebraic structure on $\Sigma_d$, the following hold. 
\begin{itemize}
\item[(i)] $\Aut X=\bigcup_{d=1}^\infty\Sigma_d$;
\item[(ii)] $ \Sigma_d\subset \Sigma_{d'}$ is a closed embedding  for any $d\le d'$;
\item[(iii)] $\Sigma_d\to\Sigma_d,\quad\varphi\mapsto\varphi^{-1}\,,$  is a morphism; 
\item[(iv)] $\Phi_{d,d'}|_{\Sigma_d\times \Sigma_{d'}}\colon \Sigma_d\times \Sigma_{d'}\to \Sigma_{dd'}$ is a morphism; 
\item[(v)] $\Sigma_d\times X\to X\subset\A^n$, $(\varphi, x)\mapsto \varphi(x)$, is a morphism.
\end{itemize}
Consequently,  $\Aut X=\varinjlim\Sigma_d$ is an ind-group acting morphically on $X$. }

\smallskip 

\noindent \emph{Proof of the claim.} Statement (i) is immediate. 

(ii) follows from the fact that $\widetilde\Sigma_d=\widetilde\Sigma_{d'}\cap (W_d\times W_d)$ is closed in $W_{d'}\times W_{d'}$. 

The map in (iii)
amounts in interchanging the coordinates in $\widetilde\Sigma_d\subset W_d\times W_{d}$. Hence this map is an automorphism of $\Sigma_d$. 

Note that the map
$$\widetilde\Sigma_d\times\widetilde\Sigma_{d'}\to
\widetilde\Sigma_{dd'},\quad \left((\varphi,\varphi^{-1}), (\psi,\psi^{-1})\right)\mapsto (\psi\circ\varphi, \varphi^{-1}\circ \psi^{-1})\,,$$ is a morphism. This implies (iv).

In turn, (v) follows from the fact that the map
$$(V_d\times V_{d})\times (\A^n\times\A^n)\to \A^n\times\A^n,
\quad \left((\varphi,\psi), (x,y)\right)\mapsto (\varphi(x), \psi(y))\,,$$ is a morphism.\qed

\smallskip

(b) Let $I=(b_1,\ldots,b_k)$, where $b_i\in A=\cO_X(X)$, $i=1,\ldots,k$.
Clearly, $g\in\Aut (X,I)$ if and only if $b_i\circ g\in I$ $\forall i=1,\ldots,k$.
We claim that the latter condition defines a closed subset of $\Sigma_d$ for any $d\ge 1$. 
Indeed, by (a), given $d$ and $i$, there exists $m=m(d,i)$ such that $b_i\circ g\in A_m$
for any $g\in\Sigma_d$. Consider the map $\psi_i\colon \Aut X\to A$, $g\mapsto b_i\circ g$. 
By virtue of (a), $\psi_{i,d}=\psi_i|_{\Sigma_d}\colon\Sigma_d\to A_m\subset A$ is a morphism. 
Since $I\cap A_m$ is a linear subspace of the finite-dimensional vector space $A_m$, 
the inverse image $\psi_{i,d}^{-1}(I\cap A_m)=\psi_i^{-1}(I)\cap\Sigma_d$ is closed in $\Sigma_d$, as stated.
It follows that $\Aut (X,I)\subset \Aut X$ is a closed ind-subgroup.
\end{proof}

\begin{remarks}\label{rem: Ramanujam-connected} 1. Let $X$ and $R$ be algebraic varieties. 
According to \cite{Ra}, a map $R\to\Aut X$ is called  an {\em algebraic family of automorphisms} of $X$ 
if the action $R\times X\to X$, $(r,x)\mapsto r(x)$, is a morphism of varieties. We call such an algebraic family \emph{affine} if $R$ is an affine variety.

2. The {\em neutral component of $\Aut X$ in the sense of Ramanujam} is the union of all the irreducible subvarieties in $\Aut X$ passing through 
the neutral element $e\in\Aut X$. 
Clearly, this neutral component coincides with $\Aut^\circ X$ defined at the beginning of this section.
\end{remarks}

\begin{lemma}\label{lem:family} 
Let $X$ be an affine algebraic variety, and let $\tau\colon R\to\Aut X$ be an algebraic family of automorphisms of $X$. 
Consider an ind-group structure $\Aut X=\varinjlim \Sigma_d$  introduced in the proof of Proposition \ref{prop: ind-str}.a. 
Then the image $\tau(R)\subset\Aut X$ is contained in $\Sigma_d$ for some $d\in\N$, and the map $R\to\Sigma_d$ is a morphism. 
Consequently, $\tau\colon R\to\Aut X$ is a morphism of ind-varieties. 
 \end{lemma}

\begin{proof} We use the notation from the proof of Proposition \ref{prop: ind-str}.a.
Fix a closed embedding $X\hookrightarrow\A^n$, which corresponds to a choice of  generators $a_1,\ldots,a_n$ of the $\K$-algebra $A=\cO_X(X)$. 
We have $A=\K[x_1, \ldots , x_n]/I$, where $I\subset \K[x_1, \ldots , x_n]$ is the ideal of relations. 

 We claim that  the   action morphism 
$\alpha\colon R\times X\to X$ extends to a morphism $F=(F_1,\ldots,F_n)\colon R\times\A^n\to\A^n$, where $F_i\in\cO_R(R)[x_1,\ldots,x_n]$ for $i=1,\ldots,n$. 
Indeed, considering the induced homomorphism $$\alpha^*\colon A=\K[X] \to \K[R \times X]=\cO_R(R)\otimes A=
(\cO_R(R) \otimes \K[x_1, \ldots , x_n])/I\,,$$ we choose for every $i=1,\ldots,n$ a representative $F_i\in\cO_R(R)[x_1,\ldots,x_n]$ of 
the element $\alpha^*(a_i)\in\cO_R(R)[a_1,\ldots, a_n]$. This gives a desired extension $F=(F_1,\ldots,F_n)$ of $\alpha$.

 Let $d=\max_{i=1,\ldots,n}\deg F_i$. 
Then $F(r, \cdot)\in W_d$ for any fixed $r\in R$. This defines a morphism $\tau\colon R\to W_d$. 
The family $$\tau'\colon R\to\Aut X,\quad  \tau'\colon r\mapsto \tau(r)^{-1}\,,$$ is again algebraic (\cite{Ra}). 
Thus, $\tau'(R)\subset W_{d'}$ for some $d'\in\N$, and $\tau'\colon R\to W_{d'}$ is a morphism. 
This yields a morphism $$R\to\widetilde\Sigma_{\max\{d,d'\}}, \quad  r\mapsto (\tau(r),\tau'(r))\,.$$ 
Finally,  $\tau\colon R\to\Sigma_{\max\{d,d'\}}\cong \widetilde\Sigma_{\max\{d,d'\}}$ is a morphism, see the  proof of Proposition \ref{prop: ind-str}.a.
\end{proof}

The following corollary is immediate (cf.\ \cite[Prop.\ 2.5]{Kraft-notes}).

\begin{corollary} Up to equivalence, the structure  of an affine ind-group on $\Aut X$ introduced in the  proof of Proposition \ref{prop: ind-str}.a 
does not depend on the choice of a closed embedding $X\hookrightarrow\A^n$. 
\end{corollary}

\begin{definition}\label{def:semisimple}
Let $X$ be an affine variety. An element $g\in\Aut X$ will be called \emph{semisimple} if there exists a finite-dimensional $g$-stable\footnote{Abusing notation, 
we denote by the same letter $g$ the induced automorphism $f\mapsto f\circ g$ of the algebra $\cO_X(X)$.} subspace $V\subset \O_X(X)$
 which contains a system of generators $a_1,\ldots,a_m$ of $\O_X(X)$ and such that $g|_{V}\in\GL(V)$ is semisimple.
\end{definition}

Recall that an {\em algebraic quasitorus} is  a product of an algebraic torus and a finite Abelian group.

\begin{lemma}\label{lem: ss-elements}
An element $g\in\Aut X$ is semisimple if and only if $g$ is contained in a  closed algebraic quasitorus  $T\subset\Aut X$. 
\end{lemma}

\begin{proof} Assume that $g\in\Aut X$ is semisimple, and let $V\subset \O_X(X)$ be as in \ref{def:semisimple}. Then 
$g^*\in\GL(V^*)$ is contained in an algebraic torus $T'\subset\GL(V^*)$. Let  $T^*$  be the Zariski closure in $T'$ of the cyclic group $\langle g^*\rangle\subset T'$ 
generated by $g^*$. Then $T^*\subset T'$
 is an algebraic quasitorus. Consider the natural embedding $\phi\colon X\embed V^*$. 
Clearly, $g^*$ leaves invariant the image $\phi(X)\subset V^*$, and $g^*\circ\phi=\phi\circ g$. Hence also $T^*$ stabilizes $\phi(X)$. 
This yields an injective affine algebraic family $T^*\hookrightarrow \Aut X$. The image, say, $T \subset\Aut X$ of $T^*$ is an algebraic quasitorus containing $g$. 
By Lemma \ref{lem:family}, $T\subset \Sigma_d$ for some $d\in\N$.  We claim that $T$ is closed in $\Aut X$. 
Indeed, let $t\in \bar T\subset\Aut X$. Then both $\bar T$ and $t$ leave the subspace $V$ invariant, and $t^*\in \overline{T^*}\subset\GL(V^*)$. 
However, the quasitorus  $T^*\subset\GL(V^*)$ is closed, hence $t^*\in T^*$, and so, $t=\phi^{-1}\circ t^*\circ\phi\in T$. Thus, $T=\bar T$ is closed in $\Aut X$. 

To show the converse, recall
(see, e.g., \cite[\S 8.6]{Hu})  
that any algebraic group $G$ acting morphically on $X$ acts  {\em locally finitely} on $\cO_X(X)$, that is, each finite-dimensional subspace of $\cO_X(X)$ extends 
to a finite-dimensional  $G$-invariant subspace. This implies that any $g\in\Aut X$ contained in a closed algebraic quasitorus in $\Aut X$ is semisimple.
\end{proof}

\subsection{Nested ind-groups}\label{sss: nested grps}
\begin{definition}\label{def: nest}
We say that a group $G$ is
a {\em nested ind-group}
if $$G=\varinjlim G_i,\quad\mbox{where}\quad
G_1 \subset\ldots \subset G_i\subset G_{i+1}\subset\ldots $$
is  an increasing sequence of algebraic groups  and their closed embeddings.
The {\em rank} of a nested  ind-group $G$
is defined as ${\rm rk}\, G=\lim_{i\to\infty} {\rm rk}\, G_i$.
If all the $G_i$ are unipotent,
then we say that the nested ind-group $G$ is unipotent. The {\em
unipotent radical} $R_u(G)$ is the largest closed normal ind-subgroup of $G$
such that any element $g\in R_u(G)$ is unipotent in $G_i$ for all $i$
sufficiently large.\footnote{Cf.\ the notions of a locally linear ind-group
 and of
its unipotent radical in \cite[I.3]{DPW}.}
\end{definition}

\begin{remarks}\label{rem: nested sbgrp}
1. An algebraic group is a nested ind-group. By contrast, the ind-group $\ZZ$ is not a nested ind-group.

2.
A closed subgroup of
a nested ind-group is nested. If $G=\varinjlim G_i$ is nested, then also its neutral component $G^\circ=\varinjlim G^\circ_i$ is.

3. If $G=\varinjlim G_i$ and $G_i$ is connected for any $i\in\NN$, then $G$ is. Conversely, if $G=\varinjlim G_i$ is connected as an ind-group, then $G=\varinjlim G^\circ_i$. 
\end{remarks}

\begin{definition}\label{def: Cartan}
Let $G$ be a nested ind-group, and let $T$ be a maximal torus in $G$ of finite rank. The \emph{Cartan subgroup} $C^\circ_G(T)$ associated to $T$ is 
the neutral component of the centralizer $C_G(T)$ of $T$ in $G$. 

Clearly, $C^\circ_G(T)$  is a closed nested ind-subgroup of $G$, cf.\ Remarks \ref{rem: nested sbgrp}.2--\ref{rem: nested sbgrp}.3. 
Note that if $\rk G=0$, i.e., $G^\circ$ is unipotent, then $T=\{1\}$ and $C^\circ_G(T)=G^\circ$. 
In a semisimple algebraic group $G$ one has $C^\circ_G(T)=T$, and in a product $\tilde G=G\times U$, where $G$ is semisimple and $U$ is unipotent,   
one has $C^\circ_{\tilde G}(T)=T\times U$. 
\end{definition}

In a nested ind-group of automorphisms, one has the following analog of the Levi decomposition for a connected algebraic group (\cite{Mo}; cf.\ \cite[Thm.\ 4.10]{LZ}). 

\begin{theorem}\label{thm: Mostow} 
Let 
$G=\varinjlim G_i$ 
be a connected nested ind-group such that the sequence $\rk G_i$ is bounded above. Then $G$ admits a Levi decomposition $G=R_u(G)\rtimes L$, 
where $L$ is a maximal reductive algebraic subgroup in $G$ and 
$R_u(G)$ is the unipotent radical of $G$. Furthermore, any semisimple element $g\in G$ is contained in a maximal torus of $G$, and any two such tori are conjugated  in $G$. 
\end{theorem}

\begin{proof}
Due to Remark \ref{rem: nested sbgrp}.3 one may assume that $G_i$ is connected for any $i\ge 1$. 
For any $i\ge 1$ consider a Levi decomposition $G_i=U_i\rtimes L_i$, where $U_i=R_u(G_i)$. Since $G_i$ is connected, $L_i$ is connected as well for any $i\ge 1$. Since 
the ranks of the Levi factors $L_i$ are uniformly bounded, the dimensions
$\dim L_i$  are uniformly bounded, too. If $L=L_k$ is of maximal dimension, then $L$ is a Levi subgroup of $G_i$ for any $i\ge k$. Thus, $G_i=U_i\rtimes L$ $\forall i\ge k$.

Let us show that $U_i=U_{i+1}\cap G_i$ for $i\ge k$. Indeed, since $L\subset G_i\subset U_{i+1}\rtimes L$, given $(u,l)\in (U_{i+1}\rtimes L)\cap G_i$, 
one has $(u,1)=(u,l)\cdot (1,l^{-1})\in G_i$.
Thus, $U_i=U_{i+1}\cap G_i$. Clearly, $R_u(G)=\varinjlim U_i$, and so, the first assertion follows. 
The remaining conclusions hold because they hold for any connected algebraic subgroup $G_i$,  $i\ge 1$, see, e.g., \cite[Prop.\ 19.4 and Cor.\  21.3.A]
{Hu}. 
\end{proof}

\begin{corollary}\label{cor:quasi-proj}
The conclusions of Theorem \ref{thm: Mostow} hold for any connected nested ind-group $G=\varinjlim G_i$, which acts morphically and faithfully\footnote{The latter assumption was omitted in the previous version. We thank Hanspeter Kraft for indicating this.} on a quasi-projective variety $X$. 
\end{corollary}

\begin{proof}
It suffices to note that  the ranks $\rk G_i$ are bounded by $\dim X$ 
due to Proposition \ref{prop: basic facts}.a and Remark \ref{rem:quasi-proj}. 
\end{proof}

 \begin{remark}\label{rem:Levi-factor} In the notation of Theorem \ref{thm: Mostow}, consider for any $i\ge 1$ the subgroup
$G_i^\prime=R_u(G_i)\rtimes L$ of $G$. It is easily seen that $G=\varinjlim G_i=\varinjlim G_i^\prime$.
So, we  may assume in the sequel that $G_i=R_u(G_i)\rtimes L$ share the same Levi factor for all $i\ge 1$.
\end{remark}

\begin{corollary}\label{cor: abel-unip-rad}
In the notation and convention of Theorem \ref{thm: Mostow} and Remark \ref{rem:Levi-factor}, suppose that the unipotent radical $R_u(G)$ is Abelian. 
Then there exists a decomposition 
$R_u(G)=\bigoplus_{j=1}^{\infty} H_j$ such that $H_j$ is normal in $G$ for all $j\ge 1$ and $R_u(G_i)=\bigoplus_{j=1}^{i} H_j$.
\end{corollary}
\begin{proof} Let as before $U_i=R_u(G_i)$.
Since the adjoint representation of $L$ on the Lie algebra  $\Lie G_i$ is completely reducible, and the subalgebras $\Lie U_i$ and $\Lie U_{i-1}$ are $L$-stable, 
there is an $L$-stable subspace $V_i\subset \Lie U_i$  complementary to  $\Lie U_{i-1}$. By our assumption, $U_i$ is Abelian. 
Hence the Lie subalgebra $V_i$ corresponds to a subgroup $H_i\subset U_i$ normalized by $L$ and $U_i=U_{i-1}\oplus H_i$. Now the assertions follow.
\end{proof}

\begin{example}\label{rem: deJonq}
Let $X$ be a normal affine surface, let $\mu\colon X\to B$ be an $\A^1$-fibration over a smooth affine curve $B$, 
and let $\Aut(X,\mu)$ be the group of all automorphisms of $X$ preserving $\mu$. 
Then the neutral component $\Aut^\circ (X,\mu)$ is a nested ind-group with an Abelian unipotent radical, see Corollary~\ref{cor: nested}. 
Hence Corollary \ref{cor: abel-unip-rad} applies in this case. 
\end{example}

\begin{lemma}\label{lem:nest-mor}
Let $\phi\colon G\to H$ be a morphism of nested ind-groups $G=\varinjlim  G_i$ and $H=\varinjlim H_j$.
Assume that the orders $|G_i/G_i^\0|$, $i\in\NN$, are bounded above.
Then $\phi(G)$ is a closed nested ind-subgroup in $H$.
\end{lemma}
\begin{proof}
For a fixed index $j$ consider the increasing sequence of algebraic subgroups $\phi(G_i)\cap H_j$, $i\in\NN$, of the group $H_j$. 
Since their dimensions and the numbers of connected components are bounded, this sequence stabilizes. Hence $\phi(G)\cap H_j$ is  a (closed) algebraic subgroup.
\end{proof}

Let $X$ be an algebraic variety, and let $G=\varinjlim G_i$ be a connected nested ind-subgroup of $\Aut X$. 
Then $G$ is algebraically generated in the sense of \cite[Def.\ 1.1]{AFKKZ}.  The following result is an analog of \cite[Prop.\ 1.7]{AFKKZ} for nested ind-groups.

\begin{proposition}
Let $X$ be an affine variety, and let $G=\varinjlim  G_i$ be a connected nested ind-group, which is a closed subgroup of $\Aut X$. 
Then there exists  $i\ge 1$ such that any $G$-orbit in $X$ coincides with a $G_i$-orbit. 
\end{proposition}
\begin{proof} We may suppose that $ G_i\subset\Aut X$ is a closed, connected algebraic subgroup for any $i\ge 1$.  
We show first that for any $x\in X$, the
$G$-orbit $G x\subset X$ coincides with a $G_i$-orbit $G_i x$ for some $i\gg 1$.
Indeed,  the sequence $\{\dim G_i x\mid i=1,\ldots\}$ stabilizes, 
hence $\overline{G_N x}=\overline{G_{N+1} x}=\ldots$ for some $N\ge 1$. 
The decreasing sequence of closed subsets 
$\{\overline{G_i x}\setminus G_i x\mid i=N,\ldots\}$ also stabilizes, say, on  an $M$th step, where $M\ge N$. 
Thus, $G x = \bigcup_{i=1}^{\infty} G_i x=G_{M} x$.

According to Corollary \ref{cor:quasi-proj} and Remark \ref{rem:Levi-factor} 
we may suppose that $G_i=U_i\rtimes L$ for any $i\ge 1$, where $U_i=R_u(G_i)$ 
is the unipotent radical and $L$ is the Levi factor of $G$. Furthermore, we have $G=U\rtimes L$, 
where $U=\varinjlim  U_i$ is the unipotent radical of $G$. If $i\ge 1$ is such that $Ux=U_ix$ for any $x\in X$,
then also $Gx=G_ix$  for any $x\in X$, as stated. 
Thus, it suffices to prove the proposition assuming that $G=U$ and $G_i=U_i$, $i\in\N$, are unipotent groups. 

Let $$m=\max_{x\in X} \{\dim U x\}\quad \mbox{and}\quad m_i=\max_{x\in X} \{\dim U_i x\}\,.$$ By the first part of the proof, 
$m=m_{i_0}$ for some $i_0\ge 1$. By the Rosenlicht Lemma, there is a dense, open subset $\Omega\subset X$ such that $\dim U_{i_0} x=m_{i_0}=m=\dim Ux$ for any $x\in \Omega$. 
It is well known that any orbit of a unipotent algebraic group acting on an affine variety is closed and isomorphic to an affine space 
(see, e.g., Proposition \ref{prop: basic facts}.d below). 
It follows that $U_{i_0} x\cong\A^m$. By the first part of the proof, $U x=U_jx$  for some $j\ge 1$. Hence also $Ux\cong\A^m$. 
Since $U_{i_0}x\subset Ux$, it follows that $Ux=U_{i_0}x$ for any $x\in \Omega$. Indeed, an open subset of an affine space isomorphic to 
an affine space coincides with the ambient affine space. 

Let $X_1,\ldots,X_k$ be the irreducible components of $X\setminus \Omega$. Assuming that $\dim X>0$ one has $\dim X_j<\dim X$ for any $j=1,\ldots,k$. 
By induction on $\dim X$ we may suppose that for any $X_j$, $j=1,\ldots,k$, the orbits of $U|_{X_j}$ coincide with those of $U_{i_j}|_{X_j}$ for some $i_j\ge 1$ and 
for any $j=1,\ldots,k$. Then the same conclusion holds for $X$ with $i=\max\{i_0,i_1,\ldots,i_k\}$.
\end{proof}

\subsection{Amalgams} \label{ss: amalgams}
Recall (\cite{Se0}, \cite{Se}) that a {\em tree of groups} $(T, \mathcal{G})$ consists in
a combinatorial tree  $T$ along with  a collection $\mathcal{G}$ of vertex groups
$(G_P)_{P\in {\rm vert}\,T}$, edge groups $(G_\nu)_{\nu\in  {\rm edge}\,T}$, and
for each edge $\nu=[P,Q]$ of $T$, monomorphisms $G_\nu\to G_P$ and $G_\nu\to G_Q$
identifying $G_\nu$ with (common) subgroups of the vertex groups  $G_P$ and $G_Q$. We will suppose that any $G_\nu$ is a proper  subgroup of $G_P$ and $G_Q$.

Given such a tree of groups $(T, \mathcal{G})$, one can construct a unique group $G=\amalgGT$
called the {\em free amalgamated product}, or simply the {\em amalgam} of
$(T, \mathcal{G})$,
where $G$  is  freely generated by the subgroups
$(G_P)$ and $(G_\nu)$ with unified subgroups $G_P\cap G_Q=G_\nu$
for each $\nu=[P,Q]\in {\rm edge}\,T$.
We refer to \cite[Ch. I, \S\S 4,5]{Se}
for the existence and uniqueness of
the amalgam $G=\amalgGT$, its presentation and the universal property.
A subgroup $H\subset G$ is called of \emph{bounded length} if there exists an integer $N>0$ 
such that each element of $H$ can be decomposed into a product of at most $N$ elements of the vertex and edge groups.

For the reader's convenience, we sketch a proof of the following theorem. 

\begin{serreth}[{\cite[Ch. I, \S 4.3, Thm.\ 8 and \S 4.5, Exerc.\ 2]{Se}}]\label{th:Serre}
Any subgroup of bounded length
of an amalgam $G=\amalgGT$ is contained in a conjugate to one
of the  factors $G_P$, where $P\in {\rm vert}\,T$. 
\end{serreth}

\begin{proof} We follow the lines of the proof of Theorem 8 in \cite[Ch. I, \S 4.3]{Se}.
Let $\cT$ be a graph such that
\begin{itemize}\item each vertex in ${\rm vert}\,\cT$ is a left coset $g\cdot G_P$, $g\in G$, $P\in {\rm vert}\,T$;
\item each edge in ${\rm edge}\,\cT$ is a left coset $g\cdot G_\nu$, $g\in G$, $\nu=[P,Q]\in {\rm edge}\,T$.
\end{itemize}
Abusing notation, we let $P_g=g\cdot G_P$ and $\nu_g=g\cdot G_\nu$. Thus, $\nu_g=[P_g,Q_g]$. It is known 
(see \cite[Ch.\ I, \S 4, Thm.\ 10]{Se}, \cite[0.6-0.8]{Wr1}) that $\cT$ is a tree containing $T$ as a subtree. 
Indeed, consider a (reduced) word $g=a_1\cdot\ldots\cdot a_n$, where $a_i\in \bigcup_{Q\in {\rm vert}\,T} G_Q$ are such that 
$a_i, a_{i+1}$ do not belong to the same vertex group $G_Q$, and $a_n\not\in G_P$. 
Then the coset $g\cdot G_P$ can be joint with $G_P$ via a sequence of cosets $\{g_k\cdot G_P\}_{k=1,\ldots,n}$, where $g_k=a_k\cdot\ldots\cdot a_n$, so that $g=g_1$. 
This gives a path $(P,P_{g_n},\ldots,P_{g_1}=P_g)$ joining the vertices $P$ and $P_g$ in $\cT$. Hence the graph $\cT$ is connected. 
The absence of cycles in $\cT$ follows from the fact that $G=\amalgGT$ is a free amalgamated product.

There is a natural action of $G$ on $\cT$, 
$$h\colon P_g\mapsto P_{hg}, \quad \nu_g\mapsto \nu_{hg}\,,\quad\mbox{for}\quad h\in G\,$$ with a fundamental domain $T$. 
Under this action, the stabilizers of vertices are conjugated subgroups of the vertex groups:
$$\Stab_G (P_g)=g\cdot G_P\cdot g^{-1}\subset G\,.$$ Thus, a subgroup $H\subset G$ fixes a vertex $P_g$ if and only if $g^{-1}\cdot H\cdot g\subset G_P$. 
Similarly, $H$ fixes an edge $\nu_g$  if and only if $  g^{-1}\cdot  H\cdot  g \subset G_\nu$. 
It follows that, if $\nu_g=[P_g,Q_g]$ is fixed by $H$, then the both vertices $P_g,Q_g$ are fixed as well. 
The latter means that $H$ acts on the set of edges of $\cT$ without reversions.

To prove the theorem, it suffices to show that any subgroup $H\subset G$ of bounded length fixes some vertex $P_g\in {\rm vert}\,\cT$.
We claim that in fact any orbit $H\cdot P_g\subset {\rm vert}\,\cT$, where  $P_g\in {\rm vert}\,\cT$, contains a fixed point of $H$. 
Indeed, suppose that the lengths of the elements $h\in H$ are bounded by $l\in\N$. 
Then the diameter of the orbit $H\cdot P_g$ is bounded by $2l$ with respect to the graph metric on $\cT$. 
The subtree $\Q_g\subset\cT$ spanned by the orbit  $H\cdot P_g$ is stable under the action of $H$. 
Hence, the set of extremal vertices of $\Q_g$ is stable as well, along with the adjacent extremal edges. 
Suppressing the extremal vertices and edges of $\Q_g$ yields an $H$-stable subtree $\Q'_g\subset\Q_g$ of diameter ${\rm diam}\,\Q'_g={\rm diam}\,\Q_g-2\le 2l-2$. 
Continuing in this way, one arrives at a nonempty, $H$-stable subtree $\Q^{(k)}_g\subset\Q_g$ of diameter $\le 1$, 
which consists then either of a single vertex, or of a single edge. Anyway, its vertices are fixed under the action of $H$.
\end{proof}

\begin{remark}[{\em Pushing forward amalgamated free product structures}$\,$]\label{rem: amalgams}
Let $X'\to X$ be an \'etale
    Galois covering with the Galois group $\Gamma$, where $X$ and $X'$ are affine algebraic varieties and $\Gamma$ is finite.
    Assume that every automorphism $\alpha\in\Aut X$ admits a lift
    to an automorphism $\widetilde\alpha\in\Aut X'$. By the Monodromy Theorem,
    the latter holds, for instance, if $\K=\mathbb{C}$ and $X'\to X$ is the (finite) universal covering.
    Clearly, under this assumption the subset  in $\Aut X'$ of all lifts of the automorphisms in
    $\Aut X$ coincides with the normalizer ${\rm Norm}_{\Aut X'}(\Gamma)$ of $\Gamma$ in
$\Aut X'$. Furthermore,
    we have $\Aut X\cong {\rm Norm}_{\Aut X'}(\Gamma)/\Gamma$.

    Assume that $\Aut X'$ admits a structure of an amalgamated free product
$\Aut X'\cong A'*_{C'} B'$,
where $C'=A'\cap B'\supset \Gamma$. Then we have the inclusion
\begin{equation}\label{eq: normalizer} {\rm Norm}_{\Aut X'}(\Gamma)\supseteq
\left({\rm Norm}_{A}(\Gamma)\right)*_{{\rm Norm}_{C}(\Gamma)}
\left({\rm Norm}_{B}(\Gamma)\right)\,.\end{equation} This
inclusion can be strict, in general. However, in case of equality
the following holds (cf.\ \cite[Lem.~4.14]{AZ}).
\end{remark}

\begin{lemma}\label{lem: pushing amalgam} In the setup as before, assume that  the equality holds in {\rm (\ref{eq: normalizer})}. Then $\Aut X\cong A*_C B$ is an amalgam of
$A={\rm Norm}_{A'}(\Gamma)/\Gamma$ and $B={\rm Norm}_{B'}(\Gamma)/\Gamma$
along the joint subgroup $C=A\cap B={\rm Norm}_{C'}(\Gamma)/\Gamma$.
\end{lemma}

\subsection{Bearable automorphism groups}\label{ss: bearable}
Let us introduce the following notions.

\begin{definition}\label{def: bearable}   Let $G=\amalgGT$ be an
amalgam of a tree  $(T, \mathcal{H})$ of groups $H_P$, $P\in {\rm vert}\,T$. We say that $G$ is {\em $\alpha$-bearable}, 
where $\alpha$ is a cardinal number (e.g., {\em finitely bearable}, {\em countably bearable}, etc.), if  
\begin{itemize}
\item
${\rm vert}\,T$ has cardinal at most $\alpha$; 
\item $H_P$ is a nested ind-group for any $P\in {\rm vert}\,T$;
\item  any edge group is a proper subgroup  of the corresponding vertex groups.
\end{itemize}

\noindent A group will be called \emph{bearable} if it is $\alpha$-bearable for some cardinal $\alpha$, and  {\em unbearable} otherwise.
\end{definition}

\begin{remarks}\label{rem: ind-topology} 
1. It is easily seen that a nontrivial bearable group $G=\amalgGT$ is a nested ind-group if only if ${\rm vert}\,T$ consists of a single vertex.

2. A connected bearable group of automorphisms of an affine algebraic variety $X$ is algebraically generated in the sense of \cite[Def.\ 1.1]{AFKKZ}. 
Hence its orbits are locally closed smooth subvarieties of $X$ (\cite[Prop.\ 1.3]{AFKKZ}).

3. For a smooth affine surface $X$, the group $\SAut(X)$ generated by its one-parameter unipotent subgroups (see \ref{sit: ML-inv}) can have an open orbit in $X$, 
which is not closed. The corresponding examples are due to Gizatullin, Danilov, and the first author, see \cite{Ko1} and the references therein. 

4.
In all known examples of bearable automorphism groups of affine surfaces, the edge groups are linear algebraic groups. 
In these examples, infinite-dimensional nested groups are the automorphism groups of $\A^1$-fibrations. 
The intersection of two such groups preserves a pair of distinct $\A^1$-fibrations, and so, occurs to be an algebraic group (usually a quasitorus); 
cf., e.g., Example \ref{sss: sing-toric}. 
\end{remarks}

It is well known (\cite[Cor.\ 4.2]{Ka1}, \cite{Wr1}) that the conclusion of Serre's Theorem  \ref{th:Serre} holds for any algebraic subgroup of $\Aut \A^2$, 
where $\Aut \A^2$ is endowed with its usual amalgam structure.  More generally, we have the following analog \ref{serre} of Serre's Theorem \ref{th:Serre}. 
It will be used on several occasions in what follows. We adopt the following convention. 

\begin{sit} {\bf Convention.} Till the end of this section, that is, in \ref{serre}--\ref{lem: add-classes}, 
and also in \ref{cor: alg-grps-solv}, \ref{prop: 4.14}, and \ref{cor: unbearable}, we suppose that the ground field $\K$ is uncountable. \end{sit}

\begin{proposition}\label{serre}   Let $G=\amalgGT$ be a countably bearable group equipped with a structure of an ind-group $G=\varinjlim \Sigma_i$. 
Then any algebraic subgroup\footnote{By an algebraic subgroup $H$ in an ind-group $G$ we mean an algebraic subvariety of $G$, which is also a subgroup of $G$, 
such that the both structures on $H$ give an algebraic group structure.} $H\subset G$ is conjugated 
to an algebraic subgroup $H'$ of one of the nested ind-groups $G_P$, $P\in {\rm vert}\,T$. 
If, in addition, the vertex groups $G_P$, $P\in {\rm vert}\,T$, are closed in $G$, then $H$ and $H'$ are as well.
\end{proposition}

In the proof we use the following simple lemma.

\begin{lemma}\label{lem:constructible-cover} 
Let $E$ be an algebraic variety, and let $A_1\subset A_2\subset\ldots\subset E$ be an increasing sequence
of constructible subsets such that $E=\bigcup_{i\in\N} A_i$. Then $E\subset A_k$ for some $k\in\NN$.
\end{lemma}
\begin{proof}
Since the increasing sequence of closed subsets $\bar A_1\subset\bar A_1\subset\ldots$ exhausts $E$ and the base field $\K$ is uncountable, 
there exists $k_0>0$ such that $\bar A_{k_0}=E$. Indeed, otherwise $\dim\bar A_i< \dim E$ for any $i\ge 1$, and so, $E$ is a  countable union of closed subsets of smaller dimension, which is impossible.

The complement $E'=\overline{E\setminus  A_{k_0}}$ is a proper closed subset of $E$. Applying the same argument to the ascending sequence of constructible subsets 
$A_i^\prime=A_i\cap E'_j$, $i=1,2,\ldots$, and to any irreducible component $E_j$ of $E'$, $j=1,\ldots,l$, one can find $k_1>0$ such that $E'=\overline{A_{k_1}^\prime}$. 
Continuing in this way, we construct a strictly  descending sequence $E\supset E'\supset E''\supset\ldots$ of closed subsets of $E$. 
Since $E$ is Noetherian, this sequence is finite.
Thus, $E\subset A_k$ for $k=\max\{k_0,k_1,\ldots,k_n\}$, where $n+1$ is the length of the constructed descending sequence. 
\end{proof}
\begin{corollary}\label{cor:constructible-cover}
Let $H=\varinjlim H_i$ be an ind-variety, $A_1\subset A_2\subset\ldots\subset H$ be an increasing sequence
of constructive subsets such that $H=\bigcup_{i=1}^{\infty} A_i$, and $E\subset H$ be an algebraic subset.
Then $E\subset A_k$ for some $k\in\NN$.
\end{corollary}
\begin{proof}[Hint]
Apply Lemma \ref{lem:constructible-cover} to the sequence $E\cap A_i$, $i=1,2\ldots$
\end{proof}

\begin{proof}[Proof of Proposition \ref{serre}] To apply Serre's Theorem \ref{th:Serre}, we need to establish that $H$ is of bounded length in the amalgam $G=\amalgGT$. 
Let $G_P=\varinjlim_n \Sigma_{P,n}$ be the structure of a  nested  ind-group on the vertex group $G_P$ for $P\in {\rm vert}\,T$.
Given a finite
sequence $$\tau=((P_1,n_1),\ldots,(P_l,n_l))\in (({\rm vert}\,T)\times\NN)^l\,,$$
consider the  morphism
$\Sigma_{P_1,n_1}\times \ldots \times \Sigma_{P_l,n_l}\to G$  induced by the multiplication.
Its image $R_\tau$ is a constructible subset of the variety $\Sigma_i$ for some $i\in\NN$.
The amalgam $G=\amalgGT$ is covered by these
constructible sets.  Any two such sets $R_{\tau'}$ and $R_{\tau''}$ are contained in a third one $R_\tau$, and the collection $\{R_\tau\}_\tau$ is countable. 
Hence one can choose an increasing sequence $R_{\tau_1}\subset R_{\tau_2}\subset\ldots\subset G$ such that $G=\bigcup_{i\in\N} R_{\tau_i}$.

Let $E$ be an
algebraic subvariety of $G$. Due to Corollary \ref{cor:constructible-cover}, $E\subset R_{\tau_i}$ for some $\tau_i 
\in (({\rm vert}\,T)\times\NN)^l$, that is, $E$ is of bounded length ($\le l$).

In particular, any algebraic subgroup $H$ of $G$ has bounded length. 
By Serre's Theorem \ref{th:Serre}, $H$ is conjugated to a subgroup, say, $H'$ of a vertex nested ind-group $G_P$ for some $P\in {\rm vert}\,T$. 
A conjugation in an ind-group $G$ is an automorphism of $G$ viewed as an ind-variety. Hence $H'\subset G_P\subset G$  is again an algebraic subgroup. 
Therefore, it is contained in some algebraic subgroup $\Sigma_{P,n}\subset G_P$, and so, is closed in $\Sigma_{P,n}$ and then also in $G_P$. 
It is closed in $G$ provided $G_P$ is closed in $G$, and then also $H$ is closed in $G$.\end{proof}

\subsection{Algebraic subgroups of bearable groups
}\label{sec:reductive-unipotent}

In the sequel we need the following fact. \footnote{The authors are grateful to V.~Arzhantsev, R.~Avdeev, M.~Borovoi, M.~Brion, D.~Panushev, G. Soifer, D.~Timashev, and E.~Vinberg 
for useful discussions and indications.}

\begin{theorem}\label{thm: finite-ss}
Let $G$ be a reductive algebraic group. Then the set of conjugacy classes of connected reductive
subgroups of $G$ is at most countable.
\end{theorem}

The proof is based on the following result of Richardson (\cite[Cor.\ 11.5(b) and Prop.\ 12.1]{Ri0}, see also \cite[Thm.\ 8.1]{Ri}). 
Alternatively, the lemma can be deduced from the classification of semisimple subgroups of reductive groups started in \cite{Dy}.

\begin{lemma}\label{lem: finiteness-ss}
The set of conjugacy classes of connected semisimple algebraic
subgroups of a reductive algebraic group $G$ is finite.
\end{lemma}

\medskip

\noindent \emph{Proof of Theorem \ref{thm: finite-ss}.}
Let $H\subset G$ be a connected reductive subgroup. Consider the Levi decomposition $H=S\cdot T$, where $T={\rm Rad}\,H$ is a central torus in $H$ and 
$S=[H,H]$ is the commutator subgroup, see \cite[Prop.\ 2.2]{BT}. Since $S$ is semisimple,  its conjugacy class in $G$ is chosen among a finite set of such classes, see the claim. 
The torus $T$ is contained in the centralizer $C_G(S)$. Together with $S$, the centralizer $C_G(S)$ also runs over a finite set of conjugacy classes in $G$, 
along with a maximal  torus in $C_G(S)$ which contains $T$, since any two maximal tori in $C_G(S)$ are conjugated. 
Furthermore, the conjugation by elements of $C_G(S)$ act trivially on $T$, hence $T$ is contained in any maximal torus of $C_G(S)$. Fixing one of them, say,  
$\mathcal{T}$, by the rigidity of subtori of  $\mathcal{T}$ there is at most countable number of possibilities to choose $T\subset \mathcal{T}$. 
The conjugacy class of $H=S\cdot T$ is then also chosen among at most countable number of such classes.
\qed

\medskip

The following extension of Theorem \ref{thm: finite-ss} to the reductive subgroups in countably bearable groups will be used in \S\,\ref{ss: 1-parameter sbgrps}, 
see  Corollary \ref{cor: unbearable}.

\begin{proposition}\label{lem: mult-classes}  Suppose that the ground field $\K$ is uncountable. Then in any countably bearable group $G$,
the set of conjugacy classes of
connected reductive algebraic subgroups
is at most countable.
\end{proposition}

\begin{proof} Any connected reductive algebraic subgroup $F$ in a nested group $H=\varinjlim H_i$ is contained in an algebraic subgroup
$H_i$ for some $i$, and, moreover, in a maximal connected reductive subgroup of $H_i$.  
By the Mostow decomposition theorem, any two maximal connected reductive subgroups are conjugated in $H_i$.
By Theorem \ref{thm: finite-ss}, the conjugacy class of $F$ in $H_i$ runs over at most countable set of such classes. 
It follows that the set of conjugacy classes of connected reductive algebraic subgroups of $H$ is at most countable.

 By Proposition \ref{serre},
any algebraic subgroup of  the countably bearable group $G$ is conjugated
to an algebraic subgroup of one of the countable collection of generating nested ind-groups. 
Hence the set of conjugacy classes of connected reductive algebraic subgroups of $G$ is at most countable.
\end{proof}

\begin{remark}\label{rem: Dani} 
The same argument shows that in an $\alpha$-bearable group $\Aut X$ of rank $r$ the set of conjugacy classes of  $r$-tori has  cardinality at most $\alpha$. 
\end{remark}

\begin{example}\label{ex: Dani}
Consider, for instance, the Danielewski surface $S_n=\{x^ny-(z^2-1)=0\}$ in $\A^3$. 
The group $\Aut S_n$ is 2-bearable, see Theorem \ref{thm: BD2}. However, the group $\Aut (S_n\times \A^1)$ is not finitely bearable. 
Indeed (\cite{Dani}), this group of rank 2 contains a sequence of pairwise non-conjugated 2-tori. Is this group countably bearable?
\end{example}

Similarly, using Proposition \ref{serre} we obtain the following result. Recall that
two $\A^1$-fibrations on an affine variety $X$ are called {\em equivalent} if one can be
transformed into the other by an automorphism of $X$.

\begin{proposition}\label{lem: add-classes} Suppose that the ground field $\K$ is uncountable. Let $X$ be a normal affine variety. 
Assume that $X$ does not admit a unipotent group action with a general orbit of dimension $\ge 2$. 
\footnote{The latter holds, in particular, for any affine surface different from $\A^2$, see Proposition \ref{prop: basic facts}(b).}
If the group $\Aut^{\0} X$ is $\alpha$-bearable for some cardinal $\alpha\le\aleph_0$ \footnote{In particular, $\Aut^{\0} X$ is countably bearable.}, 
then the set of non-equivalent $\A^1$-fibrations on $X$ over affine bases is of cardinality at most $\alpha$.
\end{proposition}
\begin{proof}
Let $\Aut^{\0} X=\amalgGT$,
where $(T,\mathcal{G})$ is a tree of nested ind-groups $(G_P)_{P\in {\rm vert}\, T}$, where ${\rm vert}\, T$ is a set of cardinality $\alpha$.

For any $\A^1$-fibration $\mu\colon X\to Z$, where $Z$ is a normal affine variety,
one can find a one-parameter unipotent subgroup $U_\mu\subset\Aut X$ acting along the fibers of $\mu$.

Assume to the contrary that the set
of pairwise non-equivalent $\A^1$-fibrations on $X$ with affine bases is of cardinality larger than $\alpha$. 
By Proposition \ref{serre}, any one-parameter group $U_\mu$ is conjugated to
a subgroup $U'_\mu$ of one of
the vertex groups $G_P$, $P\in {\rm vert}\, T$. From our assumption,  for some
vertex $P\in {\rm vert}\, T$, the vertex group $G_P$ contains
at least two
one-parameter unipotent subgroup $U'_\mu\subset\Aut X$, which act along pairwise non-equivalent  $\A^1$-fibrations on $X$.

Since $G_P=\varinjlim G_{P,n}$ is a nested ind-group,  any
one-parameter unipotent
subgroup in $G_P$ is contained in some algebraic subgroup $G_{P,n}$.
It follows that for some $n\ge 1$,
the algebraic group $H=G_{P,n}$ contains at least two unipotent one parameter
subgroups $U'_{\mu_1}, U'_{\mu_2}$ acting along two non-equivalent
$\A^1$-fibrations
$\mu_i\colon X\to Z_i$, $i=1,2$.

Let $U_{\max}$ be a maximal unipotent subgroup of $H$. \footnote{This means that $U_{\max}$ is
the set of the unipotent elements in a Borel subgroup of $H$.}
Since any two such subgroups are conjugated,
any one-parameter unipotent subgroup $U\subset H$
is conjugated to a subgroup of $U_{\max}$. Hence $U_{\max}$
contains two
one-parameter unipotent subgroups, say, $U_1, U_2$ acting on $X$ along
two non-equivalent  $\A^1$-fibrations. It follows that the general orbits of $U_{\max}$
in $X$ are at least two-dimensional, contrary to our assumption.
\end{proof}

\section{Automorphism groups and amalgams: the first examples}\label{sec: first examples}
\begin{sit}[\emph{Classical surfaces}]\label{ss: classical surfaces} In this section we describe the automorphism goups of the surfaces
\begin{equation}\label{eq: classical} \AA^2,\quad\AA^1\times\AA_*^1,
\quad (\AA^1_*)^2,\quad
V_{d,e}=\AA^2/\mu_{d,e},\quad\PP^2\setminus C,
\quad\mbox{and}\quad  (\PP^1\times\PP^1)\setminus\Delta\,,\end{equation}
where $C\subset\PP^2$ is a smooth conic, $\Delta\subset \PP^1\times\PP^1$
the diagonal, and for any $1\le e <d$ with
$\gcd(d,e)=1$, $\mu_{d,e}\subset\GL(2,\K)$
stands for the cyclic group $\{{\rm diag}(\zeta,\zeta^e)\,|\,\zeta^d=1\}$.
For any one of these surfaces, its automorphism group carries an amalgam structure.
\end{sit}

\begin{sit}[\emph{Toric affine surfaces}]\label{sit: toric} These are the normal affine surfaces $X$ with the group $\Aut X$ of rank 2.
Any  toric  affine surface is
one from the list  
\begin{equation}\label{eq: toric} \AA^2,\quad\AA^1\times\AA_*^1,
\quad (\AA^1_*)^2,
 \quad\mbox{and}\quad V_{d,e}\,.\end{equation}
The torus action on $X$ comes from the action of the diagonal
2-torus $\TTT$ on $\A^2$ (see \cite{Co}).
The smooth toric surfaces $\A^2$, $\AA^1\times\AA_*^1$, and
$(\AA^1_*)^2$ are the underlying homogeneous spaces
of the solvable algebraic groups $\G_{\rm a}^2$,
$\Aff\A^1\cong\G_{\rm a}\rtimes\G_{\rm m}$, and $\TTT\cong\G_{\rm m}^2$,
respectively.
The toric affine surfaces fall into 3 classes as follows:
\begin{itemize}
\item $(\ML_0, 2)=\{ \AA^2, \,\,V_{d,e}\}$;
\item $(\ML_1, 2)=\{\AA^1\times\AA_*^1\}$;
\item $(\ML_2, 2)= \{(\AA^1_*)^2\}$.
\end{itemize}
In \ref{sit: A2}--\ref{sit: A1star-2} we describe their automorphism groups. 
\end{sit}

\begin{sit}[\emph{The group $\Aut \A^2$}]\label{sit: A2}
By the Jung--van der Kulk
Theorem\footnote{Valid over an arbitrary field.}
(see \cite[Thm.\ 2]{Ka1} and the references therein), $\Aut\A^2$
is the amalgamated free product $A*_C B$ of the affine group $A=\Aff(\AA^2)$ and
the de Jonqu\`eres subgroup $B={\rm J}(\AA^2)$
over their intersection $C=A\cap B$. The solvable group ${\rm J}(\AA^2)$  of rank 2
can be decomposed as ${\rm J}(\AA^2)\cong\mathbb{U}\rtimes \TTT$,
where  the unipotent radical
$ \mathbb{U}={\rm R_u}({\rm J}(\AA^2))$
is Abelian and consists of the triangular transformations
$(x,y)\mapsto (x+f(y),y)$ with $f\in \K[y]$. In particular, ${\rm J}(\AA^2)$ is a nested solvable
ind-group, see Definition \ref{def: nest}. More precisely, ${\rm J}(\AA^2)$ is
an inductive limit of a sequence
of solvable, connected affine algebraic groups of rank 2
with Abelian unipotent radicals.

It follows that any algebraic subgroup $G\subset \Aut\A^2$
is conjugate to a subgroup of one of the factors $\Aff(\AA^2)$ and ${\rm J}(\AA^2)$,
see \cite{DG2}, \cite[4.3--4.4]{Ka2}, \cite{Wr1}.
If $G$ is conjugate to
a subgroup of ${\rm J}(\AA^2)$, then $G$ is solvable of rank $\le 2$.
Hence, if a subgroup $G\subset \Aut\A^2$ is algebraic and non-solvable, then $G$ is
conjugated to a  subgroup of $\Aff(\AA^2)$ (cf.\ \cite[Cor.\ 4.4.]{Ka1}).
Any reductive subgroup
$G\subset\Aut\A^2$ is linearizable, i.e., is conjugated to a  subgroup of $\GL(2,\K)$
(\cite[Thm.~2.3]{KS}). Therefore, it is conjugated either to a subgroup of $\TTT$,
or to $\SL(2,\K)$, or finally to $\GL(2,\K)$.
\end{sit}

\begin{sit}[\emph{Automorphism groups of singular  toric affine surfaces}] \label{sss: sing-toric}
For the toric affine surfaces $X=\AA^{2}/\mu_{d,e}$ of class
($\ML_0$), there are analogs of the Jung-van der Kulk and Kambayashi-Wright Theorems,
see
\cite[Thms.~4.2, 4.15, 4.17]{AZ}. Once again, $\Aut X$
is an amalgamated free product $A^+*_C A^-$, where $C=A^+\cap A^-$.
To describe this decomposition  in more detail,
one has to distinguish between the following cases:
\begin{itemize}
\item[(i)] $e=1$;
\item[(ii)]
$e^2\equiv 1\mod d$ and $e\neq 1$;
\item[(iii)]
$e^2\not\equiv 1 \mod d$.
\end{itemize}
Note that the involution  $\tau: (x,y)\mapsto (y,x)$ acting on $\A^2$
normalizes the cyclic subgroup $\mu_{d,e}$ in $\GL(2,\K)$ in  cases (i) and (ii),
and does not normalize it in case (iii).

In case (i) $X=\AA^2/\mu_{d,1}$ is the {\em Veronese cone} $V_d$, i.e.
the affine cone over a rational normal
curve in $\PP^d$.  Since $\mu_{d,1}\subset\TTT$ is central, the standard
$\SL(2,\K)$-action on $\A^2$ descends to $V_d$.
The complement to the vertex of  $V_d$ is the open
$\SL(2,\K)$-orbit; the same is true for the natural $\GL(2,\K)$-action on $V_d$.
Thus $V_d$ is
a quasihomogeneous $\SL(2,\K)$-variety.

The amalgam structure  $$\Aut X\cong A^+*_C A^-, \quad\mbox{where} \quad C=A^+\cap A^-\,,$$ is naturally related
to that on $\Aut \A^2$, see \ref{sit: A2}.
Consider the normalizer $$\mathcal{N}_{d,e}^+={\rm Norm}_{{\rm J}(\AA^2)}(\mu_{d,e})=
\mathbb{U}_{d,e}\rtimes\TTT\subset {\rm J}(\AA^2)\,,$$ where $$\mathbb{U}_{d,e}=
\left\{(x,y)\mapsto (x+f(y),y)\,\vert\,f\in y^e\K[y^d]\right\}\subset\mathbb{U}\,.$$
In case (i) we have 
$$A^+=\mathcal{N}_{d,1}^+/\mu_{d,1}\quad\mbox{and}\quad A^-=\GL(2,\K)/\mu_{d,1}$$ (cf.\ \cite[\S 11]{DG2}).
Similarly, in case (ii) 
$$A^+=\mathcal{N}_{d, e}^+/\mu_{d,e}\quad\mbox{and}\quad A^-
=\langle \TTT,\tau\rangle/\mu_{d, e}\,.$$
Finally, in case (iii) 
$$A^\pm=\mathcal{N}_{d,e}^\pm/\mu_{d,e}\quad\mbox{with}\quad
\mathcal{N}_{d,e}^-=\tau
\mathcal{N}_{d,e}^+\tau\,.$$
\end{sit}

\begin{sit}[\emph{The group $\Aut (\AA^{1}\times\AA^{1}_*)$}]\label{sit: A1-A1-star}
For the surface $X=\AA^{1}\times\AA^{1}_*$ of class
($\ML_1$) we have
$$\Aut X\cong\Aut^{\0} X\rtimes (\ZZ\rtimes (\ZZ/2\ZZ))=(U\rtimes \TTT)\rtimes (\ZZ\rtimes (\ZZ/2\ZZ))\,,$$
where the factor $\ZZ$ is generated by the transformation
$(x,y)\mapsto (xy,y)$, the factor $\ZZ/2\ZZ$ by the involution
$(x,y)\mapsto (x,y^{-1})$,
and the Abelian unipotent radical $U$ of $\Aut X$ is
$$U=\left\{\exp\left (p(y)\frac{d}{d x}\right)\colon
(x,y)\mapsto (x+p(y), y)\;|\; p(y)\in \K[y,y^{-1}]\right\}\,.$$
The group
$\Aut X$ is solvable, and so, any algebraic group acting effectively on
$\AA^{1}\times\AA^{1}_*$ is as well.  In fact, for any affine surface of class
($\ML_1$) the automorphism group has similar properties, see Section~\ref{sec: dJ}.
\end{sit}

\begin{sit}[\emph{The group $\Aut (\AA^{1}_*)^{2}$}]\label{sit: A1star-2}
If $X=(\AA^{1}_*)^{2}$, then
\[
\Aut X=\left\{(x,y)\mapsto(t_1 x^{a}y^{b}, t_2 x^{c}y^{d})\;\left|\;\left(\begin{smallmatrix} a&b\\ c&d \end{smallmatrix}\right)\in \GL_2(\ZZ), (t_1,t_2)\in\TTT\right.\right\}\cong\TTT\rtimes \GL_2(\ZZ).
\]
Indeed, $\Aut X$ surjects onto the automorphism group of the multiplicative group
$$\cO_X^\times(X)=\{t x^{a}y^{b}\,|\,t\in\K^{*}, a,b\in\ZZ\}$$ of the ring $\cO_X(X)$.
\end{sit}

\begin{sit}[\emph
{The group $\Aut((\PP^{1}\times \PP^{1})\setminus\Delta)$}]\label{sit: Aut P1-P1}
The surface $X'=(\PP^{1}\times \PP^{1})\setminus\Delta$ is isomorphic to the
smooth quadric in $\A^3$ with equation $xy-z^2=-1$.
The group $\Aut X'$ was studied in \cite[10.1]{DG2}, \cite[\S 2.1, Thm.\ 4]{La}, and \cite{ML1}.
The result in
\cite{La} can be  interpreted as follows. There is an amalgam
$$\Aut X'\cong A' *_{C'} B'\quad\mbox{with}\quad C'=A'\cap B'\cong \G_{\rm m}\times\bangle{\tau}\,,$$ where
$\tau:(u,v)\mapsto (v,u)$ is the involution  interchanging the factors of
$\PP^{1}\times \PP^{1}$,
\begin{equation}\label{eq: Aut P1-P1-A}
A'=\Aut (\PP^1\times\PP^1, \Delta)\cong
\PSL(2,\K)\times\bangle{\tau}\,,
\end{equation}  
and
\begin{equation}\label{eq: Aut P1-P1-B}
B'=(U'_\infty\rtimes\G_{\rm m})
\rtimes\bangle{\tau}\quad\mbox{with}
\quad U'_\infty=R_u(B')\cong\K[t'] \,
\end{equation}
being the unipotent radical
of $B'$.  In particular, $A'$  is semisimple
and $B'$ is solvable of rank 1.
In the affine coordinates $(u,v)$ in $\PP^{1}\times \PP^{1}$,
where $u=u_0/u_1$ and $v=v_0/v_1$, we have
$$U'_\infty\rtimes\G_{\rm m}=\left\{
(u,v)\mapsto \left(cu+P, cv+P\right)\,|\, c\in\G_{\rm m},\,\, P\in\K\left[\frac{1}{u-v}\right]\right\}\,.$$
\end{sit}

\begin{sit}[\emph{The group $\Aut(\PP^{2}\setminus C)$}] \label{sit: Aut P2-conic}
For $X=\PP^{2}\setminus C$ the group $\Aut X$ was studied in \cite[\S 2]{DG2}.
By {\it loc.cit.}, there is an amalgam
$\Aut X\cong A *_{C} B$ with $$A=\Aut (\PP^2, C)\cong {\rm SO}(3,\K)\cong\PSL(2,\K)
\,\,\,\, \mbox{and}\,\,\,\,  B=U_\infty\rtimes\G_{\rm m}\,,$$ where
$U_\infty\cong\K[t]$ and $C=A\cap B\cong\Aff(\A^1)$.
\end{sit}

\begin{remark}  \label{rem: Z2-cover}
The amalgam in \ref{sit: Aut P2-conic}
is pushforward of the one in \ref{sit: Aut P1-P1}
via the construction of Remark~\ref{rem: amalgams}.  Indeed, the surface
$X'=(\PP^{1}\times \PP^{1})\setminus\Delta$ is the Galois covering of
$X=\PP^{2}\setminus C$, and more precisely, $X=X'/\ZZ_2$,
where $\ZZ_2=\langle\tau\rangle$. In the affine coordinates $(u,v)$, the quotient morphism
$\PP^{1}\times \PP^{1}\to\PP^2$
is given by the elementary symmetric polynomials in two variables via  the classical Vieta formulas;
cf.\  \cite[Ex.\ 5.1]{FZ-LND}.
The Galois $\ZZ_2$-covering $X'\to X$ being the universal covering,
we have $\Aut X\cong {\rm Norm}_{\Aut X'}(\tau)/\bangle{\tau}$.
A comparison of the explicit formulas in \cite[(2.4.3;l)]{DG2} and \cite[\S 2]{La}
yields the isomorphism $${\rm Norm}_{\Aut X'}(\tau)\cong
 {\rm Norm}_{A'}(\tau) *_{{\rm Norm}_{C'}(\tau)} {\rm Norm}_{B'}(\tau)\,,$$
where $$ {\rm Norm}_{A'}(\tau) = A'\cong
\PSL(2,\K)\times\bangle{\tau}\quad\mbox{and}\quad {\rm Norm}_{B'}(\tau)=
B_0\times\bangle{\tau}\subset B',$$ and where
$$
B_0=\left\{\left. (u,v)\mapsto \left(cu+P,
cv+P\right) \; \right| \;c\in\K^{\times},P\in\K\left[\frac{1}{(u-v)^{2}}\right]\right\}=
U_\infty\times\G_{\rm m}\,$$
and $U_\infty=\K[{t'}^2]\subset U'_\infty=\K[t']$ with $t'=\frac{1}{u-v}$.  Finally,
$$A\cong {\rm Norm}_{A'}(\tau)/\langle\tau\rangle=A'/\langle\tau\rangle\cong
\PSL(2,\K)$$
and
$$B\cong {\rm Norm}_{B'}(\tau)/\langle\tau\rangle=B_0
\cong U_\infty\times\G_{\rm m}\,.$$
\end{remark}

Summarizing the results in
\ref{sit: A2}-\ref{sit: Aut P2-conic} we arrive at the following conclusion.

\begin{theorem}\label{thm: ex-bearable} 
For any surface $X$ in {\rm (\ref{eq: classical})} the neutral component $\Aut^\circ X\cong A^\circ\star_{C^\circ} B^\circ$ is finitely bearable, 
where $A^\circ$, $B^\circ$, and $C^\circ=A^\circ\cap B^\circ$ are connected, and either $A^\circ$ and $B^\circ$ are both solvable nested ind-groups, 
or $A^\circ$ is such a group and $B^\circ$ is a reductive affine algebraic group.
\end{theorem}

\begin{corollary}\footnote{Cf.\ Proposition \ref{prop: 4.14}.}\label{cor: alg-grps-solv}  Suppose that the base field $\K$ is uncountable. 
Let $X$ in  {\rm (\ref{eq: classical})} be one of the surfaces $\A^1\times\A^1_*, (\A^1_*)^2$, or $V_{d,e}$ with $e>1$.  
Then any connected algebraic group $G$ acting effectively on $X$ is solvable.
\end{corollary}

\begin{proof}
Indeed,  by Proposition \ref{serre}, $G$ is conjugated to a subgroup of one of the factors $A^\circ$ and $B^\circ$, which are both solvable in these cases.
\end{proof}

\section{Algebraic group actions on affine surfaces}
\label{sec: prelim}

\subsection{Generalities}\label{ss: generalities}
In this section we recall some
general facts about algebraic group actions on affine varieties and their specialization to the case of affine surfaces. 
By a $G$-variety we mean a variety with an effective (regular) action of an (algebraic) group $G$.
The next proposition is well known (see, e.g., \cite{Po1, Po2} and \cite[Lem.\ 2.7 and 2.9]{FZ-LND});
for the reader's convenience we provide either a short argument, or a reference.

\begin{proposition}\label{prop: basic facts}
Let $G$ be a connected affine algebraic group, and let $X$ be a normal affine $G$-variety.
Then the following hold.
\begin{itemize}
\item[(a)]  We have
$\rk G\le \dim X$, and $\rk G=\dim X$ if and only if $X$ is toric.
\item[(b)]  If $G$ is solvable and acts transitively on $X$, then
$X\cong\A^k\times (\A^1_*)^l$ for some $k,l\ge 0$.
\item[(c)]  If $G$ is solvable and acts on $X$ with an open orbit $O$,
then either $O=X$, or $X\setminus O$ is a divisor.
\item[(d)] If $G$ is unipotent and has  an open orbit in $X$, then $X\cong\AA^{n}$.
\item[(e)] If $G$ is reductive and acts with an open orbit, then it has a unique closed orbit, and this orbit
lies in the closure of any other orbit.
\item[(f)] If $G$ is semisimple, then $G$ has no one-dimensional orbit in $X$.
\end{itemize}
\end{proposition}
\begin{proof}
(a) Let $T\subset G$ be a maximal torus. By the rigidity  of algebraic subtori, algebraic subgroups of $T$ form a countable set. Since $T$ acts effectively on $X$,
the isotropy subgroup of $T$ at a generic point of $X$ is trivial due to the
aforementioned rigidity.
Hence $\rk G=\dim T\le\dim X$.
The second assertion in (a) follows by definition of a toric variety.

For (b) see \cite[Thm.\ 2]{Po2}.

(c) follows from (b) since the open orbit is affine in this case.

(d) Let $O$ be the  open orbit of $U$. Since any orbit of a unipotent group acting on an affine variety
is closed, $O=X$. Now the result follows from the corollary of Theorem 2 in {\em loc.cit.}

(e) is proven in \cite[Prop.\ 2]{Po1}.

(f) Since $G$ is semisimple, it admits no non-trivial homomorphism to $\Aff (\A^1)$.
Indeed, otherwise, it would act non-trivially on $\PP^{1}$ with a fixed point. 
Such an action can be  lifted to a non-trivial linear representation of the universal covering group $\tilde G$ in $\GL(2,\K)$
with a trivial one-dimensional subrepresentation, which is impossible.

It follows that $G$ cannot act non-trivially on a curve. Now (f) follows.
\end{proof}

\begin{remark}\label{rem:quasi-proj}
Note that the proof of (a) works equally for any quasi-projective variety with an effective action of $G$.
\end{remark}

\begin{corollary}\label{cor: basic facts}
Let  $X$ be a normal affine $G$-surface. Then the following hold.
\begin{itemize}
\item[(g)]  If $G$ is non-Abelian and unipotent, then $X\cong\A^2$.
\item[(h)] If $G$ semisimple, then it has
 an open orbit in $X$ with a finite complement.
\end{itemize}
\end{corollary}

\begin{proof}
(g) If $G$ acts with an open orbit on $X$, then the result follows from Proposition~\ref{prop: basic facts}(d).
Otherwise, $G$ acts with one-dimensional general orbits, and so, any one-parameter subgroup $H\subset G$
has the same algebra of invariants:
$\mathcal{O}(X)^H=\mathcal{O}(X)^G$. This algebra is affine and its spectrum $Z$ is a smooth affine curve (see \cite[Lem.~1.1]{Fi}).
Any fiber of the induced $\AA^{1}$-fibration $\pi\colon X\to Z$ is stable under the $G$-action.

If $H'$ is another one-parameter subgroup of $G$,
then the actions of $H$ and $H'$ commute. Hence these subgroups commute, and so,
$G$ is Abelian, contrary to our assumption.

Finally, (h) is  immediate from (f).
\end{proof}

\begin{remark}  The affine plane is exceptional with respect to the property
in (g).  Indeed, the Heisenberg group 
$$H=\left\{\left(\begin{matrix} 1 & a & b\\
0 & 1 & c\\
0 & 0 & 1\end{matrix}\right), \quad a,b,c\in\K\right\}\, $$ is a  non-Abelian
unipotent group acting effectively on $\A^2$ via $(x,y)\mapsto (x+ay+b, y+c)$.
\end{remark}

Let us also mention the following results.

\begin{proposition}[{\rm \cite[Thm.\ 3.3, Cor.\ 3.4]{FZ-unique}}]\label{prop-FZ-unique} 
Suppose that a normal affine surface $X\not\cong\A^1_*\times\A^1_*$ admits two effective $\G_m$-actions with distinct orbits, 
that is, with infinitesimal generators $\delta$, $\tilde\delta$, 
where $\delta\neq\pm \tilde\delta$. Then $X$ admits as well a nontrivial $\G_a$-action. 
Furthermore, if $X$ is not toric, then any two effective $\G_m$-actions on $X$, 
after possibly switching one of them by the automorphism $\lambda\mapsto \lambda^{-1}$ of $\G_m$, 
are conjugate via an automorphism provided by a $\G_a$-action on $X$.
\end{proposition}

\begin{remark} See also \cite[Thm.\ 1]{AG} for a generalization to higher dimensions. \end{remark}

\begin{proposition}\label{prop-Iitaka} 
If the neutral component $\Aut^\circ X$ of an affine algebraic variety $X$ with $\dim X\ge 2$ is an algebraic group, then $\Aut^\circ X\cong (\G_m)^r$ is an algebraic torus. 
 This is the case, in particular, for surfaces of class $(\ML_2)$. 
\end{proposition}

\begin{proof} 
The first assertion follows easily  by a lemma of Iitaka \cite[Lem.\ 3]{Ii}, see, e.g., \cite[Thm.\ 1.3]{Kr} and \cite[Thm.\ 4.10(a)]{LZ}. 
The second will be proven in a forthcoming paper \cite{PZ}. 
Let us indicate an  independent approach in the particular case of surfaces $X\in (\ML_2)$ of positive rank $r=\rk \Aut X\ge 1$. 

If $r=2$, then $X$ is a toric surface. However, by \cite{Li} the only affine toric surface of class $(\ML_2)$ is the 2-torus $X=(\A^1_*)^2$ with $\Aut^\circ X=(\G_m)^2$. 

Due to Proposition \ref{prop-FZ-unique}, for a surface $X$ of class $(\ML_2)$, 
the group $\Aut X$ contains a unique algebraic torus $\TTT$. Hence, $\TTT \subset G=\Aut^\circ X$ is a normal subgroup. 
Suppose further that $r=1$, that is, $\TTT\cong\G_m$, and so, $X$ is a $\G_m$-surface. 
Consider the set $F$ consisting of the fixed points of $\TTT$ and of a finite union of all those one-dimensional orbits of $\TTT$, 
which make obstacle to existence of a geometric $\TTT$-quotient. Since $G$ normalizes $\TTT$, $F$ is $G$-stable. 
Its complement $U=X\setminus F$ admits a geometric quotient $C=U/\TTT$, where $C$ is an algebraic curve. This yields a homomorphism 
$G\to\Aut^\circ C$. Its kernel $H\supset\TTT$ stabilizes general  $\TTT$-orbits. 
For such an orbit $O\cong\A^1_*$ one has $\TTT|_O=\Aut^\circ O\cong\G_m$. 
Since $\Aut^\circ O\supset H|_O\supset\TTT|_O$, we have $H|_O=\TTT|_O$. It follows that $H=\TTT$. 
Hence $G$ is an extension of $\TTT$ by a connected subgroup of the algebraic group $\Aut^\circ C$, that is, $G$ is an algebraic group of dimension $\le 2$. 
By the first part of the proposition, $\Aut^\circ X=G=\TTT\cong\G_m$.
\end{proof}

\subsection{Quasihomogeneous affine surfaces}\label{ss: ss}
\label{sit: semisimple}
\begin{theorem}[Gizatullin--Popov, \cite{Gi}, \cite{Po2}]\label{th:GP}
A normal affine surface $X$ admitting an action of
an algebraic group with an open orbit whose complement is finite,
is one of the surfaces
\begin{equation}\label{eq: big-orbit}
\AA^2,\quad \AA^1\times\AA_*^1,\quad  (\AA^1_*)^2, \quad V_d, \,\,\,d\ge 2,
\quad\PP^2\setminus C,
\quad\mbox{and}\quad  (\PP^1\times\PP^1)\setminus\Delta\,,
\end{equation}
where $V_d=V_{d,1}$ is a Veronese cone (see \ref{ss: classical surfaces}), $C$ is a smooth conic in $\PP^2$, and $\Delta$
is the diagonal in $\PP^1\times\PP^1$.
\end{theorem}

Using Proposition \ref{prop: basic facts}(f) we deduce the following corollary.

\begin{corollary}
Let $X$ be a normal affine surface. Then the following conditions are equivalent:
\begin{itemize}
\item[(i)] $X$ admits a nontrivial action of a connected semisimple group;
\item[(ii)] $X$ is spherical, that is, it admits a semisimple group action, such that a Borel subgroup has an open orbit;
\item[(iii)]  $X$ is one of the surfaces
\begin{equation}\label{eq:semisimple-list}
\AA^2,\quad V_d, \,\,\,d\ge 2,
\quad\PP^2\setminus C,
\quad  (\PP^1\times\PP^1)\setminus\Delta.
\end{equation}
\end{itemize}
\end{corollary}

The following proposition is a version of Proposition 4.14 in \cite{FZ-LND}. We provide a new proof.

\begin{proposition}\label{prop: 4.14}  Suppose that the ground field $\K$ is uncountable. 
Let $X$ be a normal affine $G$-surface,
where $G\subset \Aut X$ is
a connected reductive algebraic group different from a torus. Then the pair $(X,G)$ is one of the following:
\begin{itemize}\item $(\AA^2, \GL(2,\K))$ and  $(\AA^2, \SL(2,\K))$;
\item $(V_d, \GL(2,\K)/\mu_{d})$, $d\ge 2$, $(V_d, \SL(2,\K))$, $d\ge 3$ odd,
and $(V_d, \PSL(2,\K))$,
$d\ge 2$ even;
\item $(\PP^2\setminus C, \PSL(2,\K))$ and $( (\PP^1\times\PP^1)\setminus\Delta, \PSL(2,\K))$.
\end{itemize}
Furthermore,
the action of $G$ on $X$ is unique up to a conjugation in the group $\Aut X$.
\end{proposition}

\begin{proof}
Under our assumptions $G$ contains a nontrivial semisimple subgroup. 
Hence $G$ acts with an open orbit, which has a finite complement in $X$, see Corollary~\ref{cor: basic facts}(h). 
By Gizatullin-Popov Theorem~\ref{th:GP} $X$ is one of the list (\ref{eq:semisimple-list}). 
Due to the results cited in Section \ref{sec: first examples},  the group $\Aut X$ is an amalgam of two closed nested ind-groups.
By Proposition \ref{serre},
$G$ is conjugated to a subgroup of the non-solvable factor
in the amalgam decomposition of $\Aut X$.

If $X$ is one of the surfaces
$\PP^2\setminus C$ and
$ (\PP^1\times\PP^1)\setminus\Delta$, then by \ref{sit: Aut P1-P1} and \ref{sit: Aut P2-conic},  $G=
\PSL(2,\K)$, and the $G$-action on $X$  is unique up to conjugation;
see also \cite{Po1}, \cite{DG2}, and \cite[Prop.\ 4.14]{FZ-LND} for alternative proofs.

Similarly, for $X=\A^2$ the assertion follows from \ref{sit: A2}. If
$X=V_d$, $d\ge 2$, then by \ref{sss: sing-toric}, case (i), $G$ is conjugated in  $\Aut X$ to a
subgroup of the non-solvable factor $\GL(2,\K)/\mu_{d,1}$, where
$\mu_{d,1}\cong\ZZ/d\ZZ$ is contained in the center of $\GL(2,\K)$. Hence $G$ is
conjugated either to $\GL(2,\K)/\mu_{d,1}$ itself, or to $\SL(2,\K)$, or to
$\PSL(2,\K)$ canonically embedded in $\GL(2,\K)/\mu_{d,1}$, depending on the parity of $d$.
\end{proof}

\subsection{Actions with an open orbit}\label{sec: open orbit}
Recall that an effective $\G_{\rm m}$-action on a normal affine variety
defines a grading $A=\bigoplus_{j\in\ZZ} A_j$
 on the algebra $A=\mathcal{O}(X)$. For
$\dim X=2$ the $\G_{\rm m}$-action is called {\em elliptic} if
$A_j=0$ $\forall j<0$
and $A_0=\K$, {\em parabolic} if
$A_j=0$ $\forall j<0$ and $A_0\neq\K$, and {\em hyperbolic} if $A_{-1}\neq 0\neq A_1$.

The following result is essentially Proposition 2.10 in \cite{FZ-LND};
cf.\ also \cite[Prop.~2.5]{Be}\footnote{In \cite[Prop.~2.5]{Be} the condition
$\mathcal{O}(X)^{\G_{\rm m}}\neq\mathcal{O}(X)^{\G_{\rm a}}$ is lacking;
furthermore, the proof in \cite{Be}
assumes implicitly smoothness of $X$.}.
For the reader's convenience, we sketch a proof.

\begin{proposition}\label{prop: dense-orbit}
 Let $X$ be a normal affine surface different from the surfaces in {\rm (\ref{eq: classical})}.
 Then the following are equivalent:
\begin{itemize}\item[(i)] $X$ admits an effective action of a connected affine algebraic group
$G$ with an open orbit; \item[(ii)] $X$ admits an action of a semi-direct product
$\G_{\rm a}\rtimes\G_{\rm m}$  with an open orbit; \item[(iii)]
$X$ admits effective $\G_{\rm a}$- and
$\G_{\rm m}$-actions such that $\mathcal{O}(X)^{\G_{\rm m}}\neq\mathcal{O}(X)^{\G_{\rm a}}$.
\end{itemize}
Moreover, any $\G_{\rm m}$-action on $X$ as in {\rm (iii)}  is hyperbolic,
$X$ is of class $(\ML_0\cup \ML_1,1)$, and $X$ is a cyclic quotient of a normalization of a Danielewski surface $\{xy-P(z)=0\}$ in $\A^3$ for some $P\in\K[z]$.
\end{proposition}

\begin{proof}
(i)$\Leftrightarrow$(ii). As follows from Proposition \ref{prop: 4.14} and Theorem \ref{th:GP},
under our assumption the group $G$ as in (i) does not contain any semisimple subgroup, and so,
is solvable.  Since by assumption $X$ is non-toric,
$\rk G\le 1$. In fact, $\rk G=1$. Indeed, otherwise
$G$ is unipotent and acts with an open orbit. Since the orbits of a unipotent group acting on an affine variety are closed, $G$ is transitive in $X$. 
By Proposition \ref{prop: basic facts}(b), $X\cong\A^2$, which is excluded by our assumption.

The open orbit $O$ of $G$ in $X$
coincides with an open orbit (isomorphic to one of $\AA^2$, $\AA^1\times\AA_*^1$, $(\AA^1_*)^2$) of
a two-dimensional subgroup $H\cong \G_{\rm a}\rtimes\G_{\rm m}$ of $G$
(\cite[Lem.\ 2.9(b)]{FZ-LND}). This gives (i)$\Rightarrow$(ii).
The implication  (ii)$\Rightarrow$(i) is immediate, hence we have (i)$\Leftrightarrow$(ii).

(ii)$\Rightarrow$(iii). By (ii) $X$ is a
$\G_{\rm m}$-surface admitting a \emph{horizontal $\G_{\rm a}$-action}.
The latter means that the general
$\G_{\rm a}$-orbits are not the closures of general $\G_{\rm m}$-orbits, which implies (iii).

(iii)$\Rightarrow$(ii). We claim that any affine variety with effective $\G_{\rm a}$- and
$\G_{\rm m}$-actions as in (iii) possesses an effective action of a semi-direct product
$\G_{\rm a}\rtimes\G_{\rm m}$.
Indeed, let $A=\bigoplus_{i\in\ZZ} A_i$ be the grading of the algebra $A=\mathcal{O}(X)$
induced by the $\G_{\rm m}$-action,
and let $\partial\in {\rm Der} A$
be the locally nilpotent derivation corresponding to the $\G_{\rm a}$-action.
 Write $\partial=\sum_{i=k}^l \partial_i$, where $k\le l$
and $\partial_i\in {\rm Der} A$ is a homogeneous derivation of degree $i$ with
$\partial_k\neq 0\neq\partial_l$. Then $\partial_k,\partial_l$  are again locally nilpotent
(\cite[Lem.\ 2.1]{FZ-LND}, \cite{Re}).
Then the $\G_{\rm a}$-actions on $X$ generated by $\partial_k$ and $\partial_l$
are normalized by the
$\G_{\rm m}$-action. This yields the existence of an
$(\G_{\rm a}\rtimes\G_{\rm m})$-action on $X$; see \cite[Lem.\ 2.2]{FZ-LND}.
Notice that, if $k=l$, then  $\partial=\partial_k=\partial_l$, and the induced
$\G_{\rm a}$-action is horizontal due to our assumption that
$\mathcal{O}(X)^{\G_{\rm m}}\neq\mathcal{O}(X)^{\G_{\rm a}}$. Otherwise, at least
one of the indices $k$ and $l$ is different from $-1$, and again  the induced
$\G_{\rm a}$-action is horizontal, since otherwise the degree of the corresponding
locally nilpotent derivation equals $-1$, see \cite[Thm.\ 3.12]{FZ-LND}. In any case, the associate
$(\G_{\rm a}\rtimes\G_{\rm m})$-action on $X$ has an open orbit, as required in (ii).

Finally we have the equivalences  (i)$\Leftrightarrow$(ii)$\Leftrightarrow$(iii).

To show the last assertions, note that under condition (iii)
the horizontal $\G_{\rm a}$-action on $X$ is normalized by the given
$\G_{\rm m}$-action.
If the latter action were elliptic or
parabolic, then $X$ would be a toric  surface $\A^2$ or $V_{d,e}$, contrary to our assumption,
see \cite[Thms.\ 3.3 and 3.16]{FZ-LND}.
Hence the $\G_{\rm m}$-action on $X$ is hyperbolic, as claimed.

By exclusion $X$ belongs to one of the classes $(\ML_i,1)$, $i=0,1$.
Furthermore, due to  \cite[Cor.\ 3.27
and 3.30]{FZ-LND}, any hyperbolic $\G_{\rm m}$-surface
$X\in (\ML_0)\cup (\ML_1)$
is a cyclic
quotient of the
normalization of some Danielewski surface $x^ny-P(z)=0$ in $\A^3$,
where $P\in \K[z]$ and $n\ge 1$.
\end{proof}

\begin{remark}
If a surface $X\in (\ML_0)$
 is a complete intersection, then it can be realized as a hypersurface $xy-P(z)=0$ in $\A^3$,
where $P\in \K[z]$
is nonconstant,  see \cite{BML}, \cite{Dai3},  \cite{DK}. In particular, $X$ is a hyperbolic $\G_{\rm m}$-surface.
\end{remark}

\subsection{$\G_{\rm m}$-surfaces: Dolgachev-Pinkham-Demazure presentation} \label{ss: DPD}
The $\G_{\rm m}$-surfaces can be described
in terms of their {\em Dolgachev-Pinkham-Demazure presentation}, or {\em DPD presentation}, for short.
Let us recall this description, see \cite{FZ0}.

\begin{definition}\label{def: DPD-construction} In the elliptic and the parabolic cases, the DPD construction associates
to any pair $(C,D)$, where $C$ is a smooth curve and $D$ is an ample $\QQ$-divisor on $C$,
the graded $\K$-algebra $$A=\bigoplus_{j\ge 0} A_j,\quad\mbox{where}\quad A_j
=H^0(C, \mathcal{O}_C(\lfloor jD\rfloor ))\,.$$ The induced effective $\G_{\rm m}$-action
on the normal affine surface $X=\spec A$ is elliptic if $C$ is projective and parabolic otherwise.
In the hyperbolic case, the DPD construction associates to any triple $(C,D_+,D_-)$,
where $C$ is a smooth affine curve and $D_\pm$ are $\QQ$-divisors on $C$ with
$D_++D_-\le 0$, the graded $\K$-algebra $$A=\bigoplus_{j\in\ZZ} A_j,\quad\mbox{where}
\quad A_{\pm j}=H^0(C, \mathcal{O}_C(\lfloor jD_\pm\rfloor ))\, \mbox{ for } j\ge 0.$$
The resulting effective $\G_{\rm m}$-action on $X=\spec A$ is hyperbolic.
In fact,
any normal affine $\G_{\rm m}$-surface
$X$ arises in this way, and the corresponding DPD presentation is unique up to
isomorphisms of pairs $(C,D)$ and of triples $(C,D_+,D_-)$ and up to replacing $D$ (the pair $(D_+,D_-)$, respectively) 
by a linearly equivalent divisor $\hat D$ (by a pair $(D_++D',D_--D')$, where $D'$ is a principal divisor on $C$, respectively), see  \cite[Thms.\ 2.2, 3.2, and 4.3]{FZ0}.
\end{definition}

The classification of the normal affine $\G_{\rm m}$-surfaces according
to the Makar-Limanov complexity is as follows, see \cite{FZ-LND} and \cite[Cor.\ 3.30]{Li}.
We let $\{D\}$ be the fractional part of a $\QQ$-divisor $D$.

\begin{proposition} \label{prop: DPD}
Let $X$ be a normal affine $\G_{\rm m}$-surface with an associate DPD presentation $(C,D)$ for an elliptic or a parabolic $\G_{\rm m}$-action, 
and $(C,D_+,D_-)$ for a hyperbolic one.
Then
\begin{itemize}\item $X\in (\ML_0)$ if and only if one of the following holds:
\begin{itemize}\item $X$ is elliptic,
$C=\PP^1$, and $\{D\}$ is supported in at most two points; \item $X$ is parabolic,
$C=\A^1$, and $\{D\}$ is supported in at most one point; \item $X$ is hyperbolic,
$C=\A^1$, and $\{D_\pm\}$ is supported in at most one point $p_\pm$.\end{itemize}

\item $X\in (\ML_1)$ if and only if one of the following holds:
 \begin{itemize}\item $X$ is parabolic,
and either $C$ is
non-rational, or $\{D\}$ is supported in at least two points; \item $X$ is hyperbolic,
$C=\A^1$, and exactly one of the $\QQ$-divisors $\{D_+\},\,\{D_-\}$ is supported in at most one point.
\end{itemize}

\item $X\in (\ML_2)$ otherwise.
\end{itemize}
\end{proposition}

\begin{remarks}\label{rem: DPD-cases}
1. The elliptic and the parabolic $\G_{\rm m}$-surfaces of class
$(\ML_0)$ are exactly
the nondegenerate toric surfaces\footnote{Recall that a  toric affine variety $X$ is {\em nondegenerate} if any invertible function on $X$ is constant.} $\A^2$ and $V_{d,e}$,
see \cite[Thms.\ 3.3 and 3.16]{FZ-LND}. A hyperbolic $\G_{\rm m}$-surface $X$ of class $(\ML_0)$ is a nondegenerate toric surface if and only if 
$\supp \{D_+\}=\supp \{D_-\}=\{p_0\}$ for some point $p_0\in C=\A^1$, and $(D_+,D_-)=(D_0+\{D_+\},\,-D_0+\{D_-\})$ 
for some integral divisor $D_0$ on $\A^1$ (\cite[Lem.\ 4.2(b)]{FKZ-completions}).

2. There exist smooth surfaces of class $(\ML_0, 0)$, that is, smooth Gizatullin surfaces which do not admit any nontrivial $\G_{\rm m}$-action, 
see \cite[Cor.\ 4.9]{FKZ-completions}.
The subgroup $\SAut X\subset\Aut X$ of such a surface acts on $X$ with an open orbit, while there is no algebraic group action on $X$ with an open orbit.

3.
If $X\in (\ML_1)$ is a parabolic $\G_{\rm m}$-surface, then the
$\G_{\rm a}$-action on  $X$ is {\em vertical}
(or {\em fiberwise}), that is,
$\mathcal{O}(X)^{\G_{\rm m}}=\mathcal{O}(X)^{\G_{\rm a}}$. This follows from \cite[Thm.\ 3.16]{FZ-LND}, cf.\ Remark 1 above.

4. If $X\in (\ML_2)$ is a $\G_{\rm m}$-surface, then $\Aut^\circ X\cong \G_{\rm m}$ or $\G^2_{\rm m}$, see Proposition \ref{prop-Iitaka}.
\end{remarks}

In terms of the DPD presentation,
the criterion of Proposition \ref{prop: dense-orbit} becomes more concrete.

\begin{corollary} Let $X$ be a normal affine surface.
The group $\Aut X$ acts on $X$
with an open orbit if and only if either $X$ is one of the surfaces in {\rm (\ref{eq: classical})}, or $X$
is
of class $(\ML_0)$, or, finally, $X$ is a hyperbolic $\G_{\rm m}$-surface
of class $(\ML_1)$.
\end{corollary}

\begin{proof} Suppose that $X$ does not appear in (\ref{eq: classical}). By Proposition \ref{prop: dense-orbit}
the group $\Aut X$ acts on $X$ with an open orbit if and only if
a semi-direct product $\G_{\rm a}\rtimes\G_{\rm m}$ does, and so, $X$ is
a $\G_{\rm m}$-surface
of class $(\ML_0)\cup (\ML_1)$. It remains to note that the ($\Aut X$)-action
on a parabolic $\G_{\rm m}$-surface of class $(\ML_1)$ has one-dimensional
orbits, see Remark \ref{rem: DPD-cases}.2.
\end{proof}

\section{Automorphism groups of Gizatullin surfaces} \label{sec: Giz}

\subsection{Definition, characterizations, examples}\label{ss: def-Giz}

We adopt the following definition.

\begin{definition}\label{def: giz} A {\em Gizatullin surface} is a normal affine surface $X$
non-isomorphic to $\A^1\times\A^1_*$ that can be completed by
a chain  of smooth rational curves (a {\em zigzag}) into an SNC-pair $(\bar X, D)$.
\end{definition}

The following characterization goes back to Gizatullin \cite{Gi1}; see also \cite[Thm.\ 1.8]{Be},  \cite{Du}.

\begin{theorem}\label{thm: Giz} Given a  normal affine surface $X$, the following are equivalent:
\begin{itemize}
\item $X$ is a Gizatullin surface;
\item $X$ is of class $(\ML_0)$;
\item $X$ admits two
distinct $\A^1$-fibrations $X\to\A^1$;
\item the group $\SAut X$ acts on $X$ with an open orbit.
\end{itemize}
\end{theorem}

The affine plane $\A^2$ and the toric affine surfaces $\A^2/\mu_{d,e}$ are examples of Gizatullin surfaces.
Another important  examples are the {\em Danilov-Gizatullin surfaces} and the {\em special Gizatullin surfaces}.
Let us consider these classes along with their DPD presentations.

\begin{example}[\emph{Danilov-Gizatullin surfaces}]\label{sit: ex-DG}
Such a surface is the
complement
$X=\mathbb{F}_n\setminus S$ to an ample section $S$ in a Hirzebruch surface
$\mathbb{F}_n=\PP(\mathcal{O}_{\PP^1}\oplus\mathcal{O}_{\PP^1}(n))\to\PP^1$.
A section $S$ is ample if and only if $d:=S^2>n$.
Two Danilov-Gizatullin surfaces  are isomorphic if and only if they share the same invariant $d=S^2$,
see \cite[Thm.\ 5.8.1]{DG2} (see also \cite[Cor.\ 4.8]{CNR},  \cite{FKZ-DG}). We let $X_d$ denote the Danilov-Gizatullin surface with invariant $d=S^2$.
The surface $X_d$ possesses exactly $d-1$
pairwise non-conjugated $\G_{\rm m}$-actions with the DPD presentations
$$(C,D_+,D_-)=\left(\A^1, -\frac{1}{r}
[p_0],\, -\frac{1}{d - r}
[p_1]\right),\quad r=1,\ldots,d-1 \,,$$ where $p_0, p_1\in\A^1$, $p_0\neq p_1$.
see \cite[5.3]{FZ-LND} and \cite[Prop.\ 5.15]{FKZ-completions}. The automorphism group of a Danilov-Gizatullin surface $X_d$ with $2\le d\le 5$ is an
amalgam, see \cite[\S\S 5-8]{DG2}.
\end{example}

\begin{example}[\emph{Special Gizatullin surfaces}]\label{sit: special-Giz}
A smooth Gizatullin surface $X$ equipped
with a hyperbolic $\G_{\rm m}$-action is called {\em special} if
the associate DPD presentation is
$$ \left(C, D_+,D_-\right)=
\left(\A^1, \,-\frac{1}{r}[p_+]\,\,, -\frac{1}{d-r}[p_-]-D_0\right)
\,$$ with $d\ge 3$,  $1\le r\le d-1$, under the convention
that $D_+=0$ if $r=1$ and $D_- =0$ if $r=d-1$, and otherwise $p_+\ne p_-$,  
and with a reduced divisor
$D_0=\sum_{i=1}^s [p_i]$ on $\A^1$, where $s>0$ and $p_i\neq p_\pm$ $\forall i$.
\end{example}

\subsection{One-parameter subgroups and bearability on Gizatullin surfaces}\label{ss: 1-parameter sbgrps} 
For Gizatullin surfaces the following theorem is proven in \cite[Thms.\ 1.0.1, 1.0.5,
and Ex.\ 6.3.21]{FKZ-non-uniqueness}, see also \cite[\S 5.3]{FKZ-completions} and \cite[Cor.~5.15]{FKZ-uniqueness}.

\begin{theorem}\label{thm: AG} Smooth affine surfaces $X$
admitting an effective $\G_{\rm m}$-action can be divided into the following 4 classes:
\begin{itemize}
\item[1)] toric affine surfaces;
\item[2)] Danilov-Gizatullin surfaces $X_d$, $d\ge 4$;
\item[3)] special Gizatullin surfaces;
\item[4)] all the others,
\end{itemize}
so that the set of conjugacy classes of 1-tori in $\Aut X$

\begin{itemize}\item is infinite countable in case {\rm 1)}\footnote{The latter holds for any
toric affine surface $X$, not necessarily smooth.};
\item is finite of cardinality $\lfloor d/2 \rfloor$ for $X_d$  in case {\rm 2)}; 
\item forms a 1- or 2-parameter family  in case {\rm 3)};
 \item is finite of cardinality at most 2 in case {\rm 4)}.
\end{itemize}
Furthermore, the set of equivalence classes of $\A^1$-fibrations $X\to\A^1$
\begin{itemize}\item  is finite of cardinality at most 2 in cases {\rm 1)} and {\rm 4)}; 
\item forms an $m$-parameter family  in case {\rm 2)}, where $m\ge 1$ if $d\ge 7$, and
$m=m(d)\to+\infty$ as $d\to+\infty$;
\item forms an $m$-parameter family  in case {\rm 3)}, where $m\ge 1$.\end{itemize}
\end{theorem}

\begin{proof} For surfaces of class $\ML_0$, that is, for Gizatullin surfaces, the assertions follow due to the references preceding the theorem. 
Thus we need to consider just the $\ML_1$- and $\ML_2$-surfaces. By definition, such a surface $X$ admits at most one $\A^1$-fibration over an affine curve. 
By Corollary \ref{cor: neutral-comp} the group $\Aut^{\0} X$ is nested, that is, 1-bearable, and has at most one conjugacy class of maximal tori. 
Hence the assertions follow also in this case. \end{proof}

\begin{corollary}\label{cor: unbearable}  Suppose that the base field $\K$ is uncountable. 
If $X$ is either a special Gizatullin surface, or a Danilov-Gizatullin  surface $X_d$ with $d\ge 7$, then
 the group $\Aut^{\0} X$  is not countably bearable. Furthermore, $Aut X_d$ ($\Aut^{\0} X_d$, respectively) is not a nested ind-group for $d=4,5$, 
and  cannot be a nontrivial amalgam of two nested ind-groups for $d=6$. 
\end{corollary}

\begin{proof} By Theorem  \ref{thm: AG}, for a special Gizatullin surface $X$ (for a Danilov-Gizatullin  surface $X_d$ with $d\ge 7$, respectively)
the set of conjugacy classes of $\G_{\rm m}$-subgroups (of $\A^1$-fibrations over $\A^1$, respectively) in $\Aut^{\0} X$ is uncountable.
In these cases the assertion follows from Propositions \ref{lem: mult-classes} and \ref{lem: add-classes}, respectively.  For $d\le 6$, the groups $\Aut X_d$ and $
\Aut^\0 X_d$ have both $\lfloor d/2 \rfloor$ conjugacy classes of $\G_{\rm m}$-subgroups, that is, of maximal tori. Indeed, $\rk\Aut X_d =1$, since  this surface is not toric. 
Hence  for $d\le 6$ the assertion follows from Remark \ref{rem: Dani}.
\end{proof}

\begin{remarks} 
1. The presentation of the group $\Aut X_d$ as an amalgam in \cite[\S\S 6-8]{DG2} involves two factors if $d=3$ and three factors if $d=4,5$. 
It seems that for $d= 6$, no explicit amalgam structure on $\Aut X_d$ is known. 
The authors of \cite{DG2} mention  that their methods allow in principle to compute the group $\Aut X_d$  for any $d$; cf., however, Corollary \ref{cor: unbearable}.

2. See \cite{BD2} and \cite{Ko2} for spectacular examples of Gizatullin surfaces $X$ such that, if $N(X)\subset\Aut X$
is the normal subgroup
generated by all algebraic subgroups of $\Aut X$, then the quotient  $(\Aut X)/N(X)$
contains a free group
on an uncountable set of generators.
  \end{remarks}

\subsection{Standard completions  and extended divisors of Gizatullin surfaces}\label{ss: standard-completion} These combinatorial invariants are indispensable in
studies on Gizatullin surfaces.

\begin{notation}\label{nota: zigzag} {\rm Let $X$ be a Gizatullin surface, and let $(\bar X, D)$ be an SNC completion of the minimal resolution of singularities of $X$ 
\footnote{Also called a \emph{resolved completion}.} by a zigzag $D$, where
$$D = C_0 + \cdots + C_n \quad \text{with} \quad C_i\cdot C_j \ =1\quad \mbox{if}\quad |i - j| = 1\quad\mbox{and}\quad C_i\cdot C_j \ =0\quad\mbox{otherwise}\,.$$
The boundary components $C_i$, $0 \leq i \leq n$, serve as the vertices $v_i$ of the dual linear graph $\Gamma_D$ of $D$. 
Each vertex $v_i$ is weighted by the corresponding self-intersection number $w_i = C_i^2$. Thus $\Gamma_D$ is of the form
$$\Gamma_D: \quad \cou{v_0}{w_0} \lin \cou{v_1}{w_1} \lin \ \cdots \ \lin \cou{v_n}{w_n} \quad .$$}
\end{notation}

\noindent The string of weights $[[w_0,w_1,\ldots,w_n]]$ can be putted into a standard form by means of elementary transformations of weighted graphs.

\begin{definition}\label{def:elem-tr} Given an at most linear vertex $v$ of a weighted graph $\Gamma$ with weight $0$ one can perform the following transformations. 
If $v$ is linear with neighbors $v_1, v_2$ then we blow up the edge connecting $v$ and $v_1$ in $\Gamma$ and blow down the proper transform of $v$:

\begin{equation}\label{ElementaryTransformation1}
\dots \ \cou{v_1}{w_1 - 1} \lin \cou{v'}{0} \lin \cou{v_2}{w_2 + 1} \ \dots \ \dasharrow \ \dots \ \cou{v_1}{w_1 - 1} \lin \cou{v'}{-1} \lin \cou{v}{-1} \lin \cou{v_2}{w_2} \ \dots \ \to \ \dots \ \cou{v_1}{w_1} \lin \cou{v}{0} \lin \cou{v_2}{w_2} \quad .
\end{equation}

\noindent Similarly, if $v$ is an end vertex of $\Gamma$ connected to the vertex $v_1$ then one proceeds as follows:

\begin{equation}\label{ElementaryTransformation2}
\dots \ \cou{v_1}{w_1 - 1} \lin \cou{v'}{0} \ \dasharrow \ \dots \ \cou{v_1}{w_1 - 1} \lin \cou{v'}{-1} \lin \cou{v}{-1} \ \to \ \dots \ \cou{v_1}{w_1} \lin \cou{v}{0} \quad .
\end{equation}

\noindent These operations (\ref{ElementaryTransformation1}) and (\ref{ElementaryTransformation2}) and their inverses are called \emph{elementary transformations} of $\Gamma$. 
If such an elementary transformation involves only an inner blowup then we call it \emph{inner}. 
Thus (\ref{ElementaryTransformation1}) and (\ref{ElementaryTransformation2}) are inner whereas the inverse of (\ref{ElementaryTransformation2}) is not
 as it involves an outer blowup.
\end{definition}

Consider a Gizatullin surface $X$ along with a resolved SNC completion $(\bar X, D)$, where $\bar X$ is a smooth projective surface and $D\subset \bar X$ is a zigzag. 
By a sequence of blowups and blowdowns one can transform the dual graph $\Gamma_D$ into a \textit{standard form}, 
where $C_0^2 = C_1^2 = 0$ and $C_i^2 \leq -2$ for all $i \geq 2$ if $n \geq 4$ or $C_i^2 = 0$ for all $i$ if $n \leq 3$ (cf.\ \cite{Dai2}, \cite{DG},  \cite{FKZ-graphs}). 
Moreover, this representation is unique up to reversion. 
The latter  means that for two standard forms $[[0, 0, w_2, \dots, w_n]]$ and $[[0, 0, w'_2, \dots, w'_n]]$ of $\Gamma_D$, either $w_i = w'_i$ or $w_i = w'_{n + 2 - i}$ holds (\cite{FKZ1C}).

The reversion process can be described as follows. Start with a boundary divisor of type $[[0, 0, w_2, \dots, w_n]]$. 
Performing the elementary transformation (\ref{ElementaryTransformation1}) at the vertex corresponding to $C_1$ 
one gets a boundary divisor of type $[[-1, 0, w_2 + 1, w_3, \dots, w_n]]$. After $|w_2|$ steps one arrives at a boundary divisor of type
$[[w_2, 0, 0, w_3, \dots, w_n]]$. Thus, one can move pairs of zeros to the right. 
Repeating this, one obtains  finally a boundary divisor of type $$[[w_2, \dots, w_n, 0, 0]]=[[0, 0, w_n, \dots, w_2]]\,.$$ 
Notice that all the birational transformations involved are centered at the boundary (so to say, they yield the identity on the affine parts).

Let us recall the notion of an $m$-standard zigzag (see \cite[(1.2)]{DG}).

\begin{definition}\label{def: zigzag} A zigzag $D$ of type $[[0, -m, w_2, \dots, w_n]]$ with $n \geq 1$ and $w_i \leq -2$
is called \emph{$m$-standard}
(in the case $n = 1$ there is no weight $w_i$).\\
An \emph{$m$-standard pair} is a pair $(\bar X, D)$ consisting of a smooth projective surface $\bar X$ and an $m$-standard zigzag $D$ on $\bar X$. 
If $m = 0$, then $(\bar X, D)$ is called a \emph{standard pair}. 
A \emph{birational map} $\varphi: (\bar X, D) \dasharrow (\bar X', D')$ between $m$-standard pairs is a birational map $\varphi: \bar X \dasharrow \bar X'$ 
which restricts to an isomorphism $\varphi\vert_{\bar X \backslash D}: \bar X \backslash D \stackrel{\sim}{\to} \bar X' \backslash D'$. 
A reversion of an $m$-standard pair starts by reducing it to a 0-standard one by means of $m$ (non-inner) elementary transformations at the component of zero weight. 
After reversion of the resulting 0-standard pair, one returns again at an $m$-standard pair by performing $m$ elementary transformations at an extremal $0$-component. 
This requires outer blowups centered at an arbitrary point of a $0$-component   (cf. Remark~\ref{EquivalenceOfArrows}). 
\end{definition}

\begin{examples} \label{exa: special-standard} 1. The Danilov-Gizatullin surface $X_d$ (see \ref{sit: ex-DG}) has a boundary zigzag of type $[[d]]$ 
with the standard form $[[0, 0, (-2)_{d - 1}]]$ (the index $d - 1$ means that there are $d - 1$ consecutive components with self-intersection index $-2$).
Any smooth affine surface $X$  completable by a standard zigzag $[[0, 0, (-2)_{d - 1}]],$ $d\ge 2\,,$ and non-isomorphic to $\PP^2\setminus C$, 
where $C$ is a smooth conic, is isomorphic to the Danilov-Gizatullin surface $X_d$.

2.
For  a special Gizatullin surface (see \ref{sit: special-Giz}) the standard zigzag is
$$[[0, 0, -2,\ldots,-2, -w_s,-2,\ldots,-2]],\quad\mbox{where}\quad w_s<-2\,.$$ However,  a Gizatullin surface
with such a sequence of weights
does not need to be special.
\end{examples}

\begin{definition}[\emph{extended divisor}]\label{sit: standard-fibration} 
Since the underlying smooth projective surface $\bar X$ of a $0$-standard pair is rational and $C_0^2=C_1^2=0$, it is equipped with rational fibrations 
$\Phi_i = \Phi_{|C_i|}: \bar X \to \p^1$ defined by the complete linear systems $|C_i|$ on $\bar X$, $i=0,1$, respectively.  
This defines a  birational morphism (\cite[Lem.\ 2.19]{FKZ-completions})
$$\Phi = \Phi_0 \times \Phi_1: \bar X \to \p^1 \times \p^1\,.$$
 After a suitable coordinate change one may suppose that 
$C_0 = \Phi_0^{-1}(\infty)$, $\Phi(C_1) = \p^1 \times \{ \infty \}$, and $C_2 \cup \cdots \cup C_n \subseteq \Phi_0^{-1}(0)$. 
\footnote{For a map $\phi\colon A\to B$, the notation $\phi^{-1}(b)$ usually stands for the set theoretical preimage.}
The reduced effective divisor $D_{\text{ext}} := C_0 \cup C_1 \cup \Phi_0^{-1}(0)$ is called the \textit{extended divisor}.
\end{definition}

In order to determine the structure of the extended divisor, let us recall the notion of a \emph{feather}  (\cite[Def.\ 5.5]{FKZ-completions}).

\begin{definition}[\emph{feathers}]\label{def: feather}
\begin{itemize}
\item[(1)] A \emph{feather} is a linear chain

$$F: \quad \co{B}{} \lin \co{F_1} \lin \dots \lin \co{F_s}$$

\noindent of smooth rational curves such that $B^2 \leq -1$ and $F_i^2 \leq -2$ for all $i \geq 1$. The curve $B$ is called the \emph{bridge curve}.
\item[(2)] A \emph{collection of feathers} $\{ F_\rho \}$ consists of pairwise disjoint feathers $F_\rho$, $\rho=1,\ldots, r$. Such a collection will be denoted by a plus box

$$\xboxo{ \{ F_\rho \} } \quad .$$

\item[(3)] Let $D = C_0 + \cdots + C_n$ be a zigzag. A collection $\{ F_\rho \}$ is \emph{attached to a curve} $C_i$ if the bridge curves $B_\rho$ of the feathers $ F_\rho$ 
meet $C_i$ in pairwise distinct points and the feathers $F_\rho$ are disjoint with the curves $C_j$ for $j \neq i$.
\end{itemize}
\end{definition}

\begin{lemma}\label{StructureExtendedDivisor} (\cite[Prop.\ 1.11]{FKZ-uniqueness}) 
Let $(\tilde{X}, D)$ be a minimal SNC completion of the minimal resolution of singularities of a Gizatullin surface $X$. 
Furthermore, let $D = C_0 + \cdots + C_n$ be the boundary divisor in standard form. Then the extended divisor $D_{\text{ext}}$ has the dual graph

\vspace{15pt}
\begin{equation}\label{diag: graph-Gizatullin} 
\Gamma_{\text{ext}}: \quad \cou{0}{C_0} \lin \cou{0}{C_1} \lin \cu{C_2} \nlin \xbshiftup{ \{ F_{2, j} \} }{} \lin \dots \lin \cu{C_i} \nlin \xbshiftup{ \{ F_{i, j} \} }{} \lin \dots \lin \cu{C_n} \nlin \xbshiftup{ \{ F_{n, j} \} }{} \quad ,
\end{equation}

\noindent where $\{ F_{i, j} \}$, $j \in \{ 1, \dots, r_i \}$, are feathers attached to the curve $C_i$. 
Moreover, $\tilde{X}$ is obtained from $\p^1 \times \p^1$ by a sequence of blowups with centers in the images of the components $C_i$, $i \geq 2$.
\end{lemma}

\begin{remark}\label{rem: sing} 
Consider the feathers $F_{i, j} := B_{i, j} + F_{i, j, 1} + \cdots + F_{i, j, k_{i, j}}$ mentioned in Lemma \ref{StructureExtendedDivisor}. 
The collection of linear chains $R_{i, j} := F_{i, j, 1} + \cdots + F_{i, j, k_{i, j}}$ corresponds to the minimal resolution of singularities of $X$. 
Thus, if $(\bar X, D)$ is a standard completion of $X$ and $(\tilde{X}, D)$ is the minimal resolution of singularities of $(\bar X, D)$, 
then the chain $R_{i, j}$ contracts via $\mu: (\tilde{X}, D) \to (\bar X, D)$ to a singular point of $X$, which is a cyclic quotient singularity. 
In particular, $X$ has at most cyclic quotient singularities (\cite[\S 3, Lem.\ 1.4.4(1)]{Mi} and \cite[Rem.\ 1.12]{FKZ-uniqueness}).

Hence $X$ is smooth if and only if every $R_{i, j}$ is empty, \ie, if every feather $F_{i, j}$ is irreducible 
and reduces to a single bridge curve $B_{i, j}$ (\cite[1.8, 1.9 and Rem.\ 1.12]{FKZ-uniqueness}).
\end{remark}

Let us introduce the notions of a \emph{$*$-component} and a  \emph{$+$-component}.

\begin{definition}\label{StarComponentGeneral}
\begin{itemize}
\item[(1)] For a general feather $F$ with dual graph
$$\Gamma_F: \quad \cu{B} \lin \cu{D_1} \lin \dots \lin \cu{D_k}$$
\medspace

\noindent and bridge curve $B$ we call $D_k$ the \emph{tip component of} $F$.
\item[(2)] The component $C_i$ is called a \emph{$*$-component} if
\begin{itemize}
\item[(i)] $D^{\geq i + 1}_{\text{ext}}$ is not contractible and
\item[(ii)] $D^{\geq i + 1}_{\text{ext}} - F_{j, k}$ is not contractible for every feather $F_{j, k}$ of $D^{\geq i + 1}_{\text{ext}}$ such that 
the tip component of $F_{j, k}$ has \emph{mother component} $C_\tau$, that is, the component $C_\tau$ with $\tau < i$ carrying the center of  blowup in which 
the tip component of $F_{j, k}$ is born.
\end{itemize}
Otherwise $C_i$ is called a \emph{$+$-component}.
\end{itemize}
\end{definition}

\begin{lemma} Let $D_{\text{ext}}$ be the extended divisor of the minimal resolution of singularities of a $1$-standard completion of a Gizatullin surface $X$. 
Suppose that every $C_i$, $3 \leq i \leq n - 1$, is a $*$-component and that there is no feather attached to the component $C_n$. 
Then every feather $F_{i, j}$ is an $A_k$-feather, that is, every $F_{i, j}$ is contractible and therefore has the dual graph
$$\Gamma_{F_{i, j}}: \quad \cou{-1}{B} \lin \cou{-2}{D_1} \lin \dots \lin \cou{-2}{D_k} \quad ,$$

with $k$ depending on $i$ and $j$.
\end{lemma}

Note that for an $A_k$-feather the mother components of all curves $D_1, \dots, D_k$ coincide, 
since any $A_k$-feather is born by successive blowups of a point on the boundary component it is attached to.

\begin{examples}\label{ex: DG-extended}
1. A Gizatullin surface $X$ is isomorphic to a nondegenerate toric surface $V_{d,e}=\A^2/\mu_{d,e}$ if and only if for some (and then also for any) 
resolved standard completion $(\tilde X, D)$ of $X$ the dual graph $\Gamma_{\text{ext}}$ of the associated extended divisor $D_{\text{ext}}$ is a linear chain 
(\cite[Lem.\ 2.20]{FKZ-completions}).

2.
Given a Danilov-Gizatullin surface $X_d$ with $d \neq 4$, there are only $d - 1$ possible associated extended divisors (which do not depend, up to an isomorphism, 
on any further continuous parameter), with the dual graphs
\vspace{15pt}
$$\Gamma_{\text{ext}}: \quad \cu{C_0} \lin \cu{C_1} \lin \cu{C_2} \lin \cu{C_3} \lin \dots \lin \cu{C_{r - 1}} \lin \cu{C_r} \nlin \cshiftup{1 - r}{} \lin \cu{C_{r + 1}} \lin \dots \lin \cu{C_{d - 2}} \lin \cu{C_{d - 1}} \nlin \cshiftup{-1}{} \quad ,$$
\noindent where $2 \leq r \leq d - 1$. In addition, for $d = 4$, there is another extended divisor possible; 
the corresponding affine surface is called an \emph{affine pseudo-plane}. Its dual graph is
\vspace{15pt}
$$\Gamma_{\text{ext}}: \quad \cu{C_0} \lin \cu{C_1} \lin \cu{C_2} \lin \cu{C_3} \nlin \cshiftup{-1}{} \lin \cu{C_4} \quad .$$
\end{examples}

\subsection{Associated graph of groups}\label{ss: graph-of-grps}
Following \cite{DG} and \cite{BD1}, for an $\a^1$-fibered surface $X$  we introduce a (not necessarily finite) graph $\F_X$, which reflects the structure of the group  $\Aut X$.

\begin{definition} To any $\a^1$-fibered smooth
affine surface $\mu\colon X\to\A^1$ one associates the oriented graph $\F_X$ as follows:
\begin{itemize}
\item A vertex of $\F_X$ is an equivalence class of a $1$-standard pair $(\bar X, D)$ such that $\bar X \backslash D \cong X$, where two $1$-standard pairs 
$(\bar X_1, D_1, \bar{\mu}_1)$ and $(\bar X_2, D_2, \bar{\mu}_2)$ define the same vertex if and only if 
$(\bar X_1 \backslash D_1, \mu_1) \cong (\bar X_2 \backslash D_2, \mu_2)$.
\item An arrow of $\F_X$ is an equivalence class of reversions. 
If $\varphi: (\bar X, D) \to (\bar X', D')$ is a reversion, then the class $[\varphi]$ of $\varphi$ is an arrow starting from $[(\bar X, D)]$ and ending at $[(\bar X', D')]$. 
Two reversions $\varphi_1: (\bar X_1, D_1) \dasharrow (\bar X'_1, D'_1)$ and $\varphi_2: (\bar X_2, D_2) \dasharrow (\bar X'_2, D'_2)$ define the same arrow 
if and only if there exist isomorphisms $\theta: (\bar X_1, D_1) \to (\bar X_2, D_2)$ and $\theta': (\bar X'_1, D'_1) \to (\bar X'_2, D'_2)$, 
such that $\varphi_2 \circ \theta = \theta' \circ \varphi_1$. Given an arrow $\alpha$, 
we denote by $s(\alpha)$ and $t(\alpha)$, respectively, the starting and ending vertices of $\alpha$.
\end{itemize}
\end{definition}

\begin{remark}\label{EquivalenceOfArrows} 
It follows from the definition that for a $1$-standard pair $(\bar X, D)$ two reversions 
$\varphi_1: (\bar X, D) \dasharrow (\bar X_1, D_1)$ and $\varphi_2: (\bar X, D) \dasharrow (\bar X_2, D_2)$ 
centered at the points $p_1$ and $p_2$ define the same arrow if and only if there exists an automorphism $\psi \in \Aut (\bar X, D)$ such that $\psi(p_1) = p_2$, 
see Definition~\ref{def: zigzag}.
\end{remark}

The structure of the graph $\F_X$ allows to decide, whether the automorphism group $\Aut X$ is generated by automorphisms of $\a^1$-fibrations. 
One says that $\varphi \in \Aut X$ is an \textit{automorphism of $\a^1$-fibrations} if there exists an $\a^1$-fibration $\mu: X \to \a^1$ such that 
$\varphi$ induces an isomorphism $\varphi: (X, \mu) \stackrel{\cong}{\longrightarrow} (X, \mu)$. Indeed, we have the following important fact.

\begin{theorem}\label{thm: BD} {\rm (\cite[Prop.\ 4.0.7]{BD1})}
Suppose that $D$ has a component $C_i$ with  $C_i^2\le-3$. Then $\Aut X$ is generated by automorphisms of $\a^1$-fibrations if and only if $\F_X$ is a tree.
Furthermore, there is an exact sequence
$$1 \longrightarrow H \longrightarrow \Aut X \longrightarrow \pi_1(\F_X) \longrightarrow 1\,,$$
where $H$ is the normal subgroup of $\Aut X$ generated by the automorphisms of $\A^1$-fibrations and $\pi_1(\F_X)$ is the fundamental group of the graph $\F_X$.
\end{theorem}

\begin{remark} Due to Corollary \ref{cor: nested} below, each of the automorphism groups of $\A^1$-fibrations which generate $H$ is  
an extension of a  metabelian connected nested ind-group of rank $\le 2$ by an at most countable group. 
The same concerns the factors of the amalgams considered in the next subsection. 
\end{remark}

One can equip $\F_X$ with a structure of a graph of groups as follows.

\begin{definition}\label{DfGraph} 
Let $X$ be a normal quasi-projective surface, and let $\F_X$ be its associated graph. 
Then $\F_X$ admits a structure $(\mathcal{G}_X,\F_X)$ of a graph of groups by the following choice:
\begin{itemize}
\item For any vertex $v$ of $\F_X$, fix a $1$-standard pair $(\bar X_v, D_v, \bar{\mu}_v)$ in the class $v$. 
The group $G_v$ is equal to $\Aut(\bar X_v \backslash D_v, \mu_v)$.
\item For any arrow $\sigma$ of $\F_X$, fix a reversion 
$r_\sigma: (\bar X_\sigma, D_\sigma, \bar{\mu}_\sigma) \dasharrow (\bar X'_\sigma, D'_\sigma, \bar{\mu'}_\sigma)$
 in the class of $\sigma$ and also an isomorphism 
$\phi_\sigma: (\bar X'_\sigma \backslash D'_\sigma, \mu'_\sigma) \to (\bar X_{t(\sigma)} \backslash D_{t(\sigma)}, \mu_{t(\sigma)})$. 
Then the group $G_\sigma$ is equal to
$$\{ (\varphi, \varphi') \in \Aut(\bar X_\sigma, D_\sigma) \times \Aut(\bar X'_\sigma, D'_\sigma) \mid r_\sigma \circ \varphi = \varphi' \circ r_\sigma \}$$
and the monomorphisms $\kappa_\sigma: G_\sigma \to G_{s(\sigma)}$ and $\lambda_\sigma: G_\sigma \to G_{t(\sigma)}$ are given by 
$\kappa_\sigma((\varphi, \varphi')) = \phi_{\sigma^{-1}} \circ \varphi \circ \phi^{-1}_{\sigma^{-1}}$ and 
$\lambda_\sigma((\varphi, \varphi')) = \phi_\sigma \circ \varphi' \circ \phi^{-1}_\sigma$.
\item A \emph{path} in the graph of groups is a sequence $(g_0, \sigma_1, g_1, \dots, \sigma_r, g_r)$, where $g_i \in G_{v_i}$ and the sequence 
$(v_0, \sigma_1, v_1, \dots, \sigma_r, v_r)$ corresponds to a path in $\F_X$. We say that the path starts at $v_1$ and ends at $v_n$, and is closed if $v_1 = v_n$.

\item The \emph{fundamental group} of a graph of groups at a vertex $v$ consists of the closed paths starting and ending at $v$, modulo the relations 
$$(\sigma, \lambda_\sigma(h), \sigma^{-1},(\kappa_\sigma(h))^{-1}) \cong (1)\quad\mbox{and}\quad (g, \sigma, 1, \sigma^{-1}, g') \cong (gg')\,,$$ 
where $1 \in G_{s(\sigma)}$.
\end{itemize}
\end{definition}

The first version of the following theorem was established by Danilov and Gizatullin (\cite[Thm.\ 5]{DG}). 
It connects the structure of the graph of groups on $\F_X$  as in Definition \ref{DfGraph} with the group $\Aut X$.

\begin{theorem}\label{AutFundamentalGroup} {\rm (\cite[Thm.\ 5]{DG}, see also \cite[Thm.\ 4.0.11]{BD1})} 
Let $(\bar X, D)$ be a $1$-standard pair such that $D$ has a component $C_i$ with  $C_i^2\leq -3$. If $ X= \bar X \backslash D$, 
then $\Aut X\cong\pi_1(\mathcal{G}_X,\F_X)$.
\end{theorem}

The following important consequence concerns the structure of the automorphism groups of Gizatullin surfaces.

\begin{corollary}\label{cor: amalgam A1-fibrations} 
Under the assumptions of Theorem \ref{AutFundamentalGroup} suppose in addition that $\F_X$ is a tree with vertices $[(\bar X_i, D_i)]$, $i \in I$. 
Then $\Aut X$ is an amalgam of the  automorphism groups $\Aut(\bar X_i \backslash D_i, \mu_i)$ of $\a^1$-fibrations over $\A^1$. 
\end{corollary}

\subsection{Amalgam structures for Gizatullin surfaces}\label{sec: amalgams}

In this section we list the Gizatullin surfaces known to the authors, where the automorphism group is an amalgam (however, see \cite{DL} for further potential examples). 
The easiest way to present such surfaces is to describe various $1$-standard completions of them in terms of the dual graphs of their extended divisors.

Although the following theorem is a special case of Theorem \ref{StructureOfAutomorphismGroup}, it is worth to be mentioned independently.

\begin{theorem}\label{thm: BD2} {\rm (\cite[Thm.\ 5.4.5]{BD1})} Consider a Danielewski surface
$$X = \{ xy - P(z) = 0 \} \subseteq \a^3, \quad\mbox{where}\quad P(z) \in \ka[z]$$ has degree $n\ge 1$. 
Then $X$ has a standard completion $(\bar X, D)$ of type $[[0, 0, -n]]$. Letting $\tau\in\Aut X$ be the involution $(x, y, z) \mapsto (y, x, z)$ and $\mu\colon X\to\A^1$ 
be the $\A^1$-fibration $(x,y,z)\mapsto x$, we let $A = \langle \Aut(\bar X, D), \tau \rangle \subseteq \Aut X$ and $J = \Aut(X, \mu)$. Then $A \cap J  = \Aut(\bar X, D)$ and
$$\Aut X = A \star_{A \cap J} J\,.$$
\end{theorem}

This result can be generalized as follows.

\begin{theorem} {\rm (\cite[Cor.\ 3.19, cf.\ Thm.\ 4.4]{Ko1})}  \label{StructureOfAutomorphismGroup}
Let $X$ be a smooth Gizatullin surface satisfying the following condition (see diagram {\rm (\ref{diag: graph-Gizatullin})}):
\begin{eqnarray*}
(*) && X \ \text{admits a} \ 1-\text{standard completion} \ (\bar X, D) \ \text{such that} \ C_3, \dots, C_{n - 1} \ \text{are} \\ && *\text{-components
and there is no feather attached to} \ C_2 \ \text{and to} \ C_n.
\end{eqnarray*}

\noindent
Fix an $\a^1$-fibration $\mu\colon X \to \a^1$, and let $\mu^\vee: X \to \a^1$ be the $\a^1$-fibration  induced by the reversion 
$\psi: (\bar X, D) \dasharrow (\bar X^\vee, D^\vee)$ with center $p \in C_0 \backslash C_1$. Then $\F_X$ has one of the following structures:
$$\F_X: \begin{xy}
  \xymatrix{
  [(\bar X, D)] \ \bullet \ar@{<->}[r] & \bullet \ [(\bar X^\vee, D^\vee)] \\
  }
\end{xy} \quad \text{or} \quad \F_X: [(\bar X, D)] \ \bullet \rcirclearrowleft .$$
\noindent If $\F_X$ is of the form $\bullet \rcirclearrowleft$, then $D^{\geq 2}$ is a palindrome.
\begin{itemize}
\item[(a)] Let $\F_X$ be of the form $\bullet \rcirclearrowleft$, that is, $(\bar X, D) \cong (\bar X^\vee, D^\vee)$. Then
$$\Aut X = A \star_{A \cap J} J\,,$$
where $A = \langle \Aut (\bar X, D), \psi \rangle$, $\,\,J = \Aut(X, \mu)$, and $A \cap J = \Aut(\bar X, D)$.
\item[(b)] Let $\F_X$ be of the form 
$\begin{xy} \xymatrix{ [(\bar X, D)] \ \bullet \ar@{<->}[r] & \bullet \ [(\bar X^\vee, D^\vee)] } \end{xy}$.
Denote by $A$ the subgroup corresponding to the edge and by $J$ and $J^\vee$ the subgroups $J = \Aut(X, \mu)$ and $J^\vee = \Aut(X, \mu^\vee)$. 
Identifying $J \hookleftarrow A \hookrightarrow J^\vee$ we have
$$\Aut  X = J \star_A J^\vee.$$
\end{itemize}
\end{theorem}

An important particular case of Theorem \ref{StructureOfAutomorphismGroup} is that of the toric  affine surfaces with the amalgam structures 
on their automorphism groups as exposed in \ref{sit: A2}-\ref{sit: A1star-2}.

Another interesting example of a family of smooth Gizatullin surfaces, for which the automorphism groups are amalgams, is the following one. 
Consider any smooth $1$-standard pair $(\bar X, D)$ such that the dual graph of $D_{\text{ext}}$ has the following form:

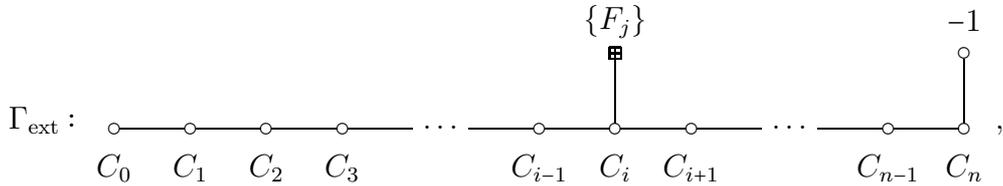
\vspace{15pt}
$$\Gamma_{\text{ext}}: \quad \cu{C_0} \lin \cu{C_1} \lin \cu{C_2} \lin \cu{C_3} \lin \dots \lin \cu{C_{i - 1}} \lin \cu{C_i} \nlin \xbshiftup{ \{ F_j \} }{} \lin \cu{C_{i + 1}} \lin \dots \lin \cu{C_{n - 1}} \lin \cu{C_n} \nlin \cshiftup{-1}{} \quad ,$$

\noindent where $C_3, \dots, C_{n - 1}$ are $*$-components. Hence any feather $F_j$ has self-intersection index $F_j^2=-1$. 
Reversion of $(\bar X, D)$ may lead to two different completions, namely those with the dual graphs of the extended divisors

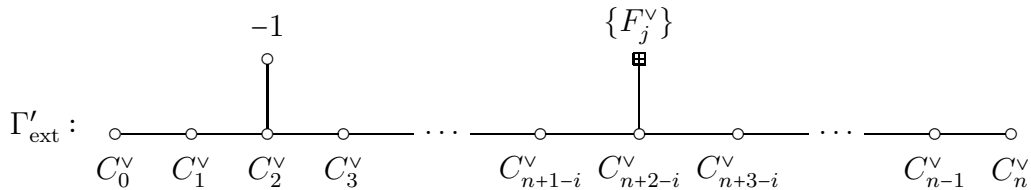
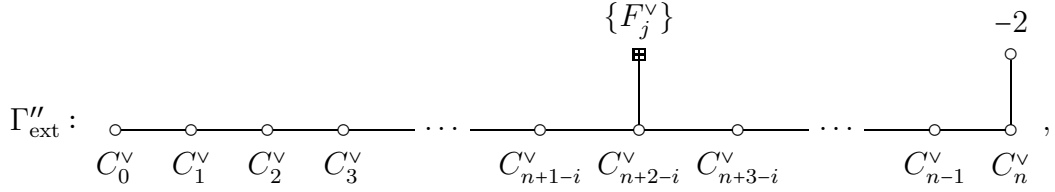
\vspace{15pt}
$$\Gamma'_{\text{ext}}: \quad \cu{C^\vee_0} \lin \cu{C^\vee_1} \lin \cu{C^\vee_2} \nlin \cshiftup{-1}{} \lin \cu{C^\vee_3} \lin \dots \lin \cu{C^\vee_{{n + 1 - i}}} \llin \cu{C^\vee_{n + 2 - i}} \nlin \xbshiftup{ \{ F^\vee_j \} }{} \llin \cu{C^\vee_{n + 3 - i}} \lin \dots \lin \cu{C^\vee_{n - 1}} \lin \cu{C^\vee_n}$$

\noindent and

\vspace{15pt}
$$\Gamma''_{\text{ext}}: \quad \cu{C^\vee_0} \lin \cu{C^\vee_1} \lin \cu{C^\vee_2} \lin \cu{C^\vee_3} \lin \dots \lin \cu{C^\vee_{{n + 1 - i}}} \llin \cu{C^\vee_{n + 2 - i}} \nlin \xbshiftup{ \{ F^\vee_j \} }{} \llin \cu{C^\vee_{n + 3 - i}} \lin \dots \lin \cu{C^\vee_{n - 1}} \lin \cu{C^\vee_n} \nlin \cshiftup{-2}{} \quad ,$$

\noindent respectively, depending on the choice of the center of reversion $\lambda \in C_0 \backslash C_1$. 
It is not difficult to see that the action of $\Aut(\bar X, D)$ on $C_0 \backslash C_1$ admits two orbits, namely an open orbit 
$(C_0 \backslash C_1) \backslash \{ p \}$ and a point $\{ p \}$ (depending on the position of the feather $G$ attached to $C_n$). 
Moreover, the actions of $\Aut(\bar X', D')$ and $\Aut(\bar X'', D'')$, respectively, on $C'_0 \backslash C'_1$ and $C''_0 \backslash C''_1$, respectively, are transitive. 
These observations lead to the following proposition.

\begin{proposition}\label{AutOfSurfaceWithThreeVertices}
The graph $\F_X$ associated to $X = \bar X \backslash D$ is
$$\F_X: \begin{xy}
 \xymatrix{
 [(\bar X', D')] \ \ar@{<->}[r]^{\sigma'} & \ [(\bar X, D)] \ \ar@{<->}[r]^{\sigma''} & \ [(\bar X'', D'')]\\
  }
\end{xy}.$$

\noindent Fixing arbitrary reversions $\alpha: (\bar X, D) \dasharrow (\bar X', D')$ and $\beta: (\bar X, D) \dasharrow (\bar X'', D'')$, 
it follows that $\Aut X$ is an amalgam of the groups $\Aut(X, \mu)$, $\Aut(X, \mu' \circ \alpha)$, and $\Aut(X, \mu'' \circ \beta)$, 
amalgamated over their pairwise intersections.
\end{proposition}

\begin{sit} 
The last statement requires an explanation. 
By \cite[Thm.\ 3.0.2]{BD1}, every automorphism of a Gizatullin surface admits (an essentially unique) decomposition in fibered modifications and reversions. 
Let us fix two reversions $\alpha: (\bar X, D) \dasharrow (\bar X', D')$ and $\beta: (\bar X, D) \dasharrow (\bar X'', D'')$ 
(which are, as we have seen, unique up to equivalence). 
Then every automorphism of $X$ has an (essentially unique) decomposition into maps of the form 
$\phi \in \Aut(X, \mu)$, $\alpha^{-1}\varphi'\alpha \in \Aut(X, \mu' \circ \alpha)$ with $\varphi' \in \Aut(X, \mu')$, 
and $\beta^{-1}\varphi''\beta \in \Aut(X, \mu'' \circ \beta)$ with $\varphi'' \in \Aut(X, \mu'')$.
\end{sit}

\begin{examples}\label{BD-amalgam} 
1. Given a Danilov-Gizatullin surface $X_d$ (see Example \ref{sit: ex-DG}),
for $d = 2, 3, 4, 5$ the group $\Aut(X_d)$  is an amalgam of a finite set of nested subgroups, hence is finitely bearable, see \cite{DG2}, \S\S 6 - 10 for details. 
Whereas for $d\ge 7$ this group is not countably bearable by Corollary \ref{cor: unbearable}.

2. An interesting example of a smooth Gizatullin surface $X$ with an amalgam structure of $\Aut X$
is given by the following construction, see \cite[5.5]{BD1}.
For $a, b \in \ka^*$, $c \in \ka$, and $a \neq b$, consider the smooth Gizatullin surface $X_{a, b, c}$  in $\a^4$ given by the  equations
\begin{eqnarray*}
xz &=& y(y - a)(y - b), \\
yw &=& z(z - c), \\
xw &=& (y - a)(y - b)(z - c)\, .
\end{eqnarray*}
\noindent The (abstract) isomorphism type of $X_{a, b, c} =: X$ does not depend on the parameters $a, b, c$, see \cite[5.5.6]{BD1}. 
Furthermore, $X$ possesses a $1$-standard completion of type $[[0, -1, -2, -3]]$. 
It is an easy exercise to show that $X$ admits 4 different families of $1$-standard completions $(\bar X_1, D_1)$, $(\bar X_{2, t}, D_{2, t})$, $(\bar X_{3, t}, D_{3, t})$, 
and $(\bar X_4, D_4)$, two of them depending on a parameter $t \in \ka \backslash \{0, 1\}$ 
(and these are isomorphic if and only if  the parameters $t, t'$ are equivalent under the relation $\sim$ generated by $t \sim t^{-1}$) 
and the other two are independent on any parameter. It is shown in \cite[5.5.4]{BD1} that the associated graph $\F_X$ has the following structure:
$$\begin{xy}
  \xymatrix{
    & & [(X_{2, s}, D_{2, s})] \ar@{<->}[r] & [(X_{3, s}, D_{3, s})] \\
  [(X_4, D_4)] \ar@{<->}[r] & [(X_1, D_1)] \ar@{<->}[ru] \ar@{<->}[rd] & \vdots & \vdots \\
    & & [(X_{2, t}, D_{2, t})] \ar@{<->}[r] & [(X_{3, t}, D_{3, t})] \\
  }
\end{xy}$$
\noindent where $(s, t)\in (\ka \backslash \{0, 1\})^2$ factorized 
by the equivalence relation $s \sim s^{-1}$. The group $\Aut X$ is 
an amalgam of the (uncountable set of) groups of automorphisms of 
$\a^1$-fibrations, see \cite[5.5.5]{BD1}. In particular, this group is uncountably bearable.
\end{examples}

Let us conclude this section with the following problem.

\medskip

\noindent {\bf Problem.} \emph{Determine, for which Gizatullin surfaces $X$ 
the neutral component $\Aut^\circ X$ is a (finitely or countably) bearable group. }

\section{Automorphism groups of $\A^1$-fibrations}
\label{sec: dJ}
As we have seen in Section \ref{sec: Giz}, the automorphism groups of $\A^1$-fibrations over affine bases play an essential role in studying the full automorphism group.
For an $\A^1$-fibered variety over an affine base of arbitrary dimension, we describe in Subsection \ref{ss: GDJ-grps}  the unipotent radical of such a group as a nested ind-group. 
In Subsection \ref{ss: surf-case} we give some immediate applications to the neutral component of the automorphism group of a given $\A^1$-fibration $\mu\colon X\to B$ 
on a normal affine surface $X$ over an affine curve $B$. Note that any such fibration is generated by some $\G_a$-action on $X$. However, the latter does not hold any longer for fibrations over projective bases.
In Sections \ref{sec:fiber-arc} and \ref{sec: dJ-surf} we dwell on a description of the full group of automorphisms $\Aut(X,\mu)$ in the surface case. 
For an $\ML_1$-surface $X$, this group coincides with the full automorphism group $\Aut X$.

\subsection{Generalized de Jonqui\`eres groups}\label{ss: GDJ-grps}
\begin{definition}\label{def: dJ}  
Let $X$ be a normal affine variety,
and let $\mu: X\to Z$ be an $\A^1$-fibration over
a normal affine variety $Z$, that is, a morphism with general scheme theoretical fibers 
isomorphic to the affine line. We assume that $\codim_Z(Z\setminus\mu(X))\ge 2$.
Consider the subgroups \begin{itemize}\item[--] $\Aut (X,\mu)\subset\Aut X$
of all automorphisms of $X$ preserving the fibration $\mu$;
\item[--] $\Aut_Z (X,\mu)\subset\Aut  (X,\mu)$ of those automorphisms 
which  preserve each $\mu$-fiber;  \item[--]
$\mathbb{U}_\mu\subset\Aut_Z (X,\mu)$
of those automorphisms which
restrict to translations on general $\mu$-fibers.\end{itemize}
Clearly, $\mathbb{U}_\mu$ is an Abelian group. 
This group is infinite-dimensional; see, e.g., Theorem \ref{thm: U-infty} below.
If $X=\A^2$ and $\mu\colon (x,y)\mapsto x$,
then $\mathbb{U}_\mu$ is the
maximal unipotent subgroup of the de Jonqui\`eres group,
see \ref{sit: A2}. In the general case,
we call $\Aut(X,\mu)$  a {\em generalized
de Jonqui\`eres group}, and $\mathbb{U}_\mu$ a {\em generalized unipotent
de Jonqui\`eres group}.
\end{definition}

\begin{remarks}\label{rem: exponentiating} 1. Recall that two
$\A^1$-fibrations $\mu_i: X\to Z_i$ on a normal affine variety $X$ over
normal affine varieties $Z_i$, $i=1,2$, are said to be equivalent if one can be sent
into another by an automorphism of $X$ which induces an isomorphism of $Z_1$ and $Z_2$.
Clearly, $\mu_1$ and $\mu_2$ are equivalent if and only if the corresponding 
subgroups $\mathbb{U}_{\mu_1}$
and $\mathbb{U}_{\mu_2}$ are conjugated
 in $\Aut X$.

2. If $\mu$ is locally trivial then $\mathbb{U}_\mu$ is the union of 
its unipotent one-parameter subgroups. In fact,
for any $\alpha\in \mathbb{U}_\mu$ there is a
a  locally nilpotent regular vertical vector field $\partial$ on $X$ such that $\alpha=\exp\partial$, 
and so,
$\alpha$ belongs to the unipotent one-parameter subgroup  
$H=\{\exp(t\partial)\}_{t\in\K}\subset  \mathbb{U}_\mu$. 
The latter holds as well if $X$ can be covered by affine charts $(U_i)_{i\in I}$ such that the restriction $\mu|_{U_i}$ 
is locally trivial for each $i\in I$ 
(such charts are automatically $\alpha$-stable). Moreover, the same conclusion remains 
true under a weaker assumption that in each chart $U_i$ of the covering the 
$\A^1$-fibration $\mu$ becomes locally trivial after a cyclic base change. 
This is the case, for instance, for any normal $\A^1$-fibered affine surface.

3.  Let $X$ be a normal affine surface. Then any $\G_{\rm a}$-action on $X$ 
acts along the fibers of an
$\A^1$-fibration $\mu: X\to Z$ over a smooth affine curve $Z$, see, e.g., 
\cite[Lem.\ 1.1]{Fi}. Thus, the group $\SAut X$ is generated by the unipotent de Jonqui\`eres
subgroups $\mathbb{U}_\mu$, where $\mu$ runs over the set of all the 
$\A^1$-fibrations on $X$ with affine bases. If $X$ is an $\ML_1$-surface, 
then $\mu$ is unique, and so, $\SAut X=\mathbb{U}_\mu$.
\end{remarks}

The group $\mathbb{U}_\mu$ admits the following presentation.
We let $\Quot (A)$ denote the quotient field of
an integral domain $A$, and $\K(Y)$ the function
field of an algebraic variety $Y$ over $\K$.

\begin{theorem}\label{thm: U-infty}
Let $\mu: X\to Z$ be an $\A^1$-fibration as in Definition \ref{def: dJ}. Suppose that
$\mathbb{U}_\mu$ is  the union of its unipotent one-parameter subgroups 
(see Remark \ref{rem: exponentiating}.2).
Then $$\mathbb{U}_\mu\cong H^0(Z, \cO_Z (D))$$ for a divisor $D$
on $Z$, where the class $[D]\in\Pic Z$
is uniquely defined by $\mu$.
If $\Pic Z=0$, then
there exists a locally nilpotent derivation
$\partial_0\in {\rm Der}\, \cO_X(X)$
such that $\mathbb{U}_\mu=\exp\left ((\ker\partial_0)\cdot \partial_0\right)$.
\end{theorem}

\begin{proof} Let $A=\mathcal{O}(X)$. Shrinking $Z$ appropriately one can
obtain an affine ruling, and even a
locally trivial $\A^1$-bundle
$X_\omega\to \omega$ on the normal affine variety $X_\omega=\mu^{-1}(\omega)$ over a normal affine base $\omega$,
where $\omega$ is a principal open subset of $Z$ (\cite{KW}; see also \cite{KM}). 
Shrinking the base further, one may assume that $X_\omega\cong \omega\times\A^1\to\omega$
is a principal cylinder in $X$.
There exists a $\G_{\rm a}$-action
$U=\exp (\K\cdot \partial)$
along the fibers of $\mu$, where $\partial\in {\rm Der} (A)$ is  
locally nilpotent and $\ker\partial=\mu^*(\mathcal{O}_Z(Z))$, see, e.g., \cite[Prop.\ 3.1.5]{KPZ}.

If $U'=\exp (\K\cdot \partial')$ is a second $\G_{\rm a}$-action on $X$
along the fibers of $\mu$, then
$\partial'=f\partial$ for some $f\in {\Quot} (\ker\partial)=
\mu^*(\mathcal{O}_Z(Z))$
such that $f\cdot\partial(A)\subset A$.
Conversely, for any rational function $f\in\mu^*(\mathcal{O}_Z(Z))$
such that $f\cdot\partial(A)\subset A$,  the derivation $\partial'=f\partial$ is
locally nilpotent on $A$, and so $U'=\exp (\K\cdot \partial')\subset \mathbb{U}_\mu$,
see, e.g., \cite[Prop.\ 1.1(b)]{FZ-LND}.

It follows that $\mathbb{U}_\mu=\exp\left (\mu^*\mathcal{H}\cdot \partial\right)$,
where
$$\mathcal{H}=\{u\in \mathcal{O}_Z(Z)\,\vert\,\partial(A)
\subset \mu^*(u^{-1})A\}\,.$$
Note that $\mathcal{H}$ is an $\mathcal{O}_Z(Z)$-module.
There is an isomorphism
$\mathcal{H}\stackrel{\cong}{\longrightarrow}\mathbb{U}_\mu$,
$u\mapsto \exp\left (\mu^*(u)\cdot \partial\right)$.
Assume that the set
$\mathfrak{D}=\{-{\rm div}(u)\,\vert\,u\in \mathcal{H}\}$
is bounded above, and consider the divisor  $D=\sup\,\mathfrak{D}$ on $Z$.
We claim that $\mathcal{H}=H^0(Z, \mathcal{O}_Z(D))$, that is,
$u\in\mathcal{H}$ if and only if $-{\rm div} (u)\le D$.
Since $\mathcal{H}\cong\mathbb{U}_\mu$,
this yields the required isomorphism $\mathbb{U}_\mu\cong H^0(Z, \mathcal{O}_Z(D))$.

To show the claim, it suffices to establish the inclusion
$H^0(Z, \mathcal{O}_Z(D))\subset \mathcal{H}$, the converse inclusion being
clear from the definition of $D$.
Let $u\in \mathcal{O}_Z(Z)$ be such that $-{\rm div} (u)\le D$. Then there exists
a cort\`ege $(u_1,\ldots,u_n)\in\mathcal{H}^n$ such that
$-{\rm div} (u)\le \max_{1\le i\le n}\,\{-{\rm div} (u_i)\}$. We claim that
$u\in \mathcal{H}$, i.e., $\mu^*(u)\partial A\subset A$, or, which is equivalent,
that
$-{\rm div}(\mu^*(u)\partial a)\le 0$ for any $a\in A$. Indeed, one has
$$-{\rm div}(\mu^*(u)\partial a)\le \max_{1\le i\le n}\,\{-{\rm div}(\mu^*(u_i)\partial a)\}\le 0\,,$$
since $-{\rm div}(\mu^*(u_i)\partial a)\le 0$ for $i=1,\ldots,n$. This proves our claim.

 To finish the proof of the first statement of the lemma, it remains to
show that $\mathfrak{D}$ is bounded above. Choose an element
$a\in \ker\partial^2\setminus \ker\partial\subset A$. Then
$\partial a=\mu^*(h)\in\ker\partial=\mu^*(\mathcal{O}_Z(Z))$,
where $h\in\mathcal{O}_Z(Z)$.
For $u\in \mathcal{H}$ we have
$\mu^*(h)\in \partial(A)\subset \mu^*(u^{-1})A$. Hence there exists
$b\in \mathcal{O}(Z)$ such that $h=u^{-1}b$.
Thus, $-{\rm div}\, u\le {\rm div}\, h$ and so, ${\rm div}\, h$
is an upper bound for $\mathfrak{D}$. Actually our argument shows
that $D\le D_0$, where
the effective divisor $D_0=\inf\{{\rm div} (h)\,|\,\mu^*h\in \partial ( \ker\partial^2)\}$
on $Z$ is given by the
zero locus of the ideal   $I=\mu_*(\partial ( \ker\partial^2))\subset \mathcal{O}_Z(Z)$.

Replacing in our construction $\partial$ by $\partial'$ results in replacing the divisor $D$ by a linearly
equivalent one $D'$. Starting with a suitable derivation $\partial'$ of the form $f\partial$, 
where $f\in\mu^*(\mathcal{O}_Z(Z))$,
one can get as $D'$ an arbitrary representative of the class $[D]\in\Pic Z$. Thus our construction
associates canonically the class $[D]$ to the $\A^1$-fibration $\mu\colon X\to Z$.

Now the second assertion follows easily. Indeed, if $\Pic Z=0$, then 
$D=-{\rm div}(u)={\rm div}_\infty(u)$ for a rational function
$u\in \K(Z)$. Then the locally nilpotent derivation
$\partial_0=\mu^*(u)\partial\in {\rm Der}\, A$
satisfies  $\mathbb{U}_\mu=
\exp\left ((\ker\partial_0)\cdot \partial_0\right)$. We leave the details to the reader.
\end{proof}

\begin{corollary}\label{cor: jonq=nested} Under the assumptions of Theorem \ref{thm: U-infty},
$\mathbb{U}_\mu$
is a unipotent Abelian nested ind-group.
\end{corollary}

\begin{proof}
Let $(\bar Z,D')$ be a completion of $Z$ by a divisor $D'=\bar Z\setminus Z$.
Then
\[
H^0(Z, \mathcal{O}_Z(D))=\varinjlim_n H^0(\bar Z, \mathcal{O}_{\bar Z}(D-nD'))\,,
\]
where $H^0(\bar Z, \mathcal{O}_{\bar Z}(D-nD'))$ is 
a finite-dimensional vector group for each $n$. Therefore, 
the vector group $\mathbb{U}_\mu\cong H^0(Z, \mathcal{O}_Z(D))$
is a nested ind-group.
\end{proof}

\begin{remarks}\label{rems: added}
1. Suppose that the affine variety $X$ as in Definition \ref{def: dJ} admits a free $\G_a$-action along the $\mu$-fibers. 
Then the corresponding locally nilpotent $\mu$-vertical vector field $\partial_0$ (that is, $\partial_0$ is tangent to the $\mu$-fibers) has no zero, 
and so, divides any other locally nilpotent $\mu$-vertical vector field on $X$. Then the equality $\mathbb{U}_\mu=\exp\left ((\ker\partial_0)\cdot \partial_0\right)$ holds. 
This is the case, for instance, for any smooth Danielewski surface $xy-p(z)=0$ in $\A^3$. 

2. Consider  a line bundle $L=(\mu\colon X\to Z)$, that is, a locally trivial $\A^1$-bundle on $Z$ with a fixed (zero) section. 
Then any $\G_a$-action on $X$ along the $\mu$-fibers is uniquely defined by the image of the zero section. 
Vice versa, given a section $S$ of $\mu$, there is a unique $\G_a$-action on $X$ along the $\mu$-fibers which sends the zero section to $S$. 
Hence $\mathbb{U}_\mu\cong H^0(Z, \O_Z(L))\cong H^0(Z, \O_Z(D))$ for any  divisor $D$, which represents the class of $L$ in $\Pic Z$.
\end{remarks}

\begin{proposition}\label{prop: jonq=nested1} 
Under the assumptions of Theorem \ref{thm: U-infty}, there is an exact sequence
\begin{equation}\label{eq: Aff-mu} 1\to \UU_{\mu} \to \Aut_Z (X,\mu)\to \Delta_\mu
\to 1\,, \end{equation} where $ \Delta_\mu\cong\Upsilon_\mu\times \ZZ^l$ for some $l\ge 0$ and some subgroup $\Upsilon_\mu\subset\G_m$.
\end{proposition}
\begin{proof} 
We use the notation from the proof of Theorem \ref{thm: U-infty}. 
Since $X_\omega\cong_Z \omega\times\A^1$ we have $\Aut_\omega (X_\omega,\mu)\cong\cO^+_\omega (\omega)\rtimes\cO^\times_\omega (\omega)$, 
where $A^+$ stands for the additive group of an algebra $A$ and $A^\times$ for its multiplicative group.  
The natural embedding $\Aut_Z (X,\mu)\embed \Aut_\omega (X_\omega,\mu)$ induces the  commutative diagram
 \begin{equation}\label{diagr: exact0}
\tikzset{math mode/.style = {execute at begin node=$\displaystyle, execute at end node=$}}
\begin{tikzcd}[math mode]
\displaystyle
1\arrow{r}& \cO^+_\omega(\omega) \arrow{r}& \Aut_\omega (X_\omega,\mu) \arrow{r}& \cO^\times_\omega (\omega) \arrow{r}& 1 \\
1\arrow{r}& \UU_{\mu} \arrow{r}\arrow[hook]{u}& \Aut_Z (X,\mu) \arrow{r}\arrow[hook]{u}&
 \Delta_\mu \arrow{r}\arrow[hook]{u}& 1 
\end{tikzcd}
\end{equation} where $ \Delta_\mu $ 
 is the image of $\Aut_Z (X,\mu)$ in $ \cO^\times_\omega (\omega)$. By Samuel's Units Lemma (\cite[Lem.\ 1]{Sa}, 
see also \cite[Lem.\ 4.3]{LZ}) we have $\cO^\times_\omega (\omega)\cong \G_m\times \ZZ^N$ for some $N\ge 0$.
 It is easily seen that any subgroup of the product $\G_m\times \NN^N$ is a product of subgroups of the factors. 
Hence $ \Delta_\mu\cong\Upsilon_\mu\times \ZZ^l$ for some $l\ge 0$ and some subgroup $\Upsilon_\mu\subset\G_m$.
\end{proof}

\subsection{The surface case}\label{ss: surf-case} 
Here we give some immediate applications of Theorem \ref{thm: U-infty} to the neutral component and the unipotent radical of the group $\Aut(X,\mu)$ 
in the case of a normal affine surface $X$. In the next two sections we enterprise a more thorough study of this group. 

\begin{remarks}\label{rem: surf-case} 
1. In the surface case, sequence (\ref{eq: Aff-mu}) splits and the subgroup $\Upsilon_\mu\subset\G_m$ is closed, 
see Proposition \ref{prop: splitting} and Theorem \ref{th:jonq-mu-general}. Plausibly, the latter holds in the general case as well.

2.
Rentchler's Theorem (\cite{Re}) says that any locally nilpotent derivation of the polynomial ring $\K[x,y]$ is conjugated to the derivation 
$\partial_0=f(x)\partial/\partial y$ for some $f\in\K[x]$. This can be generalized to the surfaces of class $(\ML_0)$ as follows. 
\end{remarks}

\begin{proposition}\label{cor: LND} 
Let $X$ be a smooth Gizatullin $\G_m$-surface, which is neither Danilov-Gizatullin, nor special (see Examples \ref{sit: ex-DG} and \ref{sit: special-Giz}). 
Then there exist two locally nilpotent derivations $\partial_0, \partial_1\in {\rm Der}\, \mathcal{O}_X(X)$ such that any other 
locally nilpotent derivation of $\mathcal{O}_X(X)$ is conjugated to a one of the form $f_i\partial_i$ for some $i\in\{0,1\}$ and $f_i\in\ker\partial_i$. 
\end{proposition}

\begin{proof}
This is an immediate consequence of Theorems \ref{thm: AG}  
and \ref{thm: U-infty}. 
\end{proof}

For the next theorem we address the reader to \cite{Be}, \cite{PZ}; cf.\ Proposition \ref{prop-Iitaka}. 

\begin{theorem}\label{thm: neutral-comp} Let $X$ be a normal affine surface of class $(\ML_1)$ or $(\ML_2)$.
Then the neutral component $\Aut^{\circ}X\subset \Aut X$ is isomorphic to
\begin{itemize}
\item $\G_m^r$ if $X$ is of type $(\ML_2,r)$, $r\in\{0,1,2\}$, and 
\item $\UU_\mu\rtimes\G_m^r$ if $X$ is of type $(\ML_1,r)$, $r\in\{0,1,2\}$, where $\mu\colon X\to B$ is a unique $\A^1$-fibration on $X$ over an affine curve $B$.
\end{itemize}
\end{theorem}

\begin{remark} 
Recall (see \ref{sit: toric}) that, up to isomorphism, the class $(\ML_2,2)$ consists of a single surface $(\A^1_*)^2$, 
and  the class $(\ML_1,2)$ of a single surface $\A^1\times\A^1_*$; 
see \ref{sit: A1-A1-star} and \ref{sit: A1star-2} for a description of the corresponding automorphism groups.
\end{remark}

From Corollary \ref{cor: jonq=nested} and Theorem \ref{thm: neutral-comp} we deduce the following result; 
see \cite[Cor.\ 2.3]{Be} for an alternative proof in the case of a rational surface.

\begin{corollary}\label{cor: neutral-comp} 
If $X$ is a surface of class $(\ML_1)$ or $(\ML_2)$, then $\Aut^{\0} X$ is
a solvable nested ind-group, and  any two maximal tori in $\Aut^{\0} X$ are conjugated.
\end{corollary}

\section{Formal neighborhood of a fiber in an $\A^1$-fibration}\label{sec:fiber-arc}
In this section we consider a normal affine surface  $X$ equipped with an 
$\AA^{1}$-fibration $\mu\colon X\to B$ over a smooth affine curve $B$. We study formal neighborhoods of fibers, the corresponding arc spaces, and their stabilizers.
These technical tools are used in the next section in the proofs of our main Theorems~\ref{th:jonq-mu} and \ref{th:jonq-mu-general} 
on the structure of the automorphism group of an  $\A^1$-fibration. 
Note that our technique is rather different from that of "tails" introduced in \cite[\S\S 1, 3]{DG2} in studies of 
automorphism groups of the Danilov-Gizatullin surfaces and based on Zariski's theory of complete ideals.

\subsection{Chain of contractions}\label{ss: contraction} 

\begin{notation}\label{nota: decomposition} Fix a minimal resolved SNC completion $(\bar X,D)$ of $X$ such that $\mu$ extends to a 
$\PP^{1}$-fibration 
$\bar\mu\colon \bar X\to\bar B$. Recall  that $\bar X$ is smooth and contains the minimal resolution of singularities of $X$, see Notation~\ref{nota: zigzag}. 

We call the degenerate $\bar\mu$-fibers  {\em special} and denote their union by $\Tt\subset\bar X$. Thus, $\Tt$ is a reduced effective divisor in $\bar X$. 
The components of a special fiber are smooth rational curves, and its dual graph is a tree.

Let $S$ be the unique section of $\bar\mu$ contained in  the boundary divisor $D= \bar X\setminus X$. 
The dual graph $\Gamma_D$ of $D$ is a rooted tree with a  prescribed root vertex $S$. 
The irreducible components of $F:=\bar\mu^{-1}(\bar B\setminus B)$ correspond to the $0$-vertices of degree 1 that are neighbors of $S$. 
The other branches of $\Gamma_D$ at $S$ are nonempty connected subgraphs of the dual graphs of the special fibers over $B$. 
Their intersection bilinear forms are negative definite. The extended divisor $D_{\rm ext}=S+F+\Tt$ of $( \bar X,\bar\mu)$ contains $D$.
\end{notation}

\begin{remark}
By Miyanishi's Theorem \cite[Ch.\ 3, Lem.\ 1.4.4(1)]{Mi0}, $X$ has only cyclic quotient singularities. 
The  weighted dual graph of the minimal resolution of such a singular point is a Hirzebruch-Jung string, see \cite{BHPV}.
\end{remark}

The following lemma is well known; see, e.g., \cite[Lem.\ 7]{Gi-2}. For the reader's convenience, we recall the proof.

\begin{lemma}\label{lem: contraction} Let $V$ be a smooth projective surface, $Z$ be a smooth projective curve, and let $\pi\colon V\to Z$ 
be a $\PP^1$-fibration, which admits a section $s\colon Z\to V$ with image $S$. 
Then there is a sequence of contractions 
$$ V=V_n\to V_{n-1}\to\ldots\to V_0$$ of $(-1)$-components of degenerate fibers disjoint from $S$ 
and from the subsequent images of $S$ that terminates by a ruled surface $V_0$ with an induced ruling $\pi_0\colon V_0\to Z$. 
\end{lemma}

\begin{proof} The lemma is an immediate consequence of  the following claim.

\medskip

\noindent \emph{Claim. If a fiber $F$ of $\pi$ is degenerate, then either $F$ contains at least two $(-1)$-components, or such a component is unique and multiple. 
In any case, at least one of the $(-1)$-components of $F$ is disjoint from $S$. }

Since $F\cdot S=1$, the second assertion 
follows from the first. 
Then also the lemma follows  
by induction on the total number of components of degenerate fibers.
Indeed, while contracting a $(-1)$-component of the fiber as an induction step, we reproduce the setting of the lemma with a smaller total number of fiber components. 

To fix the first statement of the claim, we
let $F=\pi^*(z)=\sum_{i=1}^{n}m_iF_i$, $z\in Z$. Since $F$ is degenerate and non-multiple we have $n\ge 2$.
Since $F^2=0$, by the adjunction formula we obtain $$-K_V\cdot F=2-2\pi_a(F)=2\quad\mbox{and}\quad 
-K_V\cdot F_i=F_i^2+2-2\pi_a(F_i)=F_i^2+2, \quad i=1,\ldots,n\,.$$ 
Hence $$\sum_{i=1}^n m_i(F_i^2+2)=2\,.$$ Since $F_i^2\le -1$ $\forall i$,
the positive summands correspond exactly to the $(-1)$-components of $F$. If $F_i$ 
is a unique such component, then necessarily $m_i>1$. Now the statement follows. 
\end{proof}

\begin{notation}\label{nota: decomposition-1}
Let $\Phi\colon \bar X\to\bar X_0$ be
a birational morphism, which contracts all degenerate $\bar\mu$-fibers to non-degenerate ones yielding a $\PP^1$-fiber bundle $\bar \mu_0\colon\bar X_0\to \bar B$.
Then $\Phi$ can be decomposed into a sequence of blowups of smooth points
\begin{equation}\label{eq: decomposition} \Phi\colon \bar X=
\bar X_m\xlongrightarrow{\sigma_{m}} \bar X_{m-1}\xlongrightarrow{\sigma_{{m-1}}} 
\ldots\xlongrightarrow{\sigma_{1}} \bar X_0,
\end{equation}
where $\sigma_{i}$ contracts the component $T_i\subset \bar X_i$ of the image $\sigma_{i+1}\circ\ldots\circ\sigma_m(\Tt)\subset\bar X_i$  to a point $p_{i}\in \bar X_{i-1}$. 
The proper transforms of the curves $T_i$ ($i< j$), 
$S$, and $F=\bar\mu^{-1}(\bar B\setminus B)$  on the surfaces $\bar X_j$  will be denoted by the same letters. The 
$\PP^1$-fibration $\bar\mu\colon\bar X\to\bar B$ induces 
$\PP^1$-fibrations $\bar\mu_i\colon\bar X_i\to\bar B$ 
so that $\sigma_i\colon\bar X_i\to\bar X_{i-1}$ 
becomes a morphism of $\PP^1$-fibrations identical on $\bar B$, $i=1,\ldots,m$. 
By Lemma~\ref{lem: contraction} we may assume that $T_1,\ldots, T_m$ do not meet $S$, and so, the 
centers of blowups $p_1,\ldots,p_m$ do not belong to $S$ or its images. Thus, $\Phi(T_1\cup\ldots\cup T_m)$ is a finite subset of the smooth affine surface   
$X_0:=\bar X_0\setminus (S\cup F)$.
 The induced $\AA^1$-fibration $\mu_0=\bar\mu_0|_{X_0}\colon X_0\to B$ is a locally trivial bundle with fiber $\A^1$.
\end{notation}

\begin{lemma}\label{lem:general-bundle}
$\mu_0\colon X_0\to B$  admits 
a structure of a line bundle. \end{lemma}

\begin{proof} 
The ruling $\bar\mu_0\colon\bar X_0\to\bar B$ is the projectivization of a rank 2 vector bundle $V\to\bar B$. 
The section $S$ of $\bar\pi$ corresponds to a line subbundle $L\subset V$. 
The exact sequence of vector bundles over the affine curve $B$, $$0\to L|_B\to V|_B \to (V/L)|_B\to 0$$ splits. Indeed, the obstacle to splitting sits in the group 
$${\rm Ext}^1(L, V/L)\cong H^1(B, \mathcal{H}om_{\cO_B}(\cO_B(L), \cO_B(V/L))\cong H^1(B, \cO_B(L)^{\vee}\otimes_{\cO_B} \cO_B(V/L))$$
 that vanishes due to Serre's analog of Cartan's A and B Theorems. This provides a section of $V|_B\to B$ disjoint with $S$, 
which can be taken for the zero section of a line bundle $\mu_0\colon X_0\to B$. 
\end{proof}

\begin{notation}\label{not: special-fiber}
Fix a special fiber, say, $\Tt'=\bar\mu^{-1}(\beta')$, with its reduced structure, where $\beta'\in B$. 
Since the blowups with centers in different fibers commute, with a suitable enumeration
we may assume that 
$$\Tt'=T_{0}\cup T_{1}\cup\ldots\cup T_{m'},$$ 
where $m'\le m$ and $T_0$ is the proper transform of $\bar\mu_0^{-1}(\beta')\subset\bar X_0$
and the only component of $\Tt'$ meeting $S$. We denote by $\Tt^{(i)}$ the image of $\Tt'$ in $\bar X_i.$
\end{notation}

\begin{definition} \label{def: star-comp} 
The blowup $\sigma_i$ is called {\em inner} if $p_i$ is a singular point of $\Tt^{(i-1)}$, and {\em outer} otherwise. 
The corresponding component $T_i$ is also called   {\em  inner} or  {\em outer}, respectively.
\end{definition}

\begin{definition}\label{def:parent} 
Let $T_i$, $T_k$ ($k<i$) be two components of $\Tt'$ such that $p_i=\sigma_i(T_i)\in T_k$, see \ref{nota: decomposition-1}. 
We say that $T_k$ is a \emph{parent} of $T_i$ if either $T_i$ is outer, 
or $T_i$ is inner and $p_i\in T_k\cap T_j$ on $\bar X_{i-1}$ for some $j<k$. 
Any component $T_i$, where $i>0$, has exactly one parent. 
\footnote{The notion of a {\em parent} is not related to the notion of a \emph{mother component} in \ref{StarComponentGeneral}.} 
\end{definition}

\subsection{Formal neighborhoods and coordinate charts}
\label{ss: loc-coordinates} 

\begin{definition}\label{def:formal-neigh}
Given an algebraic variety $Y$ and a closed subset $Z\subset Y$, we denote by $\hat\O_{Y,Z}$ the completion
of the local sheaf $\O_{Y,Z}$ with respect to the filtration by powers of the ideal sheaf $\mathcal{I}$ of $Z$. The corresponding formal scheme $\Spf\hat\O_{Y,Z}$ 
 is called a \emph{formal neighborhood} of $Z$ in $Y$,  see, e.g., 
\cite[Ch.\ 9]{Bad}
or
\cite[\S 10]{Grot}.
\end{definition}

\begin{notation}\label{sit: coord-abstr} 
Given a surface $V$ and 
 a local coordinate chart $(x_p,y_p)$ on $V$ centered at a smooth point $p\in V$ 
\footnote{Such a coordinate chart $(x_p,y_p)$ on $V$ can be defined as follows. 
Consider an affine neighborhood $X$ of $p$ in $V$, a closed  embedding $X\hookrightarrow\A^N$, and a linear projection $\pi\colon\A^N\to T_pX$. 
Let $U,V$ be linear functions  on $\A^N$ that restrict to coordinates in the tangent plane $T_pX$.  Then we let $x_p=U|_X$ and $y_p=V|_X$. Thus, 
$p$ is the origin of this local chart.}, we identify the completion
$\hat\O_{V,p}$ of the  local ring $\O_{V,p}$ with the ring $\K[\![x_p,y_p]\!]$.
 Let $\sigma\colon \tilde V \to V$ be the blowup of $p$ with exceptional $(-1)$-curve 
$E$. The  rational function  $\yE=y_p/x_p$ defines an isomorphism $\yE\colon E\stackrel{\cong}{\longrightarrow}\PP^1$.
For each $q\in\PP^1\cong\AA^1\cup\{\infty\}$ we let $E(q)=y_{\scriptscriptstyle E}^{-1}(q)\in E$.
The  sheaf $\hat\O_{\tilde V,E}$ inherits coordinates of $\hat\O_{V,p}$ as follows:
\begin{align*}
\hat\O_{\tilde V,E}(\left\{\yE\neq 0\right\})&=
\hat\O_{V,p}\left[\frac{1}{\yE}\right]\,=\,\K[1/\yE][\![y_p]\!],\\
\hat\O_{\tilde V,E}(\left\{\yE\neq \infty\right\})&=\hat\O_{V,p}[\yE]\,\,\,=\,\K[\yE][\![x_p]\!].
\end{align*}
 For a point $E(q)\in E$ the local coordinate chart centered at $E(q)$  is given by
\begin{equation}\label{eq:pt-coord}
\left(x_{E(q)},y_{E(q)}\right)=\begin{cases}
(x_p,\yE-q),&\quad q\neq\infty,\\
\left(\frac{1}{\yE},y_p\right),&\quad q=\infty\,.
\end{cases}
\end{equation}
 Reversing formulas (\ref{eq:pt-coord}) yields
\begin{equation}\label{eq:pt-coord-2}
\left(x_p,y_p\right)=\begin{cases}
(x_{E(q)},\,x_{E(q)}(y_{E(q)}+q)),&\quad q\neq\infty,\\
\left(x_{E(q)}y_{E(q)},\,y_{E(q)}\right),&\quad q=\infty\,.
\end{cases}
\end{equation}
Letting $x_{\scriptscriptstyle E}=x_p$ we call $(x_{\scriptscriptstyle E},\yE)$ the {\em local coordinates near} $E$. 
\end{notation} 

\begin{notation}\label{not:succ}
We write $p'\succeq p$
if $p'\in\Tt^{(i+k)}$ and $\sigma_{i+1}\circ\ldots\circ\sigma_{i+k}(p')=p\in \Tt^{(i)}$ for some  $k\ge 0$ (that is, $p'$ is an infinitely near point of $p$). 
If $p$ is the center of a blowup $\sigma_{i+j}$,  $1\le j\le k$, then we write $p'\succ p$. Otherwise,  by abuse of notation, we write $p=p'$. 
\end{notation}
\begin{notation}\label{sit:coord} 
Tensoring with the ring $\hat\O_{B,\beta'}=\K[\![t]\!]$, where $t$ is a local coordinate on $B$ centered at $\beta'$, 
we restrict $\mu_0$ to the formal neighborhood of the fiber $T_0^*:=T_0\setminus S=\mu_0^{-1}(\beta')\cong\A^1$ in $X_0$, namely, to
\[\hat\O_{X_0, T_0^{*}}(T_0^{*})= \K[y_0][\![x_0]\!],\] 
where the coordinates $(x_0,y_0)$ near the fiber $T_0^{*}\subset X_0$ are chosen so that $x_0=\mu_0^*(t)$, and $y_0=0$  defines the zero section of $\mu_0\colon X_0\to B$.
Regarding $y_0$ as a $\PP^1$-coordinate on $T_0\subset\bar X_0$, 
we define local coordinates at each point $T_0(q)$
of $T_0^{*}=T_0\cap X_0$ via (\ref{eq:pt-coord}).

For every $i=1,\ldots,m'$ we define  recursively local coordinates at the points of the fiber $\Tt^{(i)}\subset\bar X_i$  as follows. 
Assume that the local coordinates $(x_{p_i},y_{p_i})$ centered at a point $p_i\in\bar X_{i-1}$ are already defined. 
As in \ref{sit: coord-abstr} we infer first a $\PP^1$-coordinate $y_i=y_{\scriptscriptstyle T_i}$ on $T_i$ and then local coordinates near $T_i(q)$ for each $q\in\PP^1$. 
For any point $p\in \Tt^{(i)}\setminus (T_i\cup S)$ 
we  keep the same local chart $(x_p,y_p)=(x_{\sigma_i(p)},y_{\sigma_i(p)})$ as on the surface $X_{i-1}$. In more detail, each point $p\in \Tt^{(i)}$ admits a unique 
representation of the form $p=T_j(q)$ as in Notation~\ref{not:succ}, where $j\le i$ and $q\in\PP^1$. 
In particular, if $p=T_j\cap T_k$ with $j>k$, then $p=T_j(q)$, where $q\in\PP^1$, and $p$ cannot be represented as $T_k(q')$ for $q'\in\PP^1$. 
Then we let $(x_p,y_p)=(x_{\scriptscriptstyle T_j(q)},y_{\scriptscriptstyle T_j(q)})$. 
\end{notation}

\begin{remark}\label{rem: coord}
If $T_j\cap T_k=\{p\}\subset \bar X_i$, then $T_j$ and $T_k$ are the coordinate lines in the local coordinates $(x_p,y_p)$ near $p$. Up to permuting $j$ and $k$ 
one has $p\succeq
T_j(\infty)$, $p\succeq T_k(q)$, where $q\neq\infty$, and \[(x_p,y_p)=\left(\frac{1}{y_j},y_k-q\right).\] 
 Note that $T_j(\infty)$ is the point of $T_j$ closest to $S$, and the component $T_k$ separates $T_j$ and $S$ in $\Tt^{(i)}$.
\end{remark}

\subsection{Arcs and multiplicities}
\label{ss: arcs} In this subsection we introduce the arc spaces  of an $\A^1$-fibration (see \ref{sit: coord-abstr} and \ref{def:
arc-space}) and the multiplicities of arcs (see \ref{sit:pui-mult} and \ref{prop:mult}).

\begin{definition}[\emph{arc space}; see, e.g., \cite{EM,
Is}]\label{def: arc-space}  Given a variety $V$, an \emph{arc} in $V$ is a parameterized formal curve germ $\Spec\K[\![t]\!]\to V$. 
The \emph{arc space} $\Arc(V)$ is the $\K[\![t]\!]$-scheme consisting of all the $\K[\![t]\!]$-rational 
points of $V$. Given a closed subvariety $Z\subset V$, we say that an arc $\xi\colon\Spec\K[\![t]\!]\to V$ is \emph{centered in} $Z$ if 
 the  image in $\K[\![t]\!]$
of the vanishing ideal  of $Z$ in $\mathcal{O}_V(V)$  is contained in  the maximal ideal of  $\K[\![t]\!]$.  
In particular, if $Z$ is a reduced point $p\in V$, then we say that $\xi$ is \emph{centered at} $p$.
We let $\Arc(V)_Z$ denote the subscheme of all arcs centered in $Z$. 
This subscheme can be identified with a scheme of arcs in the formal neighborhood of $Z$ in $V$.
\end{definition}

\begin{example}\label{ex: arc-p}
Let $V$ be a surface and $p\in V$ be a smooth point 
with a local coordinate chart $(x_p,y_p)$ centered at $p$. 
The corresponding arc space is
\[\begin{aligned}
\Arc(V)_{p}=&\big\{h\colon\hat\O_{V,p}\to\K[\![t]\!]\mid h(\mathfrak{m}_{V,p})
\subseteq t\K[\![t]\!]\big\}\\
=&\big\{(x_p,y_p)\mapsto(x(t),y(t))\mid x,y\in t\K[\![t]\!]\big\}.
\end{aligned}\]
If $\sigma\colon \tilde V \to V$ is the blowup of $p$ with exceptional $(-1)$-curve 
$E$,
then $\sigma$  induces an isomorphism $\Arc(\tilde V)_E\setminus 
\Arc(E)\cong \Arc(V)_p\setminus\{0\}$. 
\end{example}

\begin{notation}\label{nota: descent} For a point  $p\in\Tt^{(i)}\setminus S\subset \bar X_i\setminus S$ ($i\in\{0,\ldots,m'\}$)  we
consider  the subset
$\Arc(\bar X_i)_p^*\subset \Arc(\bar X_i)_p$ of all arcs in $\bar X_i$ centered at $p$
whose generic point does not belong to $\Tt'$. 
 The map $\Arc(\bar X_i)_p^*\to \Arc(X_0)_{T_0^*}^*\subset\Arc(\bar X_0)_{T_0}^*$ induced by 
$\sigma_1\circ\ldots\circ\sigma_i\colon\bar X_i\to\bar X_0$ is injective.  This allows to
identify $\Arc(\bar X_i)_p^*$ with its image in $\Arc(X_0)_{T_0^*}^*$, where 
\begin{align*}\label{eq:arc-sp} \Arc(X_0)_{T_0^*}^*=& \Arc(X_0)_{T_0^*}\setminus \Arc(T_0^*)\\=&\big\{(x_{0},y_{0})
\mapsto(x(t),y(t))\;|\; x\in t\K[\![t]\!]\setminus\{0\},y\in\K[\![t]\!]\big\}\,.
\end{align*} 
\end{notation}

\begin{notation} 
Let $\mult T_i=m_i$ be the multiplicity of $T_i$ in the divisor $\bar\mu_{m'}^*(\beta')=\sum_{i=0}^{m'} m_iT_i\,$, 
which corresponds to a special fiber $\bar\mu_{m'}^{-1}(\beta')$. 
\footnote{We distinguish between the divisor $\mu^*( \beta )$ and its reduced version, that is, the geometric fiber $\mu^{-1}(\beta)$.}
Thus, $\mult T_i=\mult T_j$ if $T_i$ is outer and $\sigma_i(T_i)=p_i\in T_j$, and $\mult T_i=\mult T_j+\mult T_k$ if $T_i$ is inner and $p_i\in T_j\cap T_k$ 
(see \ref{not: special-fiber}). In particular, $\mult T_i$ is the same on every surface $\bar X_j$, $j\ge i$. 
\end{notation}

\begin{definition}[\emph{multiplicity of an arc}]\label{sit:pui-mult}
For an arc $h\in\Arc(\bar X_0)_{T_0}^*,$ $(x_0,y_0)\mapsto(x(t),y(t))$, we define its multiplicity  by $\mult h=\ord_t x \,(=\ord_t (\bar\mu_0)_*h)$. 
The multiplicity of an arc $h \in\Arc(\bar X_i)_{\Tt^{(i)}}^*$ is defined as the multiplicity  of the image of $h$ in $\Arc(\bar X_0)_{T_0}^*$.
\end{definition}

\begin{proposition}\label{prop:mult}
Let $p\in\Tt^{(i)}$ and $h\in\Arc(\bar X_i)_p^*,\; h\colon(x_p,y_p)\mapsto(x(t),y(t))$.
\footnote{Since $h\in\Arc(\bar X_i)_p^*$, $x(t)\neq 0$ in the former case and $x(t),y(t)\neq0$ in the latter one, 
thus the formulas (\ref{eq:arc-mult-1})--(\ref{eq:arc-mult-2}) are well defined.}
If $p\notin\Sing \Tt^{(i)}$ is a point of $T_j$,
then
\begin{equation}\label{eq:arc-mult-1}
\mult h=\mult(T_j)\ord_t x.
\end{equation}
If $p\in \Sing \Tt^{(i)}$ is an intersection of components $T_j$ and $T_k$ with $\sigma_{k+1}\circ\ldots\circ\sigma_i(p)=T_k(\infty)$
(cf.\ \ref{rem: coord}), then
\begin{equation}\label{eq:arc-mult-2}
\mult h=\mult(T_j)\ord_t x+ \mult(T_k)\ord_ty.
\end{equation}
\end{proposition}
\begin{proof}
We proceed by induction on $i$. The case $i=0$ is trivial. Assume that the assertion holds on $\bar X_{i-1}$.
Then it also holds on $\Tt^{(i)}\setminus T_i$. Let then $p=T_i(q)$, $q\in\PP^1$, so $\sigma_i(p)=p_i$.
We distinguish the  following five cases:\begin{itemize}  \item $T_i$ is outer and $q=\infty$;
\item  $T_i$ is outer and  $q\neq\infty$;
\item  $T_i$ is  inner and $q=\infty$;
\item $T_i$ is  inner and $q=0$;
\item $T_i$ is  inner and  $q\neq 0,\infty$. 
\end{itemize}
These are depicted below in Figure~1
 for $T_i$ outer and in Figure~2 for $T_i$ inner.
\tikzset{pt/.style={circle,draw,fill=white, inner sep=0pt, minimum size=4pt}}
\begin{figure}[ht!]
\centering
\begin{minipage}[c]{0.4\linewidth}
{\begin{tikzpicture}[baseline={([yshift=-.5ex]current bounding box.center)}]
	\draw (1,2)--(2,5) node[left, very near end] {$T_j$};
	\draw (1,4)--(2,1) node[below left, very near end] {$T_i$};
	\node at (1.33,3)[pt, label=right:$T_i(\infty)$] {};
	\node at (1.66,2)[pt, label=right:$T_i(q)$] {};
	\draw [->, thick](3.2,3)--(4.2,3) node[above,midway]{$\sigma_i$};
	\draw (5,2)--(5,5) node[left, near end] {$T_j$};
	\node at (5,3)[pt, label=right:$p_i$] {};
\end{tikzpicture}}
\label{fig:out}
\caption{$T_i$ is outer.}
\end{minipage}
\begin{minipage}[c]{0.4\linewidth}
{\begin{tikzpicture}[baseline={([yshift=-.5ex]current bounding box.center)}]
	\draw (0,4)--(2,5) node[above left, very near start] {$T_j$};
	\draw (1,5)--(1,1) node[above left, midway] {$T_i$};
	\draw (0,2)--(2,1) node[below left, very near start] {$T_k$};
	\node at (1,4.5)[pt, label={[yshift=-1.2ex]0:$T_i(\infty)$}] {};
	\node at (1,3)[pt, label=right:$T_i(q)$] {};
	\node at (1,1.5)[pt, label={[yshift=1ex]0:$T_i(0)$}] {};
	\draw [->, thick](2.6,3)--(3.6,3) node[above,midway]{$\sigma_i$};
	\draw (4,2)--(5,5) node[left, very near end] {$T_j$};
	\draw (4,4)--(5,1) node[below left, very near end] {$T_k$};
	\node at (4.33,3)[pt, label=right:$p_i$] {};	
\end{tikzpicture}}
\label{fig:inn}
\caption{$T_i$ is inner.}
\end{minipage}
\end{figure}
All these cases are treated similarly, so we consider just  two of them and leave the others to the reader.

Let first $T_i$ be outer with a parent $T_j$, where $p=T_i(q)$, $q\neq\infty$.
Then $(x_{p_i},y_{p_i})=(x_p,x_p(y_p+q))$ by (\ref{eq:pt-coord-2}), so $h\colon(x_{p_i},y_{p_i})\mapsto(x(t),x(t)(q+y(t)))$ on $X_{i-1}$, and by the induction conjecture
\[\mult h=\mult(T_j)\ord_t x=\mult(T_i)\ord_t x.\]
Let further $T_i$ be inner and $p=T_i(\infty)$. Then $h\colon(x_{p_i},y_{p_i})\mapsto (x(t)y(t),y(t))$ on $X_{i-1}$ by (\ref{eq:pt-coord-2}), and so, by the inductive conjecture,
\begin{multline*}
\mult h=\mult(T_j)(\ord_tx+\ord_ty)+\mult(T_s)\ord_ty=
\mult(T_j)\ord_tx+\\
(\mult(T_j)+\mult(T_s))\ord_ty=\mult(T_j)\ord_tx+\mult(T_i)\ord_ty\,.
\end{multline*}
\end{proof}  

\subsection{Puiseux arc spaces} In this subsection we introduce and study Puiseux arc spaces. 

\begin{sit}[\emph{Puiseux arcs}]\label{sit:arc-puiseux} 
Fix a point\footnote{Warning: in what follows we never consider the points `at infinity' $p\in\Tt^{(i)}\cap S$.} $p\in\Tt^{(i)}\setminus S$. 
\begin{enumerate}[1.] 
\item An \emph{invertible substitution} is a change of variable $t\mapsto A(t)$, where $\ord_tA=1$. 
All such substitutions form a group acting on the arc space $\Arc(\bar X_i)_p^*$ via
\[A(t). h\colon(x_p,y_p)\mapsto(x_h(A(t)),y_h(A(t))),\]
where $h\colon(x_p,y_p)\mapsto(x_h(t),y_h(t)).$ We call two arcs \emph{equivalent} if they belong to the same orbit of the action.
\item 
The coordinate line $x_p=0$ in a local chart $(x_p,y_p)$ is a part of a component of $\Tt^{(i)}$. 
Hence  $x(t)\neq 0$ for any arc $h\in\Arc(\bar X_i)_p^*$, $h\colon(x_p,y_p)\mapsto(x(t),y(t))$. 
So, a suitable invertible substitution $A(t)$ sends $h$ to an arc $\tilde h=A(t).h\colon(x_p,y_p)\mapsto(t^n,\tilde y(t))$, 
where $n=\ord_tx$ and $\ord_t\tilde y=\ord_ty.$  

Such a change of variable $t\mapsto A(t)$ and an arc $\tilde h$ are defined 
uniquely up to a composition with $t\mapsto\alpha\cdot t$, where $\alpha\in\K$, $\alpha^n=1$. 
All the arcs equivalent to $\tilde h$ share the same Puiseux expansion $y=\tilde y(x^{1/n})$. They form an orbit of the cyclic Galois group for the reduction problem. 
\item An arc $\tilde h\colon (x_p,y_p)\mapsto(t^n,\tilde y(t))$ with $\tilde y=\sum_{i=\ord_ty}^\infty a_it^i$ is called a \emph{Puiseux arc} if  
\[\gcd(\{i\mid  a_i\neq0\}\cup\{n\})=1.\]
\end{enumerate}
\end{sit}

\begin{definition}[\emph{Puiseux arc space}]\label{def: Puiseux} Consider a point $p\in\Tt^{(i)}\setminus S$, a
pair of positive integers $n,d\in\ZZ_{>0}$, and a polynomial $\psi=\sum_{i=1}^{d-1}\psi_it^i\in t\K[t]$ such that 
\begin{equation}\label{eq:gcd-puiseux}
\gcd(\{i\;|\;\psi_i\neq0\}\cup\{n\})=1,
\end{equation}
that is, such that  $(x_p,y_p)\mapsto(t^n,\psi(t))$ is a Puiseux arc. 
Assume also that $\psi\neq0$ if $p$ is a node.
The \emph{Puiseux arc space} $W=\Pui_{\bar X_i, p}(\psi, n, d)$ on $\bar X_i$
relative to the coordinate system $(x_p,y_p)$ centered at $p$ 
consists of all arcs equivalent to Puiseux arcs of the form $h\colon (x_p,y_p)\mapsto(t^n,y(t))$ with $y\in\psi+t^d\K[\![t]\!]$. In other words,
\[
W=\{(x_p,y_p)\mapsto(A(t)^n,y(A(t)))\mid  \ord_tA=1,\,y\in\psi+t^d\K[\![t]\!]\}.
\] 
Thus, the elements of $W$ share the same starting piece $\psi$ of the Puiseux expansion. 
The condition that $\psi\neq0$ if $p$ is a node ensures that $W \subset \Arc(\bar X_i)_p^*$. 
Furthermore, the elements of $W$ share the same multiplicity (see \ref{sit:pui-mult}), which we denote by $\mult W$.

Any Puiseux arc space can be expressed in the coordinates $(x_0,y_0)$ on $X_0$, see \ref{nota: descent}. To this end, given
a pair of positive integers $n,d\in\ZZ_{>0}$ and a polynomial $\psi=\sum_{i=0}^{d-1}\psi_it^i\in \K[t]$ with a possibly nonzero constant term, we let
\[\Pui(\psi,n,d)=\Pui_{\bar X_0,T_0(\psi_0)}(\psi-\psi_0,n,d).\]
Thus, the constant term $\psi_0$ is responsible for the choice of the center $T_0(\psi_0)\in T_0\setminus S$. 
\end{definition}

\begin{lemma}\label{lem: Puiseux}
For any point $p=T_i(q)\in \Tt^{(i)}\setminus S$, $i\in\{0,\ldots, m'\}$, and any 
Puiseux arc space $W=\Pui_{\bar X_i,p}(\psi, n, d)\subset \Arc(\bar X_i)_p^*$,  
the image $W'$ of $W$  under the embedding $\Arc(\bar X_i)_p^*\hookrightarrow \Arc(\bar X_{i-1})_{p_i}^*$ induced by 
$\sigma_i\colon \bar X_i\to\bar X_{i-1}$  is a Puiseux arc space in $\Arc(\bar X_{i-1})_{p_i}^*$. 
More precisely, 
\begin{align}\label{eq:Puiseux-descent-finite}
W'=&\Pui_{\bar X_{i-1}, p_i}(t^n(q+\psi), n, d+n) &\mbox{if}\quad q\neq\infty,\\
W'=&\Pui_{\bar X_{i-1}, p_i}(\widetilde\psi, n+\ord\psi, d), &\mbox{if}\quad q=\infty,\label{eq:Puiseux-descent-inf}
\end{align}
where $\ord\widetilde\psi=\ord\psi$.
\end{lemma}
\begin{proof}
According to Definition \ref{def: Puiseux} one has
\[
W=\big\{(x_p,y_p)\mapsto (A(t)^n,\psi(A(t))+\eta(A(t)))\mid \ord_tA=1, \eta\in t^d\K[\![t]\!]\big\}\,.
\] 
Suppose first that $p=T_i(q)$, where $q\neq\infty$. By (\ref{eq:pt-coord}) we have $(x_{p_i},y_{p_i})=\left(x_p,x_p(y_p+q)\right)$. Hence
\[
W'=\big\{(x_{p_i},y_{p_i})\mapsto(A(t)^n,y(A(t)))\mid \ord_tA=1, 
y\in t^n(q+\psi)+ t^{d+n}\K[\![t]\!]\big\}\]
is again a Puiseux arc space relative to the coordinate system $(x_{p_i},y_{p_i})$ in $\bar X_{i-1}$. More precisely, 
$W'=\Pui_{\bar X_{i-1}, p_i}(\widetilde\psi, n, d+n)$, where
$\widetilde\psi(t)=t^n(q+\psi(t))\in\K[t]$ is a polynomial of degree $<d+n$ satisfying (\ref{eq:gcd-puiseux}). 

Let further $p=T_i(\infty)$.  Then 
one can write $t=\sum_{i=1}^{\infty}\alpha_is^i\in s\K[\![s]\!]$, where 
$\alpha_1\neq0$ and $t^{n}\psi(t)=s^{n+\ord\psi}$.
Plugging in $t=t(s)$ sends the set $\psi(t)+t^d\K[\![t]\!]$   into 
$\widetilde\psi(s)+s^d\K[\![s]\!]$ for some polynomial $\widetilde\psi\in s\K[s]$ of degree $<d$ and of order $\ord\widetilde\psi=\ord\psi$.  
By (\ref{eq:pt-coord}) we have $(x_{p_i},y_{p_i})=(x_py_p,y_p)$. Hence
\[\begin{aligned}
W'=&\big\{(x_{p_i},y_{p_i})\mapsto(A(t)^ny(A(t)),y(A(t)))\mid \ord_tA=1, y\in \psi+ t^{d}\K[\![t]\!]\big\}\\
=&\big\{(x_{p_i},y_{p_i})\mapsto(\widetilde A(s)^{n+\ord\psi},y(\widetilde A(s)))\mid \ord_s\widetilde A=1, y\in \widetilde\psi+ s^{d}\K[\![s]\!]\big\}
\end{aligned}\]
coincides with the Puiseux arc space $\Pui_{\bar X_{i-1}, p_i}(\widetilde\psi, n+\ord\psi, d)$ relative to the coordinate system $(x_{p_i},y_{p_i})$ in $\bar X_{i-1}$, 
where the polynomial $\widetilde\psi\in s\K[s]$ still satisfies (\ref{eq:gcd-puiseux}). 
Indeed, otherwise one can write  $\widetilde\psi(s)=\widetilde\varphi(s^k)$, where $\widetilde\varphi\in\K[s]$, $k>1$, and $k|n$. 
However, plugging in the expression  $s=s(t)\in t\K[\![t]\!]$ yields $\psi(t)=\varphi(t^k)$ for some polynomial $\varphi\in\K[t]$. 
The latter contradicts condition (\ref{eq:gcd-puiseux}) for $\psi$.
\end{proof}

The  following corollary is straightforward.
 
\begin{corollary}\label{cor: const-mult} 
Given a Puiseux arc space $W=\Pui_{\bar X_i,p}(\psi,n,d)$, its image 
   under the embedding $\Arc(\bar X_i)_p^*\embed  \Arc(\bar X_0)_{T_0}^*$
   is a Puiseux arc space, say
 $\Pui(\tilde\psi,\tilde n,\tilde d)$. In particular, $\mult W=\tilde n.$ 
\end{corollary}

\begin{notation}\label{def:pui-p}
Given a point $p\in\Tt^{(i)}\setminus \Sing (\Tt^{(i)}\cup S)$, we let $\Pui(p)$ be the image of $\Pui_{\bar X_i,p}(0,1,1)$ in $\Arc(\bar X_0)_{T_0}^*$. 
By the preceding corollary, $\Pui(p)$ is a Puiseux arc space.
\end{notation}

\begin{corollary}\label{cor: n=1}
Let $\Pui(p)=\Pui(\psi,n,d)$, where $p\in T_j\setminus S$ is not a node of $\Tt^{(i)}$. Then $\mult\Pui(p)=n=\mult T_j.$
Furthermore, $n=1$ if and only if $T_j$ is obtained via a sequence of outer blowups. 
\end{corollary}
\begin{proof} The first statement follows from Corollary \ref{cor: const-mult} and (\ref{eq:arc-mult-1}). 
The second follows from the first due to the fact that $\mult(T_j)=1$ if and only if the component $T_j$ is obtained via a sequence of outer blowups. 
\end{proof}

\begin{remark}
 For each Puiseux space $\Pui(\psi,n,d)\subset\Arc\, (X_0)$ centered at a point $p\in T_0^*$ there exists a surface $X'$ and 
a sequence of blowups $X'\to \bar X_0$ with centers at infinitely near points  of $p$  such that $\Pui(\psi,n,d)=\Pui(p')$ for some point $p'\in X'$, $p'\succeq  p$. 
So, there is a one-to-one correspondence between infinitely near points of $T_0^*\subset X_0$, and  the Puiseux arc spaces. 
\end{remark}
\subsection{Stabilizer of a special fiber}
 In this subsection we
study the action  of
 the automorphism group of an $\A^1$-fibration on the Puiseux arc spaces of a special fiber. We use the following notation.

\begin{notation}\label{nota:aff-module} 
Given a $\K$-module $M$  (a commutative $\K$-algebra $A$, respectively), we let $\G_a(M)$ ($\G_m(A)$, respectively) 
denote the additive group of $M$ (the group of units of $A$, respectively). We let also
$\Aff(A)=\G_a(A)\rtimes\G_m(A)$ denote the group of affine transformations of the affine line over $A$.
\end{notation}

\begin{notation}\label{sit: Jonq0-local}
Consider the following groups of  automorphisms of  the arc space\footnote{Recall that $(x,y)$ stands for an arc  in $\bar X_0$ and $(x_0,y_0)$ for the local coordinates in $\bar X_0$.} $\Arc(\bar X_0)_{T_0}^*$:
\[\Q=\left\{(x,y)\mapsto(ax,Q(x)y)\mid a\in\K^\times,\, Q=\sum_{i=0}^{\infty} b_ix^i\in \K[\![x]\!]^\times\right\}\cong \G_m( \K[\![t]\!])\rtimes\G_m\,,\]
 \[\TTT=\left\{(x,y)\mapsto(ax,by)\;|\; (a,b)\in(\K^\times)^2\right\}\cong (\mathbb{G}_m)^2\,,\]
\[G(i)=\left\{(x,y)\mapsto(x,y+c_i x^i)\;|\; c_i\in\K\right\}\cong\G_{a}\,,\]
where as before $x\in t\K[\![t]\!]\setminus\{0\},y\in\K[\![t]\!]$,  see \ref{nota: descent}.
\end{notation}

The following lemma is immediate.

\begin{lemma}\label{lem:new}
We have
$\,\qquad\Aut \Arc(\bar X_0)_{T_0}^*=$
\[\left\{(x, y)\mapsto(ax,Q(x)y+P(x))\;|\; a\in\K^\times, P=\sum_{i=0}^{\infty} c_it^i\in\K[\![t]\!],\, Q=\sum_{i=0}^{\infty} b_it^i\in\K[\![t]\!]^\times \right\}\,.\]
Consequently,
$$\Aut \Arc(\bar X_0)_{T_0}^*=\left(\prod_{i=0}^{\infty} G(i)\right)\rtimes  \Q\cong \Aff ( \K[\![t]\!])\rtimes\G_m\,,$$
where the  factor $\G_m$ acts on $ \Aff ( \K[\![t]\!])$ via $t\mapsto at$ for $a\in\K^\times$.
\end{lemma}

\begin{notation}\label{nota:stabilizer}
 We let $$\Stab_{\mathrm{fps}}W\subset\Aut \Arc(\bar X_0)_{T_0}^*$$ denote the stabilizer of a subset $W\subset \Arc(\bar X_0)_{T_0}^*$, 
where `fps' stands for `formal power series'. 
Attributing the lower index $B$ to a group of automorphisms means passing to the subgroup of automorphisms that act trivially on the first coordinate, 
that is, verify $a=1$ in the notation as above (cf., e.g., \ref{def: dJ}). In particular, we consider the subgroups 
$\Aut_B \Arc(\bar X_0)_{T_0}^*\subset\Aut \Arc(\bar X_0)_{T_0}^*$, $\Q_B\subset\Q$, and the one-dimensional torus $\TTT_B\subset\TTT$.
 \end{notation}

\begin{notation}Given a Puiseux arc space $W=\Pui(\psi,n,d)$, 
we can decompose 
\begin{equation}\label{eq: split} \psi(t)=\psi^\reg(t^n)+\psi^\sing(t)\,,\end{equation} where $\psi^\sing\in\K[t]$ is the sum of all monomials in $\psi$ with exponents not divided by $n$.
\end{notation}

\begin{lemma}\label{lem:Q-stab-jet}
Consider a Puiseux arc space  $W=\Pui(\psi,n,d)\subset \Arc(\bar X_0)_{T_0}^*$.
 An automorphism  $g\in \Aut \Arc(\bar X_0)_{T_0}^*$, 
$g\colon (x,y)\mapsto (ax,Q(x)y+P(x))$,
stabilizes $W$ if and only if 
\begin{equation}\label{eq: stab-P} 
P(x)=\psi^\reg(ax)-\psi^\reg(x)Q(x)\quad\mod x^{\lceil\frac{d}{n}\rceil}\K[\![x]\!],\\
\end{equation}
where $\lceil \frac{d}{n}\rceil$  stands for the smallest integer  $\ge\frac{d}{n}$, 
and in the case $n>1$ also
\begin{equation}\label{eq: stab-Q}
\begin{cases} 
Q(s^n)\psi^\sing(s)=\psi^\sing(\alpha s)\quad\mod s^d\K[\![s]\!]\\
\alpha^n=a\,
\end{cases}
\end{equation}
for some  $\alpha\in\K$.
\end{lemma}
\begin{proof} Consider an arc $h\in W$,
$$h\colon (x_0,y_0)\mapsto(A(t)^n,y(A(t)))\quad\mbox{with}\quad A\in t\K[\![t]\!]\setminus  \{0\}\quad\mbox{and}\quad y\in \psi+ t^d\K[\![t]\!]\,.$$  
The automorphism $g$ sends $h$
to the arc 
$$g. h\colon (x_0,y_0)\mapsto(a\cdot A(t)^n,\,Q(A(t)^n)\cdot y(A(t))+P(A(t)^n))\,.$$
Hence $g. h\in W$ if and only if 
 \[\begin{cases} a\cdot A(t)^n=\widetilde A(t)^n\\
Q(A(t)^n)\cdot y(A(t))+P(A(t)^n)=\widetilde y(\widetilde A(t))\end{cases} 
 \]
for some $\widetilde A(t)\in t\K[\![t]\!]$ and $\widetilde y\in \psi+ t^d\K[\![t]\!]$.
The first equation means that $\widetilde A(t)=\alpha\cdot A(t)$ for  an $n$th root $\alpha$ of $a$. Letting $s= A(t)$, 
 the second equation holds for some $ \widetilde y\in \psi+ t^d\K[\![t]\!]$ if and only if, with this root $\alpha$,
\begin{equation}\label{eq:stab-eq}
Q(s^n)\cdot\psi\left(s\right)+P\left({s^n}\right)- \psi(\alpha s)\in s^d\K[\![s]\!]\,.
\end{equation}
Splitting $\psi$ as in (\ref{eq: split}) leads to equations (\ref{eq: stab-P}) and (\ref{eq: stab-Q}).
\end{proof}
 
\begin{proposition}\label{prop:stab-intersect}
Consider a finite collection of Puiseux arc spaces $W_k=\Pui(\psi_k,n_k,d_k)$, $k=1,\ldots,r$, ordered so that $\frac{d_1}{n_1}\ge\ldots\ge\frac{d_k}{n_k}$. 
Let $N=\lceil\frac{d_1}{n_1}\rceil$. Then in $\Aut \Arc(\bar X_0)_{T_0}^*$ one has
\begin{equation}
\bigcap_{k=1}^r \Stab_{\mathrm{fps}}\Pui(\psi_k,n_k,d_k)=\left( \prod_{i=N}^{\infty} G(i)\right)\rtimes\bar\Q^h \cong \G_a(t^N\K[\![t]\!])\rtimes\bar\Q^h\,
\end{equation}
for a subgroup $\bar\Q\subset\Q$,  where $\bar\Q^h=h\circ\bar\Q\circ h^{-1}$ with $h=\psi^\reg_1\in\bigoplus_{i=0}^{N-1}G(i)$.
\end{proposition}
\begin{proof}
In the system of equations (\ref{eq: stab-P})--(\ref{eq: stab-Q}) for all $W_k$, $k=1,\ldots,r$,  we can eliminate $P(x)$ for $k=2,\ldots,r$. This yields the system
\begin{equation}\label{eq:c-inters}
\begin{cases}
P(x)=\psi^\reg_1(ax)-\psi^\reg_1(x)Q(x)&\mod x^{N}\K[\![x]\!], \\
\psi^\reg_k(ax)-\psi^\reg_1(ax)=(\psi^\reg_k(x)-\psi^\reg_1(x))Q(x)&\mod x^{\lceil\frac{d_k}{n_k}\rceil}\K[\![x]\!],\;\; k=2,\ldots,r,\\
Q(s^n)\psi^\sing_k(s)=\psi^\sing_k(\alpha_k s)&\mod s^{d_k}\K[\![s]\!],\quad\,\,\, k=1,\ldots,r,\\
\alpha_k^{n_k}=a& \qquad\qquad\qquad\qquad\,\,\,\,\,\,  k=1,\ldots,r.
\end{cases}
\end{equation}
The first equation expresses $P(x)$ in terms of $a$ and $Q(x)$. This defines the subgroup
\begin{multline}
\left\{(x,y)\mapsto(ax,Q(x)y+\psi^\reg_1(ax)-\psi^\reg_1(x)Q(x)+\sum_{i=N}^{\infty}c_ix_i)\right\}\\
= \left(\prod_{i=N}^{\infty} G(i)\right)\rtimes\left(h\circ \Q\circ h^{-1}\right)\,,
\end{multline} where $h=\psi^\reg_1$. 
The remaining equations define a subgroup $\bar\Q\subset\Q$.
\end{proof}

\begin{notation}\label{def:fiber-aut}
Given a special fiber $\Tt'=\bar\mu^{-1}(\beta')\subset\bar X$, we let
\begin{equation}\label{eq: stab-fib}\Stab_\fps\Tt'=\bigcap_{k\,\colon\, \sigma_k\; \mathrm{is\,\,outer}} 
\Stab_{\mathrm{fps}}(\Pui(p_k))\subset\Aut \Arc(\bar X_0)_{T_0}^*\,,\end{equation}
where $p_k=\sigma_k(T_k)\in\bar X_{k-1}\setminus S$.
Taking in Proposition~\ref{prop:stab-intersect} a suitably reordered collection 
$$(W_k=\Pui(p_k)\,|\,\sigma_k\; \mathrm{is\,\,outer})\,,$$ we let $N=N(\Tt')$, $\bar\Q=\Q(\Tt')$, and $h=h(\Tt')$ denote the corresponding objects provided by this proposition, and also let  $\TTT(\Tt')=\TTT\cap\Q(\Tt')$.\end{notation}
According to Proposition~\ref{prop:stab-intersect}, with this notation we have
\begin{equation}\label{eq: decomposing} 
\Stab_\fps\Tt'= \left(\prod_{i=N}^{\infty} G(i)\right)\rtimes  \bar\Q^h \cong\G_a(t^N\K[\![t]\!])\rtimes\bar\Q^h\,.
\end{equation}

\begin{corollary}\label{cor:Tt-linear}
Given a special fiber $\Tt'=\mu^{-1}(\beta')\subset\bar X$, the following conditions are equivalent:
\begin{itemize}
\item $\Q(\Tt')=\Q$,
\item $\Q(\Tt')\supset\TTT_B$,
\item  the dual graph $\Gamma_{\Tt'}$ is a linear chain.
\end{itemize}
\end{corollary}
\begin{proof}
We start with the following observation. Clearly, an inner component $T_i$ of $\Tt'$ which is not a parent is a $(-1)$-curve. 
After contraction of $T_i$ we obtain a new special fiber, say $\Tt''$, where $\Tt''$ and $\Tt'$ are both linear or non-linear simultaneously, 
and $\Stab_\fps\Tt''=\Stab_\fps\Tt'$, hence also $\Q(\Tt'')=\Q(\Tt')$.
Thus, we may assume in the sequel that each inner component of $\Tt'$ is a parent.

Assume first that there exists  an inner component $T_i$ of $\Tt'$. By the previous observation
such a component  $T_i$ with a maximal value of $i$ is a parent of an outer component. Hence in this case $\Gamma_{\Tt'}$ is non-linear. Furthermore, 
being inner, $T_i$ belongs to the preimage of $T_k(\infty)$ for some $T_k$. Then
by (\ref{eq:Puiseux-descent-inf}) the corresponding Puiseux arc space is of form $\Pui(\psi,n,d)$ with $n>1$.  
So, the corresponding equations (\ref{eq: stab-Q}) are nontrivial. It follows that $\dim \Q(\Tt')\cap\TTT\le 1$ and $\Q(\Tt')\cap\TTT_B$ is finite. 
Thus, all three conditions of the lemma fail. 

Assume further that all components of $\Tt'$ are outer. If $\Tt'$ is linear, then $T_{m'}$ is the only non-parent, 
all the other components of $\Tt'$ being its successive parents.
By (\ref{eq:Puiseux-descent-finite}) the corresponding Puiseux arc space is of form $\Pui(\psi,1,d)$, and $\Stab_\fps\Tt'=\Stab_{\mathrm{fps}}\Pui(\psi,1,d)$.
 In this case  (\ref{eq:c-inters}) contains just one equation (of form (\ref{eq: stab-P})), and so, $\Q(\Tt')=\Q\supset\T_B$.
 Thus,  under this setup all three conditions of the lemma are fulfilled.

Finally, assume that $\Tt'$ is non-linear. 
Then there is a component $T_k\subset\Tt'$ which is a parent for at least two other components with centers, say, $T_k(q_1)$ and $T_k(q_2)$.
The corresponding  Puiseux arc spaces are of form $\Pui(\psi+q_1t^{d-1},1,d)$ and $\Pui(\psi+q_2t^{d-1},1,d)$ for some $q_1\neq q_2$, some $d\in\ZZ_{>0}$, 
and some polynomial $\psi\in \K[t]$ of degree $\le d-2$.
Inspecting system (\ref{eq:c-inters}), for $i=d-1$ we obtain the equalities $c_{d-1}=q_1(a^{d-1}-b)=q_2(a^{d-1}-b)$. This implies that $b=a^{d-1}$. 
Hence we can conclude that $\dim \Q(\Tt')\cap\TTT\le 1$ and $\Q(\Tt')\cap\TTT_B$ is finite. So, once again, all three conditions of the lemma fail.
\end{proof}

\begin{example}\label{ex:222-1}
Consider a sequence  of blowups
$$\bar X_3\stackrel{\sigma_3}{\longrightarrow} \bar X_2\stackrel{\sigma_2}{\longrightarrow} \bar X_1\stackrel{\sigma_1}{\longrightarrow} \bar X_0=\PP^1\times\PP^1\,,$$ 
where
\begin{align*}
\sigma_1\,\colon\,& T_1\mapsto T_0(0);\\
\sigma_2\,\colon\,& T_2\mapsto T_1(\infty)=T_0\cap T_1;\\
\sigma_3\,\colon\,& T_3\mapsto T_2(1)
\end{align*}
The stabilizer of the Puiseux arc space of the point $\sigma_3(T_3)=T_2(1)$ is contained in the stabilizer of  $\sigma_1(T_1)=T_0(0)$,
hence it coincides with $\Stab_\fps\Tt'$. Using Lemma~\ref{lem: Puiseux} we obtain
\[
\Pui(\sigma_3(T_3))=\Pui_{\bar X_2,T_2(1)}(0,1,1)=\Pui_{\bar X_1,T_1(\infty)}(t,1,2)=\Pui(t,2,2)\,,\] where by definition
\[\Pui(t,2,2) = \left\{ (x_0 ,y_0) \mapsto (A(t ) ,y(A(t ) )) \,|\, \ord_t A = 1, y \in t +t^2 K[\![t]\!]\right\}\,.
\] 
 By Lemma~\ref{lem:Q-stab-jet}, $\Stab_\mathrm{fps}(\Pui_{\bar X_0,T_0(0)}(t,2,2))$ is defined by equations $c_0=0$, $b_0^2=a$. It follows that 
$$\Aut_\fps\Tt'=\left(\bigoplus_{i=1}^{\infty} G (i)\right)\rtimes\TTT(\Tt')\cong \G_a(t \K[\![t]\!])\rtimes\TTT(\Tt')\,,$$
where $\TTT(\Tt')\subset \TTT$ is the one-parameter subgroup defined by $b^2=a$, see \ref{nota: descent}.
 The subgroup $\TTT(\Tt')\cap\TTT_B$ has order two and is generated by the involution $(x_0, y_0)\mapsto(x_0, -y_0)$.

Our blowup procedure leads to an SNC completion $(\bar X_3, D)$ of a smooth affine surface $X=\bar X_3\setminus D$. 
The $(-1)$-standard extended graph of this completion looks as follows

\medskip

\[\;D_{\rm ext}:\qquad \cou{}{0}\lin
\cou{}{-1}\lin\cou{}{-2}\lin 
\cou{}{-2} \nlin\cshiftup{-1
}{} \lin \cou{}{-2}\]
After contraction of the subchain $[[-1,-2,-2,-2]]$ we arrive at a new completion $(\PP^2, C)$ of $X$, 
where $C\subset \PP^2$ is a smooth conic. 
Thus, $X\cong\PP^2\setminus C$. The $(-1)$ component $T_3$ of multiplicity 2 in the central fiber becomes in $\PP^2$ a tangent line $L$ to $C$. 
The original $\A^1$-fibration $X\to\A^1$ extends to the pencil of conics in $\PP^2$ generated by $C$ and $2L$. 
The group $\Aut(\PP^2\setminus C)$ is well known; see, e.g., \cite[\S 2]{DG2} and also \ref{sit: Aut P2-conic} and \ref{rem: Z2-cover} and the references therein.
\end{example}

\section{Automorphism groups of $\A^1$-fibrations on surfaces}
\label{sec: dJ-surf} For an $\AA^1$-fibration $\mu\colon X\to B$ on a normal affine surface $X$ over a smooth affine curve $B$,  we describe in Subsection \ref{ss: final}
 the automorphism group $\Aut (X,\mu)$ up to passing to a finite index normal subgroup, see Theorems~\ref{th:jonq-mu} and \ref{th:jonq-mu-general}. 
In particular, this applies to the full automorphism group of an $\ML_1$-surface. 

\subsection{Preliminaries} \begin{sit}\label{sit: nots} We keep the notation of Section \ref{sec:fiber-arc}. 
In particular, we consider the induced $\PP^1$-fibration $\bar\mu\colon \bar X\to \bar B$ on a minimal resolved completion $\bar X$ of $X$, and a sequence of blowdowns 
\begin{equation}\label{eq: decomposition-2} \Phi\colon \bar X=
\bar X_m\xlongrightarrow{\sigma_{m}} \bar X_{m-1}\xlongrightarrow{\sigma_{{m-1}}} 
\ldots\xlongrightarrow{\sigma_{1}} \bar X_0\end{equation}
of $(-1)$-components of the special fibers $\Tt_i=\bar\mu^{-1}(\beta_i)$ of $\bar\mu$ with $\beta_1,\ldots,\beta_{n_s}\in B$, 
which terminates by a smooth ruling $\bar\mu_0\colon\bar X_0\to\bar B$. We assume as before that 
$\sigma_{i}$ contracts the component $T_i\subset \bar X_i\setminus S$ to a point $p_{i}\in \bar X_{i-1}$, where $S$ is the unique   horizontal component of 
$D=\bar X\setminus X$ and a section of $\bar\mu$. Let also $F=\bar\mu^{-1}(\bar B\setminus B)$. 
Since $B$ is affine, $F\neq\emptyset$, and we may suppose that $F$ is a union of  irreducible fibers of  $\bar \mu$. 
We let 
$$\Tt=\bigcup_{i=1}^{n_s}\Tt_i=\bigcup_{j=1}^m T_j\cup\bigcup_{}^{n_s}T_{0,\beta_i}\,,$$ where $T_{0,\beta_i}=\bar\mu_0^{-1}(\beta_i)\subset\bar X_0$, $i=1,\ldots,n_s$. 
Thus, $D\subset D_{\rm ext}:=S\cup F\cup \Tt$, and the dual graphs of both $D$ and $D_{\rm ext}$ are trees. We let $\Tt_i^{(j)}$ be  the image of $\Tt_i$ in $X_j$. 
\end{sit}

\begin{notation}\label{sit:omega}
 Let $\Aut_B (X,\mu)\subset \Aut(X,\mu)$ be the subgroup of all automorphisms of $X$
that send each fiber of $\mu$ into itself, and $\Aut_\mu B\subset\Aut B$ 
be the subgroup of all automorphisms of $B$ induced by the elements of 
$\Aut(X,\mu)$.\end{notation}

 The following fact is immediate.

\begin{lemma}\label{lem:ex-sec} There is an exact sequence
\begin{equation}\label{eq: exact-seq} 1\to\Aut_B (X,\mu)\to\Aut(X,\mu)\to
\Aut_\mu B\to 1.\end{equation}\end{lemma}

\subsection{Stabilizers of arc spaces}
\label{ss: arcs-aut}

We need the following fact. 

\begin{lemma}\label{lem: extension} Any automorphism $\alpha\in\Aut (X,\mu)$ lifts to the minimal resolution of singularities of $X$ 
and extends to an automorphism of $\bar X\setminus F$. 
\end{lemma}

\begin{proof}
The first statement is  well known  (cf.\ e.g., \cite[Lem.\ 2.2]{FKZ-completions}) and follows, for instance, 
from the uniqueness of the minimal resolution of singularities of surfaces.
Thus, any automorphism  $g\in\Aut (X,\mu)$ induces a birational automorphism of $\bar X$ regular in $\bar X\setminus D$. 
Since $g$ preserves the $\A^1$-fibration $\mu$, it extends regularly to the section $S$. 
Furthermore, $g$ induces a birational transformation $g_*$ of the dual graphs $\Gamma_D$ and $\Gamma_{D_{\rm ext}}$ fixing the vertex $S$, 
which transforms the dual graphs  $\Gamma_F$ and $\Gamma_{\Tt}$ into themselves. 
By our convention in \ref{nota: decomposition} and \ref{sit: nots} $\Gamma_D$ is minimal. So, 
the section $S$ and the components of $F$ are the only possible zero vertices of $\Gamma_D$. 
All the maximal linear chains in $\Gamma_{D\ominus (S+F)}$ are admissible, that is, with all weights $\le-2$. According to Theorem 3.1
in \cite{FKZ-graphs}, $g_*$ can be decomposed into a sequence of elementary transformations in zero vertices in $F$ followed by an automorphism, 
see Definition \ref{def:elem-tr}. Indeed, since $S$ is fixed by $g_*$, also these elementary transformations and the automorphism fix $S$. 
Hence the elementary transformations in the decomposition of $g_*$ are performed only near components of $F$. Now the second statement follows. 
\end{proof}

\begin{notation}\label{nota: fat-point}
In the notation of \ref{nota: decomposition}, we
let $$\Aut^\bullet (\bar X,F)\subset\Aut(X,\mu)$$  stand for the subgroup of 
all automorphisms of $X$ preserving $\mu$ and admitting an extension to automorphisms 
of $\bar X\setminus F$, which send each component of $\Tt$ into itself.
 Similarly, given $i\in\{0,\ldots,m\}$, we let $\Aut^\bullet(\bar X_i,F)$ be the group of 
all birational automorphisms of $\bar X_i$ which preserve $\bar\mu_i$, send the section $S$ and each component 
$T_{0,\beta_1},\ldots, T_{0,\beta_{n_s}},T_1,\ldots,T_i$ of $\Tt^{(i)}$ into itself, 
and induce automorphisms of $\bar X_i\setminus F$ (see \ref{nota: decomposition-1} and \ref{sit: nots}). 
Thus, $\Aut^\bullet(\bar X,F)=\Aut^\bullet(\bar X_m,F)$.
\end{notation}

\begin{lemma}\label{prop:aut-bullet-finite}
$\Aut^\bullet (\bar X, F)\subset \Aut(X,\mu)$  is a normal subgroup  of finite index.
\end{lemma}
\begin{proof} Let $\mathbb S(n)$ stand for the symmetric group on $n$ symbols.
By Lemma \ref{lem: extension} there is a natural embedding $\Aut (X,\mu)\hookrightarrow \Aut(\bar X\setminus F,\bar\mu)$.  
Clearly, any $\alpha\in \Aut(\bar X\setminus F,\bar\mu)$ 
permutes the special fibers $\Tt_j$ of $\bar\mu$ and the components of $\Tt$. 
Hence $\alpha$ defines a permutation $\rho(\alpha)\in \mathbb{S}(m+n_s)$, where $\rho\colon\Aut(X,\mu)\to\mathbb{S}(m+n_s)$ is
a homomorphism with $\ker\rho=\Aut^\bullet(\bar X, F)$. Now the lemma follows. 
\end{proof}

\begin{notation}[\emph{stabilizers of arc spaces}]\label{nota: stab} The group $\Aut^\bullet(\bar X_i,F)$ acts naturally on the arc space $\Arc(\bar X_i\setminus F)$.
Given a subset $W\subset \Arc(\bar X_i\setminus F)$, we let $\Stab_i(W)$ be the  stabilizer of $W$ in $\Aut^\bullet(\bar X_i,F)$.
\end{notation}

In the next proposition we
identify the groups $\Aut^\bullet(\bar X_i,F)$ and $\Stab_i(W)$ with their images in $\Aut^\bullet(\bar X_0,F)$.

\begin{proposition}\label{prop:aut-bullet-stab} 
There is a  natural embedding $\Aut^\bullet(\bar X,F)\embed \Aut^\bullet(\bar X_0,F)$ such that 
\begin{equation}\label{eq: inters-stabilizers}
 \Aut^\bullet(\bar X,F)=\bigcap_{i\,\colon\; \sigma_i{\rm\; is\,\,outer}} \Stab_0(\Pui(p_i))
\subset\Aut^\bullet(\bar X_0,F)\,.
\end{equation}
\end{proposition}
\begin{proof}
We proceed by induction on $i$. Assume that our assertion holds for $\bar X_{i-1}$.
Since  $\Aut^\bullet (\bar X_i,F)$ stabilizes $T_i$, there is a natural homomorphism $\Aut^\bullet (\bar X_i,F)\to \Aut^\bullet (\bar X_{i-1},F)$, 
which embeds $\Aut^\bullet (\bar X_i,F)$  onto the stabilizer of the point $p_i=\sigma_i(T_i)$ in $\Aut^\bullet (\bar X_{i-1},F)$. 
The latter stabilizer coincides with $\Stab_{i-1} (\Arc(\bar X_{i-1})_{p_i})$.

If $T_i$ is inner, then $p_i$ is already stabilized by $\Aut^\bullet (\bar X_{i-1},F)$, thus we have $\Aut^\bullet (\bar X_i,F)\cong \Aut^\bullet (\bar X_{i-1},F)$. 
Assume now that $T_i$ is outer. Then \[\Stab_{i-1} (\Arc(\bar X_{i-1})_{p_i})=\Stab_{i-1}(\Pui(p_i))\,,\]
where by abuse of notation we write $\Pui(p_i)$ for $\Pui_{\bar X_{i-1},p_i}(0,1,1)$ (cf.\ \ref{def:pui-p}).  
Indeed, $\Pui(p_i)$ is the subset of arcs of minimal multiplicity (equal to $\mult(T_i)$) in $\Arc(\bar X_{i-1})_{p_i}$. 
This subset is stable under the action on $\Arc(\bar X_{i-1})_{p_i}$ of the stabilizer of $p_i$ in $\Aut^\bullet (\bar X_{i-1},F)$.
 This gives the inclusion $$\Stab_{i-1} (\Arc(\bar X_{i-1})_{p_i})\subset\Stab_{i-1}(\Pui(p_i))\,.$$ 
The inverse inclusion is also clear, since the elements of $\Stab_{i-1}(\Pui(p_i))$ fix the point $p_i$.
Passing to the images of our subgroups under their natural embeddings in $\Aut^\bullet (\bar X_0,F)$, 
which we denote by the same symbols, we obtain the equalities
\[\Aut^\bullet (\bar X_i,F)=\Stab_{i-1}(\Pui(p_i))=\Aut^\bullet (\bar X_{i-1},F)\cap\Stab_0(\Pui(p_i)). \]
This yields (\ref{eq: inters-stabilizers}) for $\Aut^\bullet (\bar X_{i},F)$, since  by the inductive conjecture, it holds for the group $\Aut^\bullet (\bar X_{i-1},F)$.
\end{proof}

\subsection{Automorphism groups of $\AA^1$-fibrations}\label{ss: final}
\begin{notation}\label{nota: omega} We fix an $(\Aut_\mu B)$-stable Zariski
open subset  $\omega\subset B\setminus\{\beta_1,\ldots,\beta_{n_s}\}$, where as before $\beta_1,\ldots,\beta_{n_s}\in B$ are the points 
that correspond to the special fibers of $\mu\colon X\to B$, see \ref{sit: nots},  such that 
$\mu$ admits a trivialization over $\omega$. We  assume that $\omega$ is maximal with these properties. 
We let $X_\omega=\mu^{-1}(\omega)\cong \omega\times \AA^{1}$; this is an 
$(\Aut(X,\mu))$-stable dense open subset in $X$. 
\end{notation}
\begin{remark} If the curve $B$ is rational, then $\omega=B\setminus\{\beta_1,\ldots,\beta_{n_s}\}$
and $X_\omega$ is the complement in $X$ to the union of special fibers. 
Indeed, $\mu|_{X_\omega}=\mu_0|_{(X_0)_\omega}$ is the projection of a line bundle (see Lemma \ref{lem:general-bundle}), which is trivial in this case. 
As follows from Lemma \ref{lem: extension}, the open set $B\setminus\{\beta_1,\ldots,\beta_{n_s}\}$ is $(\Aut_\mu B)$-stable.\end{remark}

In what follows we treat separately
the cases $\omega\cong \AA^1_*,\AA^1$ and $\omega\ncong\AA^1_*,\AA^1$, see
 Theorems~\ref{th:jonq-mu} and \ref{th:jonq-mu-general}, respectively. 
In the second case, the base curve $B$ is not supposed to be rational. 
\subsubsection{Case $B=\A^1$, $\omega\cong\AA^1_*$} 
If  $\omega=B\cong\AA^1$ or $\omega=B\cong\AA^1_*$, then $X=X_\omega\cong\A^2$ and $X=X_\omega\cong\AA^1_*\times\AA^1$, 
respectively, and the group $\Aut (X,\mu)$ is the usual de Jonqui\`eres group and its analog as in \ref{sit: A2} and \ref{sit: A1-A1-star}, respectively. 
Hence we assume in the sequel that $B=\A^1$ and $\omega=\A^1_*$. This is the case, for instance, for any 
Gizatullin surface different from the plane $\A^2$.  For the $\ML_1$-surfaces, this case was studied in \cite{Be}; 
our Theorem \ref{th:jonq-mu} precises Corollary 2.3 in \cite{Be}. 

\begin{sit}\label{sit: standard-projection} We let $F=\bar\mu^{-1}(\infty)$ and $\Tt=\Tt_1=\bar\mu^{-1}(0)$.
Performing suitable elementary transformations with centers on $F$ (see Definition~ \ref{def:elem-tr}) 
one can achieve the equality $S^2=0$.
Then  the linear pencil $|S|$ on $\bar X$ defines a second $\PP^1$-fibration $\bar\nu\colon\bar X\to\PP^1$ such that $S=\bar\nu^{-1}(\infty)$. The birational morphism
\[\Phi\colon \bar X\to\bar X_0=\PP^1\times\PP^1,\; w\mapsto(\bar\mu(w),\bar\nu(w))\,,\]
is  biregular on $\bar X\setminus\bar\mu^{-1}(0)$. (In the notation of \ref{sit: standard-fibration}, $\bar\mu=\Phi_0$ and $\bar\nu=\Phi_1$, while $F=C_0$ and $S=C_1$.)
The affine coordinate on $\bar B\cong\PP^1$ is chosen so that $\bar B\setminus B=\{\infty\}$, $B\setminus\omega=\{0\}$, and so, 
$F=\{\infty\}\times\PP^1\subset \PP^{1}\times \PP^{1}$.
\end{sit}

\begin{remark} \label{sit: Jonq0}
On $\bar X_0=\PP^{1}\times \PP^{1}$ we have (cf. \ref{sit: Jonq0-local})
\begin{align*}
 \Aut^\bullet(\bar X_0,F)=&\;\left\{(x_0,y_0)\mapsto(ax_0,by_0+P(x_0))\;|\; (a,b)\in (\K^\times)^2, P=\sum_{i=0}^{N} c_ix_0^i\in\K[x_0]\right\}\\
 \cong&\;\left(\bigoplus_{i=0}^{\infty} G(i)\right) \rtimes\TTT\cong\G_a(\K[\![t]\!]) \rtimes(\mathbb{G}_m)^2 \,,
\end{align*}
\noindent where
\[\TTT=\left\{(x_0,y_0)\mapsto(ax_0,by_0)\;|\; (a,b)\in(\K^\times)^2\right\}\cong (\mathbb{G}_m)^2\,\]
and
\[G(i)=\left\{(x_0,y_0)\mapsto(x_0,y_0+c_i x_0^i)\;|\; c_i\in\K\right\}\cong\G_{a}\,.\]
The natural embedding $\Aut^\bullet(\bar X_0,F)\embed\Aut \Arc(\bar X_0)_{T_0}^*$ 
is tautological in coordinates $(a,b,c_0,\ldots)$ under substitution $(x_0,y_0)=(x,y)$
and corresponds to the embedding of the polynomial ring into the ring of formal power series.
So, all results of the previous section hold  automatically for $\Aut^\bullet(\bar X,F)$.
In particular, $\Stab_0 W=\Stab_{\mathrm{fps}}W\cap\Aut^\bullet(\bar X_0,F)$ for any subset $W\subset \Arc(\bar X_0)_{T_0}^*$, 
and $\Aut^\bullet(\bar X,F)=\Stab_\fps\Tt\cap\Aut^\bullet(\bar X_0,F).$
 \end{remark}

\begin{theorem}\label{th:jonq-mu} Let $X$ be a  normal affine surface, and let $\mu\colon X\to\AA^1$ be an $\AA^1$-fibration with a unique special fiber $\mu^{-1}(0)$. 
Then  the automorphism group $\Aut(X,\mu)$ is a finite extension of
\begin{equation}\label{eq:jonq-mu}
\Aut^\bullet(\bar X,F)= \left(\bigoplus_{k=d}^{\infty} G(k)\right)\rtimes\Lambda_\mu\cong \G_a(t^d\K[\![t]\!])\rtimes\Lambda_\mu\,,
\end{equation}
where $d\in\ZZ_{>0}$ and $\Lambda_\mu$ is conjugate to a closed subgroup of the standard torus 
$\TTT$ 
by an element of $\bigoplus_{k=0}^{d-1} G(k)$.
Furthermore, 
if $\dim \Lambda_\mu=1$, then the action of $\Lambda_\mu$ on $X$ is transversal, i.e., the intersection of a $\Lambda_\mu$-orbit and a $\mu$-fiber is always finite.
\end{theorem}
\begin{proof}
As mentioned in Remark~\ref{sit: Jonq0}, 
for any $W\subset \Arc(\bar X_0)_{T_0}$ we have $$\Stab_0 W=\Stab_\fps W\cap\Aut^\bullet(\bar X_0,F)\,.$$ Hence by 
Proposition~\ref{prop:aut-bullet-stab}, $$\Aut^\bullet(\bar X,F)=\Stab_\fps\Tt\cap\Aut^\bullet(\bar X_0,F)\,.$$
Now Proposition~\ref{prop:stab-intersect} implies (\ref{eq:jonq-mu}).
Finally, $\Aut(X,\mu)$ is a finite extension of $\Aut^\bullet(\bar X,F)$ by Lemma~\ref{prop:aut-bullet-finite}.
\end{proof}

\subsubsection{Case $\omega\ncong \AA_*^1,\AA^1$} \begin{sit}  As before, the notation $\Aut_{B}$ with the base curve $B$ as a subscript means 
passing to the subgroup of automorphisms of $(X,\mu)$ which induces the identity on $B$. In particular,
 $$\Aut^\bullet_{B} (\bar X,F)=\Aut^\bullet(\bar X,F)\cap\Aut_B  (\bar X\setminus F)\,.$$ \end{sit}
\begin{lemma} \label{lem: norm-sbgrp}
If $\omega\ncong \AA_*^1,\AA^1$ then $\Aut^\bullet_B(\bar X,F)$ is a  normal subgroup of $\Aut(X,\mu)$ of finite index. 
\end{lemma}
\begin{proof} By Lemma \ref{prop:aut-bullet-finite}, $\Aut^\bullet (\bar X, F)\lhd \Aut(X,\mu)$  is a normal subgroup  of finite index. 
Due to (\ref{eq: exact-seq}), $\Aut_B(X,\mu)\lhd \Aut(X,\mu)$. Hence $$\Aut^\bullet_B (\bar X,F)=\Aut^\bullet (\bar X,F)\cap\Aut_B(X,\mu)\lhd \Aut(X,\mu)\,.$$  
Likewise in (\ref{eq: exact-seq}) we have an exact sequence
\begin{equation}\label{eq: exact-seq1} 1\to\Aut^\bullet_B (\bar X,F)\to\Aut^\bullet (\bar X,F)\to
\Aut_\mu B\subset \Aut B\,.\end{equation}
Since $\omega$ is  $\Aut_\mu B$-stable, there is an inclusion
$\Aut_\mu B\subset\Aut\omega$, where $\omega\ncong \AA_*^1,\AA^1$ is an affine curve  of non-exceptional type. Hence $\Aut\omega$ is a finite group. Now the assertion follows. 
\end{proof}
\begin{sit}\label{sit: affine-grp}
Using the equivariant local trivialization $X_\omega\cong\omega\times\AA^1$ (see \ref{nota: omega}), for $(x_0,y_0)\in\omega\times\AA^1$ we can write
\begin{equation}\label{eq: coord-presen}
\Aut_\omega X_\omega=\big\{(x_0,y_0)\mapsto (x_0, Qy_0+P)\mid Q\in\O^{\times}_\omega(\omega), P\in\O_\omega^+(\omega)\big\}\,,\end{equation}
where $A^{\times}$ stands as before for the multiplicative group  of an algebra $A$ and 
$A^+$ for its additive group. Note that $\O_\omega^+(\omega)$ is an infinite-dimensional Abelian unipotent group. In fact, $\O_\omega^+(\omega)=\UU_\mu(X_\omega)\subset
\Aut_\omega X_\omega$ is the subgroup of automorphisms  that act on the $\mu$-fibers by translations, and  
$\O^{\times}_\omega(\omega)\subset\Aut_\omega X_\omega$ is the subgroup of
automorphisms that fix the zero section of the (trivial) line bundle $\mu|_{X_\omega}\colon X_\omega\to\omega$. 

The functions $Q\in\O^{\times}_\omega(\omega)$ have their zeros and poles in $\bar B\setminus \omega$.
So, there is a homomorphism $\O_\omega^{\times}(\omega)\to \ZZ^{N}$ with kernel $\K^\times\subset \O^{\times}_\omega(\omega)$ consisting of the nonzero constants, 
where $N$ is the number of punctures of $\omega$.
 Thus,  $\O^{\times}_\omega(\omega)\cong \G_m\times\ZZ^{r}$ for some $r< N$, and
\[\Aut_\omega X_\omega\cong \O_\omega^+(\omega)\rtimes \O_\omega^{\times}(\omega)\cong \O_\omega^+(\omega)\rtimes(\G_m\times\ZZ^{r}).\]
\end{sit}

\begin{notation}\label{nota: new-grp}
 We let $X_0=\bar X_0\setminus (S\cup F)$. By Lemma \ref{lem:general-bundle}, $\mu_0=\bar\mu_0|_{X_0}\colon X_0\to B$ has a structure of a line bundle, 
say, $L_0$ with a zero section given in $X_\omega\cong\omega\times\A^1$, $X_\omega\subset X_0$, by equation $y_0=0$.  
Let $\UU_{\mu_0}\subset\Aut_B (X_0,\mu_0)$ be the unipotent de Jonqu\`eres group of the automorphisms which
restrict to translations on general $\mu$-fibers (see Definition~\ref{def: dJ}). 
\end{notation}

\begin{proposition} \label{pr:Aut-B-X0}  For some $s\ge 0$
there are decompositions
\begin{equation}\label{eq: decomp-gamma} 
\Aut_B^\bullet(\bar X_0,F)=\Aut_B(\bar X_0,F)\cong \UU_{\mu_0}\rtimes\Lambda_B\cong\UU_{\mu_0}\rtimes (\G_m\times\ZZ^s)\,,
\end{equation} where $\Lambda_B=\Aut_B L_0\cong\cO^\times_B(B)$.
\end{proposition}
\begin{proof} 
The subgroup $\UU_{\mu_0}\subset \Aut_\mu^\bullet(\bar X_0,\bar\mu_0)$ acts freely and transitively on the space $H^0(B,\mathcal{O}_B(L_0))$ of global sections of $L_0$. 
Therefore, $\Aut_B^\bullet (\bar X_0,F)$ is
generated by $\UU_{\mu_{0}}\cong H^0(B,\mathcal{O}_B(L_0))$ and the subgroup $\Lambda_B\subset\Aut_B^\bullet(\bar X_0,F)$ of  automorphisms that fix the zero section. 
This leads to the first decomposition in (\ref{eq: decomp-gamma}).
By our assumption, the trivialization $X_\omega\cong \omega\times\AA^1$ and the  line bundle $L_0$ share the same zero section. 
Hence there is an embedding $\Lambda_B\hookrightarrow \O_\omega^\times(\omega)\cong\G_m\times\ZZ^r$ with $\Lambda_B\supset\G_m$, see \ref{sit: affine-grp}. 
Now the second isomorphism in (\ref{eq: decomp-gamma}) follows.

To show the last assertion, it suffices to observe that
the group $\Aut_B^\bullet(\bar X_0,F)=\UU_{\mu_0}\rtimes\Lambda_B$ acts on $H^0(B, \cO_B(L_0))$ via
\[g=(P,Q)\colon \Psi\mapsto Q\Psi+P\quad\forall\Psi\in H^0(B, \cO_B(L_0))\,,\]
where $P\in \UU_{\mu_0}=H^0(B, \cO_B(L_0))$ and $Q\in\Lambda_B=\Aut_B L_0=\cO_B^\times(B)$.
\end{proof}

\begin{remark}\label{rem: 1-torus}  
The neutral component $\G_m$ of the group $\cO_B^\times(B)$ acts on $L_0$ by homotheties; 
this defines a $1$-torus $\TTT_B\subset \Aut_B L_0$ (cf.\ \ref{sit: Jonq0-local}).
\end{remark}

\begin{sit} \label{sit: nota-j}
For every $j=1,\ldots,n_s$ there is a natural embedding
\[\iota_j\colon\Aut^\bullet_B(\bar X_0,F)\hookrightarrow \Aut_B \Arc(\bar X_0)_{T_{0,\beta_j}}^*\,\] (see Remark \ref{sit: Jonq0}),
where $\iota_j$ embeds the factors of the decomposition 
$\Aut_B(\bar X_0,F)\cong\UU_{\mu_0}\rtimes\Lambda_B$ into the respective factors of the  decomposition 
\[\Aut_B \Arc(\bar X_0)_{T_{0,\beta_j}}^*=\left(\prod_{i=0}^\infty G(i)\right)\rtimes\Q_B\cong\G_a(\K[\![t]\!])\rtimes\G_m(\K[\![t]\!])\cong \Aff(\K[\![t]\!]) \,,\] see \ref{nota:aff-module} and \ref{sit: Jonq0-local}.
Indeed, in both cases the first factor is the unipotent radical (acting  by translations on fibers), and the second consists of the automorphisms preserving the zero section.
Thus,  $$\UU_{\mu_0}\embed\prod_{i=0}^\infty G(i)\cong\G_a(\K[\![t]\!])\quad\mbox{and}\quad  \Lambda_B\embed\Q_B\cong\G_m(\K[\![t]\!])\,.$$

For a special fiber $\Tt'=\Tt_j\subset \Tt$ we have by  (\ref{eq: decomposing}),
\[\Stab_\fps\Tt_j=\left(\prod_{i=N_j}^{\infty}G(i)\right)\rtimes\Q_j^{h_j}\,,\]
where  in the notation of~\ref{def:fiber-aut}, $h_j=h(\Tt_j)\in\bigoplus_{i=0}^{N_j-1}G(i)$ with  $N_j=N(\Tt_j)$, and $\Q_j=\Q(\Tt_j)\subset\Q$.\end{sit}

In Proposition \ref{prop: splitting} below we gather these local data for different special fibers. In the proof we use the following lemma. 

\begin{lemma}\label{lem:chinese}
Given natural numbers $N_1,\ldots,N_{n_s}\in\ZZ_{>0}$ and  a collection $(\psi_j)_{j=1,\ldots,n_s}$, 
where $\psi_j\in\O_{B,\beta_j}$ $\forall j$, there exists a section $\Psi \in H^0(B, \cO_B(L_0))$
such that $$\iota_j(\Psi)\equiv\psi_j\mod \m_{\beta_j}^{N_j}\quad\forall j=1,\ldots,n_s\,.$$ 
\end{lemma}
\begin{proof}
Indeed, consider  the coherent ideal sheaf \[\cI =\left(\prod_{j=1}^{n_s}\m_{\beta_j}^{N_j}\right)\cdot\cO_B(L_0)\subset \cO_B(L_0)\,.\] 
The local data $(\psi_j)_j$ defines a section of the skyscraper sheaf $\cO_B(L_0)/\cI$ on the affine curve $B$. 
By the Serre analog of Cartan's A and B Theorems, $H^1(B, \cI)=0$. Hence the latter
section can be interpolated by a global section $\Psi\in H^0(B, \cO_B(L_0))$.
\end{proof}

The following corollary is immediate.

\begin{corollary}\label{cor: h} There exists $h \in\UU_{\mu_0}=H^0(B, \cO_B(L_0))$ such that \begin{equation}\label{eq: cond-on-h}
\iota_j(h)\equiv h_j\mod t^{N_j}\K[\![t]\!]\quad\forall j=1,\ldots,n_s\,,\end{equation} where $N_j$ and $h_j$ are as in \ref{sit: nota-j}.\end{corollary}

\begin{notation}\label{nota: from-prop}
We let $D_0=\sum_{j=1}^{n_s}N_j[\beta_j]$ and
$$\Lambda_{\mu,B}=\bigcap_{j=1}^{n_s}\iota_j^{-1}(\iota_j(\Lambda_B)\cap\Q_j)\subset\Lambda_B\,,$$ 
$$\UU_\mu=\{P\in H^0(B, \cO_B(L_0))=\UU_{\mu_0}\,\mid\,\mathrm{div} P\ge D_0\}=H^0(B,\mathcal{O}_B(-D_0))\,.$$
\end{notation}

\begin{proposition}\label{prop: splitting} We have
\begin{equation}\label{eq:aut-decomp-general}
\Aut^\bullet_B (\bar X,F)=\UU_\mu\rtimes\Lambda_{\mu,B}^h\,,
\end{equation} where $h\in\UU_{\mu_0}$ verifies (\ref{eq: cond-on-h}).
\end{proposition}
\begin{proof}
Let $h\in\UU_{\mu_0}$  verifies (\ref{eq: cond-on-h}), see Corollary \ref{cor: h}. 
Letting $h^{-1}\circ g\circ h=(P,Q)\in\UU_{\mu_0}\rtimes\Lambda_B$ (that is, letting $g=(P+h(1-Q), Q)$), 
we let also $P_j=\iota_j(P)$ and $Q_j=\iota_j(Q)$.
We claim that the following are equivalent:
\begin{itemize}
\item[(i)] $g\in\Aut^\bullet_B (\bar X,F)$;
\item[(ii)]  $\iota_j(g)\in\Stab_\fps \Tt_j$ for each $j=1,\ldots,n_s$;
\item[(iii)]  $\iota_j(h^{-1}\circ g\circ h)\in \left(\prod_{i=N_j}^{\infty}\G_a(i)\right)\rtimes\Q_j$ $\forall j$;
\item[(iv)]   $P_j\in \prod_{i=N_j}^{\infty}\G_a(i)$ and  $Q_j\in\Q(\Tt_j)$  $\forall j$;
\item[(v)] $\mathrm{div} P\ge D_0=\sum_{j=1}^{n_s}N_j[\beta_j]$ and $Q\in \bigcap_{j=1}^{n_s}\iota_j^{-1}(\iota_j(\Lambda_B)\cap\Q_j)$;
\item[(vi)] $P\in \UU_\mu$  and $Q\in\Lambda_{\mu,B}$.	
\end{itemize}
Indeed, it follows from Propositions~\ref{prop:stab-intersect}  and \ref{prop:aut-bullet-stab} that an element $g\in\Aut^\bullet_B (\bar X_0,F)$
belongs to $\Aut^\bullet_B (\bar X,F)$ if and only if $\iota_j(g)\in\Stab_\fps\Tt_j$ for each $j$, where $\Stab_\fps\Tt_j$ is defined as in (\ref{eq: stab-fib}).
 This proves the equivalence (i)$\Leftrightarrow$(ii). 

After replacing $h$ in (iii) by $h_j$, the equivalence (ii)$\Leftrightarrow$(iii) follows from Proposition~\ref{prop:stab-intersect}. 
By Corollary \ref{cor: h} this holds even without this replacement.

The equivalence (iii)$\Leftrightarrow$(iv) is immediate, and (iv)$\Leftrightarrow$(v)$\Leftrightarrow$(vi) follow from our definitions, 
see \ref{sit: Jonq0-local} and  \ref{nota: from-prop}. This proves the equivalence (i)$\Leftrightarrow$(vi). Now the proposition  follows. 
\end{proof}

\begin{theorem}\label{th:jonq-mu-general} 
Let $\mu\colon X\to B$ be an $\A^1$-fibration on a normal affine surface $X$ over a smooth affine curve $B$. 
If, in the notation  of \ref{nota: omega}, $\omega\ncong \AA_*^1,\AA^1$, then
the automorphism group $\Aut(X,\mu)$ is a finite extension of
\begin{equation}\label{eq: bullet-split}
\Aut_B^\bullet(\bar X,F) \cong\UU_\mu\rtimes (\Upsilon_\mu\rtimes\ZZ^l )\quad\mbox{for some}\quad l\ge 0\,,
\end{equation}
where $\UU_\mu=H^0(B,\mathcal{O}_B(-D_0))\,$ with $D_0$ as in \ref{nota: from-prop}, and $\Upsilon_\mu=\TTT_B\cap\Lambda_{\mu,B}$, see \ref{rem: 1-torus}.
 Furthermore, either $\Upsilon_\mu$  is a  finite cyclic group, or $\Upsilon_\mu=\TTT_B\cong \G_m$ and each $\mu$-fiber $\mu^{-1}(b)$, $b\in B$, is isomorphic to $\A^1$. 
\end{theorem}

\begin{proof} By Lemma \ref{lem: norm-sbgrp}, $\Aut_B^\bullet (\bar X, F)\cong\UU_\mu\rtimes \Lambda_{\mu,B}$ is  a normal subgroup of finite index in $\Aut (X, \mu)$. 
To deduce (\ref{eq: bullet-split}) it suffices to apply Proposition \ref{prop: splitting},
 where $$\Lambda_{\mu,B}\hookrightarrow\Lambda_B=\cO_B^\times(B)
\cong\TTT_B\times\ZZ^s\cong\G_m\times\ZZ^s\,,$$
see \ref{pr:Aut-B-X0} and \ref{rem: 1-torus}. For the last assertion, see \cite[\S 3]{FZ0} and \cite[Rem.\ 3.13(iii)]{FZ-LND}.
\end{proof}

The following corollary is immediate from Corollary \ref{cor: jonq=nested} and Theorems \ref{th:jonq-mu} and \ref{th:jonq-mu-general} (cf.\ Corollary \ref{cor: neutral-comp}).

\begin{corollary}\label{cor: nested} 
For any  $\A^1$-fibration $\mu\colon X\to B$ on a normal affine surface $X$ over a smooth affine curve $B$, the group $\Aut(X, \mu)$ is 
an extension of a metabelian nested ind-group of rank $\le 2$ by at most countable group. Any two maximal tori in $\Aut(X, \mu)$ are conjugated.
\end{corollary}

\begin{sit} If $\mu\colon X\to B$ admits an effective $\G_m$-action along the fibers of $\mu$, that is, $\Upsilon_\mu\cong\G_m$, 
then $\mu$ is the projection of a parabolic $\G_m$-surface (see Definition \ref{def: DPD-construction}).
 \end{sit}

\begin{proposition}\label{pr:T-lin}
 Under the assumptions of Theorem \ref{th:jonq-mu-general},
 $\Upsilon_\mu\cong\G_m$ if and only if $X$ is a parabolic $\G_m$-surface, if and only if the
connected 
components of the dual graph  $\Gamma_{\Tt}$ are linear chains. 
In the case $\omega\cong\AA^1_*$ we have the equivalences
$$\Upsilon_\mu=\G_m\Leftrightarrow \Lambda_{\mu,B}\cong \TTT\Leftrightarrow \rk\Aut X=2\,,$$ where the latter means that $X$ is an affine toric surface.
\end{proposition}
\begin{proof}
By Proposition~\ref{pr:T-lin}, the components of $\Gamma_\Tt$ are linear if and only if $\Q=\Q_j$ for every special fiber $\Tt_j$. 
The latter is true if and only if $\Q_j\supset\TTT_B$ for each $j$, which is equivalent to $\Gamma_\Tt\supset\TTT_B$.  Now the proposition follows.
\end{proof}

\end{document}